\documentclass[a4paper,12pt,english]{article}
\usepackage[utf8x]{inputenc}
\usepackage{amsmath, amsthm, amssymb, graphicx, makeidx, mathrsfs, enumerate}
\usepackage{setspace}
\usepackage{hyperref}
\usepackage[labelformat=empty]{caption}
\usepackage[T1]{fontenc} 
\usepackage{geometry}
\usepackage{siunitx} 
\usepackage[multiple]{footmisc}
\usepackage{adjustbox}
\usepackage[T1]{fontenc} 
\usepackage{geometry}
\usepackage{siunitx}
\usepackage{array, booktabs}
\usepackage{array, booktabs}
\usepackage{multirow}
\usepackage[T1]{fontenc}
\usepackage{babel}
\usepackage{romannum}
\usepackage{authblk}
\usepackage{mathtools}

\newcommand{\Z}{\mbox{$\mathbb Z$}}	
\newcommand{\Q}{\mbox{$\mathbb Q$}}	
\newcommand{\R}{\mbox{$\mathbb R$}}     
\newcommand{\C}{\mbox{$\mathbb C$}}     

\newcommand{\F}{\mbox{$\mathbb F$}}	

\usepackage{blindtext}
\usepackage{sectsty}
\sectionfont{\centering}

\usepackage{vmargin}
\setmarginsrb     { 0.9in}  
                        { 0.6in}  
                        { 0.78in}  
                        { 0.6in}  
                        {  20pt}  
                        {0.25in}  
                        {   9pt}  
                        { 0.3in}  
\newtheorem{theorem}{Theorem}[section]
\newtheorem{lemma}[theorem]{Lemma}
\newtheorem{corollary}[theorem]{Corollary}
\theoremstyle{definition}
\newtheorem{definition}[theorem]{Definition}
\newtheorem{example}[theorem]{Example}

\theoremstyle{remark}

\begin{document}
\addcontentsline{toc}{chapter}{\bf Notation}
\addcontentsline{toc}{chapter}{\bf Asymptotic Notation for Runtime Analysis}
\pagenumbering{arabic}

\author[]{{\small { SUMANDEEP KAUR\footnote{Department of Mathematics, Panjab University Chandigarh-160014,  India. Email : sumandhunay@gmail.com}  }~~AND~ SUDESH K. KHANDUJA\footnote{Corresponding author} \footnote{Indian Institute of Science Education and Research Mohali, Sector 81, Knowledge City, SAS Nagar, Punjab - 140306, India  \textsc{\&} Department of Mathematics, Panjab University, Chandigarh - 160014, India. Email : skhanduja@iisermohali.ac.in}} \\{\small{(Dedicated to Professor Ram Prakash Bambah on his $95^{th}$ Birthday)} }}

\date{}
\renewcommand\Authands{}

\title{\large{\bf{\textsc{ Discriminant and Integral basis of sextic fields defined by $x^6+ax+b$}}}}

\maketitle
\begin{center}
{\large{\bf {\textsc{Abstract}}}}
\end{center}
\noindent Let $K=\Q(\theta)$ be an algebraic number field with $\theta$ a root of an irreducible trinomial $f(x)=x^6+ax+b$ belonging to $\mathbb{Z}[x]$. In this paper, for each prime  number $p$ we compute the highest power of $p$ dividing the discriminant  of $K$ in terms of the prime powers dividing $a,~b$ and discriminant of $f(x)$. An explicit $p$-integral basis  of $K$ is also given for each prime  $p$ and a method is  described to obtain an integral basis of $K$ from these $p$-integral bases which is illustrated with examples.

\bigskip


\bigskip

\newpage
\section{{Introduction and statements of results}}
Discriminant is a basic invariant associated with an algebraic number field.  Its computation is one of the most important problems in algebraic number theory. For an algebraic number field  $K=\mathbb {Q}(\theta)$  with $\theta$ in the ring $A_K$ of algebraic integers of $K$  having $f(x)$ as its  minimal polynomial over the field $\mathbb {Q}$ of rational numbers, the discriminant $d_K$ of $K$ and the discriminant of the polynomial $f(x)$ are related by the formula $discr(f)=[A_K:\mathbb Z[\theta]]^2d_K.$ So the computation of $d_K$ is closely connected with that of the  index of the subgroup $\mathbb {Z}[\theta]$ in $A_K$; this group index is called the index of $\theta$. The computation of discriminant as well as  integral basis for an infinite family of algebraic number fields defined by a particular type of  irreducible polynomials is in general a difficult task. In 1900, Dedekind \cite{Ded} described an integral basis of all pure\footnote{By a pure number field, we mean a field $\Q(\theta)$ where $\theta$ is a root of an irreducible polynomial of the type $x^n-a$ over $\Z $.} cubic fields. In 1910, Westlund \cite{JW} gave an integral basis for fields of the type  $\Q(\sqrt[p]{a})$ having prime degree $p$. In 1984, Funakura \cite{Fun} determined  an integral basis of all pure quartic fields. In 2015, Hameed and Nakahara \cite{HN} obtained an integral basis of those pure octic fields $\Q(\sqrt[8]{a})$ where $a$ is a squarefree integer. In 2020, Jakhar et al  gave an explicit construction of an integral basis of all those $n$-{th} degree pure fields $\Q(\sqrt[n]{a})$ which are such that for each prime $p$ dividing $n$, either $p\nmid a$ or $p$ does not divide $v_p(a)$, where $v_p(a)$ stands for the highest power of $p$ dividing $a$; clearly this condition is satisfied when either $a,~n$ are coprime or $a$ is squarefree (cf. \cite{JKS3}, \cite{JKS1}). A different approach using $p$-integral basis defined below has been followed by A. Alaca, S. Alaca and K. S. Williams in  \cite{AA}, \cite{Al}, \cite{AW1} to construct integral bases of all cubic fields and all those quartic as well as quintic fields which are generated over $\Q$ by a root of an irreducible trinomial of the type $x^n+ax+b$ belonging to $\Z[x]$ with $n=4,~5$.

For a prime $p$, let $\mathbb{Z}_{(p)}$ denote the localisation of the ring $\Z$ of integers at the prime ideal $p\Z$ and $S_{(p)}$ the integral closure of $\Z_{(p)}$ in an algebraic number field $K$. The ring $S_{(p)}=\{\frac{\alpha}{a}|~\alpha\in A_K,~a\in \Z\setminus p\Z~\}$ is a free module over the principal ideal domain $\Z_{(p)}$ having rank equal to the degree of the extension $K/\Q$. A $\Z_{(p)}$-basis of the module  $S_{(p)}$ is called a $p$-integral basis of $K$. Clearly an integral basis of an algebraic number field is its $p$-integral basis for each prime $p$. For $K=\Q(\theta)$ with $\theta$ an algebraic integer, if a prime $p$ does not divide the index of $\theta$, then using Lagrange's theorem for finite groups, it can be easily seen that $A_K \subseteq \Z_{(p)}[\theta]$ and hence \{1,$\theta,\cdots ,\theta^{n-1} $\} is a $p$-integral basis of $K$ in this situation, $n$ being the degree of the extension $K/\Q$.

In the present paper, we consider all those  sextic fields $K=\Q(\theta)$ for which $\theta$ is a root of an irreducible trinomial $x^6+ax+b$ belonging to $\Z[x]$; for each prime $p$, we computer the highest power of $p$ dividing the discriminant of $K$ besides giving an explicit $p$-integral basis of $K$. The  $p$-integral basis denoted by  $\{\alpha_0, \alpha_1, \alpha_2, \alpha_3, \alpha_4,\alpha_5\}$ constructed in the paper  is of special type given by 
\begin{equation}\label{equa 1}
\begin{rcases}
\begin{split}
\alpha_0=1,~\alpha_1=\frac{a_0+\theta}{p^{k_1}},~ \alpha_2=\frac{b_0+b_1\theta+\theta^2}{p^{k_2}},~\alpha_3=\frac{c_0+c_1\theta+c_2\theta^2+\theta^3}{p^{k_3}},\\ \alpha_4=\frac{d_0+d_1\theta+d_2\theta^2+d_3\theta^3+\theta^4}{p^{k_4}},~\alpha_5=\frac{e_0+e_1\theta+e_2\theta^2+e_3\theta^3+e_4\theta^4+\theta^5}{p^{k_5}}
\end{split}
\end{rcases}
\end{equation}
with $a_{0},~{b_i}$'$ s,~{c_i}$' $s$,$~{d_i}$'$s,~{e_i}$'$s$ in $\Z$ and $0\leq k_i\leq k_{i+1}$   for each $i$. It turns out that the highest power of $p$ dividing the index of $\theta$ equals $k_1+k_2+k_3+k_4+k_5$. We then describe a method to obtain an integral basis of $K$ from these $p$-integral bases taking into consideration all the primes $p$ dividing the index of $\theta$ and illustrate the computation of integral basis as well as of the discriminant of $K$ by examples. (see Examples \ref{p1.3}-\ref{p1.5})

In what follows, $A_K$, $d_K$ will stand respectively for the ring of algebraic integers and the discriminant of $K=\Q(\theta)$, where $\theta$ satisfies the irreducible trinomial $f(x)=x^6+ax+b$ over $\Z$. We shall denote by $ind$ $\theta$ the index of the subgroup $\Z[\theta]$ in $A_K$. For any prime $p$, we shall denote by $\Z_{(p)}$ the localisation of $\Z$ at $p\Z$ and  by $v_p$  the $p$-adic valuation of the field $\Q $ of rational numbers defined for any non-zero integer $m$ to be the highest power of the prime $p$ dividing $m$. If  a prime $p$ is such that $p^5$ divides $a$ and $p^6$ divides $b$, then $\theta/p$ is a root of the polynomial $x^6+(a/p^5)x+(b/p^6)$ belonging to $\Z[x]$. Hence on replacing $\theta$ by $\theta/s$ for a suitable number $s$, we may assume that for any prime $p$ either $v_p(a)<5$ or $v_p(b)<6$. Throughout $D$ will stand for the discriminant of $f(x)$ and $D_2$ for the integer $D/2^{v_2(D)}$. As is well-known (cf. \cite[Exercise 4.5.4]{Es-Mu})
\begin{equation}\label{eq:p1.1}
D=5^5a^6-6^6b^5.
\end{equation}

With the above notations, we prove
\begin{theorem}\label{p1.1}
Let $K = \Q(\theta)$ be an algebraic number field with $\theta$ a root of an irreducible trinomial $f(x)=x^6+ax+b$ belonging to $\Z[x]$. Let $D$, $d_K$ denote respectively the discriminants of $f(x)$ and $K$. Then a $2$-integral basis, a $3$-integral basis, a $5$-integral basis and a $p$-integral basis for  primes $p$ greater than 5 together with the values $v_p(D)$, $v_p(d_K)$ for all primes $p\geq2$ are given in Table $\Romannum{1}$, Table $\Romannum{2}$, Table $\Romannum{3}$ and Table $\Romannum{4}$ respectively.
\end{theorem}
\begin{table*}
\centering
\caption{Table I}
\begin{adjustbox}{width=\textwidth}
\begin{tabular}{|l|c|c|c|l|}
\hline
Case & Conditions & $v_2(D)$ & $v_2(d_K)$ & 2-Integral Basis\\
\hline
E1 & $v_2(a)=0$  & 0 & 0 & \{1,$\theta,\theta^2,\theta^3,\theta^4,\theta^5 $\}\\
\hline
E2 & $v_2(b)=1,v_2(a)=1$  & 6 & 6 & \{1,$\theta,\theta^2,\theta^3,\theta^4,\theta^5 $\}\\
\hline
E3 & $v_2(b)=1,v_2(a)\geq 2$  & 11 & 11 & \{1,$\theta,\theta^2,\theta^3,\theta^4,\theta^5 $\}\\
\hline
E4 & $v_2(b)\geq 2,v_2(a)=1$  & 6 & 4 & \{1,$\theta,\theta^2,\theta^3,\theta^4,\theta^5/2 $\}\\
\hline
E5 & $v_2(b)\geq 3,v_2(a)=2$  & 12 & 4 & \{1,$\theta,\theta^2,\theta^3/2,\theta^4/2,\theta^5/2^2 $\}\\
\hline
E6 & $v_2(b)=3,v_2(a)=3$  & 18 & 6 & \{1,$\theta,\theta^2/2,\theta^3/2,\theta^4/2^2,\theta^5/2^2 $\}\\
\hline
E7 & $v_2(b)=3,v_2(a)\geq 4$  & 21 & 9 & \{1,$\theta,\theta^2/2,\theta^3/2,\theta^4/2^2,\theta^5/2^2 $\}\\
\hline
E8 & $v_2(b)\geq 4,v_2(a)=3$  & 18 & 4 & \{1,$\theta,\theta^2/2,\theta^3/2,\theta^4/2^2,\theta^5/2^3 $\}\\
\hline
E9 & $v_2(b)\geq 5,v_2(a)=4$  & 24 & 4 & \{1,$\theta,\theta^2/2,\theta^3/2^2,\theta^4/2^3,\theta^5/2^4 $\}\\
\hline
E10 & $v_2(b)= 5,v_2(a)=5$  & 30 & 10 & \{1,$\theta,\theta^2/2,\theta^3/2^2,\theta^4/2^3,\theta^5/2^4 $\}\\
\hline
E11 & $v_2(b)=5,v_2(a)\geq 6$  & 31 & 11 & \{1,$\theta,\theta^2/2,\theta^3/2^2,\theta^4/2^3,\theta^5/2^4 $\}\\
\hline
E12 & $v_2(a)=1$  & 7 & 7 & \{1,$\theta,\theta^2,\theta^3,\theta^4,\theta^5 $\}\\
&$b\equiv3~(mod~4)$ & & &\\
\hline
E13 & $v_2(a)=1$  & $v_2(D)\geq 9$ & 7 & \{1,$\theta,\theta^2,\theta^3,\theta^4,$\\
&$b\equiv1~(mod~4)$ & odd& &$(-5x_0^5+x_0^4\theta+x_0^3\theta^2+x_0^2\theta^3+x_0\theta^4+\theta^5)/2^{\frac{v_2(D)-7}{2}} $\},\\
& $v_2(D) $  odd & & & $5(\frac{a}{2})x_0+3b\equiv 0~(mod~2^{\frac{ v_2(D)-7}{2}})$\\
\hline

E14 & $v_2(a)=1$  & $v_2(D)\geq 8$ & 4 & \{1,$\theta,\theta^2,\theta^3,\theta^4,$\\
&$b\equiv1~(mod~4),$ & even& &$(-5x_1^5+x_1^4\theta+x_1^3\theta^2+x_1^2\theta^3+x_1\theta^4+\theta^5)/2^{\frac{v_2(D)-4}{2}} $\},\\
& $v_2(D) $  even,  & & &   $5(\frac{a}{2})x_1-2^u+3b\equiv 0~(mod~2^{\frac{v_2(D)-4}{2}}),$\\ 
&$ D_2\equiv1~(mod~4)$& & & $u={\frac{v_2(D)-6}{2}}$\\
\hline
E15 & $v_2(a)=1$  & $v_2(D)\geq 8$ & 6 & \{1,$\theta,\theta^2,\theta^3,\theta^4,$\\
&$b\equiv1~(mod~4),$ & even& &$(-5x_2^5+x_2^4\theta+x_2^3\theta^2+x_2^2\theta^3+x_2\theta^4+\theta^5)/2^{\frac{v_2(D)-6}{2}} $\},\\
&$v_2(D) $  even,  & & &  $5(\frac{a}{2})x_2+3b\equiv 0~(mod~2^{\frac{v_2(D)-6}{2}})$ \\ 
&$ D_2\equiv3~(mod~4)$ & & &
\\

\hline
E16 & $v_2(a)\geq 2$  & 6 & 6 & \{1,$\theta,\theta^2,\theta^3,\theta^4,\theta^5 $\}\\
&$b\equiv1~(mod~4)$ & & &\\
\hline
E17 & $v_2(a)\geq 2$  & 6 & 0 & \{1,$\theta,\theta^2,(1+\theta^3)/2,(\theta+\theta^4)/2,(\theta^2+\theta^5)/2 $\}\\
&$b\equiv3~(mod~4)$ & & &\\
\hline
E18 & $v_2(b)=2,v_2(a)=2$  & 12 & 6 & \{1,$\theta,\theta^2,\theta^3/2,\theta^4/2,\theta^5/2 $\}\\
\hline
E19 & $v_2(b)=2,v_2(a)=3 $  & 16 & 6 & \{1,$\theta,\theta^2,\theta^3/2,(2\theta+\theta^4)/2^2,(2\theta^2+\theta^5)/2^2 $\}\\
&$\frac{b}{4}\equiv3~(mod~4)$ & & &\\
\hline
E20 & $v_2(b)=2,v_2(a)\geq 4 $  & 16 & 4 & \{1,$\theta,\theta^2,(2+\theta^3)/2^2,(2\theta+\theta^4)/2^2,(2\theta^2+\theta^5)/2^2 $\}\\
&$\frac{b}{4}\equiv3~(mod~4)$ & & &\\
\hline
E21 & $v_2(b)=2,v_2(a)\geq 3 $  & 16 & 8 & \{1,$\theta,\theta^2,\theta^3/2,\theta^4/2,(2\theta^2+\theta^5)/2^2 $\}\\
&$\frac{b}{4}\equiv1~(mod~4)$ & & &\\
\hline
E22 & $v_2(b)=4,v_2(a)=4 $  & 24 & 4 & \{1,$\theta,\theta^2/2,\theta^3/2^2,(4\theta+\theta^4)/2^3,$\\
&$\frac{b}{16}\equiv1~(mod~4)$ & & &$(8\theta+4\theta^2+\theta^5)/2^4 $\}\\
\hline
E23 & $v_2(b)=4,v_2(a)=4 $  & 24 & 6 & \{1,$\theta,\theta^2/2,\theta^3/2^2,(4\theta+\theta^4)/2^3,(4\theta^2+\theta^5)/2^3 \}$\\
&$\frac{b}{16}\equiv3~(mod~4)$ & & &\\
\hline
E24 & $v_2(b)=4,v_2(a)=5 $  & 26 & 6 & \{1,$\theta,\theta^2/2,\theta^3/2^2,(4\theta+\theta^4)/2^3,(4\theta^2+\theta^5)/2^4 \}$\\
&$\frac{b}{16}\equiv3~(mod~4)$ & & &\\
\hline
E25 & $v_2(b)=4,v_2(a)\geq 6 $  & 26 & 4 & \{1,$\theta,\theta^2/2,(4+\theta^3)/2^3,(4\theta+\theta^4)/2^3,(4\theta^2+\theta^5)/2^4 \}$\\
&$\frac{b}{16}\equiv3~(mod~4)$ & & &\\
\hline
E26 & $v_2(b)=4,v_2(a)\geq 5 $  & 26 & 8 & \{1,$\theta,\theta^2/2,\theta^3/2^2,(4\theta+\theta^4)/2^3,\theta^5/2^3 \}$\\
&$\frac{b}{16}\equiv1~(mod~4)$ & & &\\
\hline
\end{tabular}}
\end{adjustbox}
\end{table*}

\begin{table*}
\centering
\caption{Table II}
\begin{adjustbox}{width=\textwidth}
\begin{tabular}{|l|c|c|c|l|}
\hline
Case & Conditions & $v_3(D)$ & $v_3(d_K)$ & 3-Integral Basis\\
\hline
F1 & $v_3(a)=0$  & 0 & 0 & \{1,$\theta,\theta^2,\theta^3,\theta^4,\theta^5 $\}\\
\hline
F2 & $v_3(b)=1,v_3(a)=1$  & 6 & 6 & \{1,$\theta,\theta^2,\theta^3,\theta^4,\theta^5 $\}\\
\hline
F3 & $v_3(b)=1,v_3(a)\geq 2$  & 11 & 11 & \{1,$\theta,\theta^2,\theta^3,\theta^4,\theta^5 $\}\\
\hline
F4 & $v_3(b)\geq 2,v_3(a)=1$  & 6 & 4 & \{1,$\theta,\theta^2,\theta^3,\theta^4,\theta^5/3 $\}\\
\hline
F5 & $v_3(b)=2,v_3(a)=2$  & 12 & 6 & \{1,$\theta,\theta^2,\theta^3/3,\theta^4/3,\theta^5/3 $\}\\
\hline
F6 & $v_3(b)=2,v_3(a)\geq 3$  & 16 & 10 & \{1,$\theta,\theta^2,\theta^3/3,\theta^4/3,\theta^5/3 $\}\\
\hline
F7 & $v_3(b)\geq 3,v_3(a)=2$  & 12 & 4 & \{1,$\theta,\theta^2,\theta^3/3,\theta^4/3,\theta^5/3^2

 $\}\\
\hline
F8 & $v_3(b)\geq 4,v_3(a)=3$  & 18 & 4 & \{1,$\theta,\theta^2/3,\theta^3/3,\theta^4/3^2,\theta^5/3^3 $\}\\
\hline
F9 & $v_3(b)=4,v_3(a)=4$  & 24 & 8 & \{1,$\theta,\theta^2/3,\theta^3/3^2,\theta^4/3^2,\theta^5/3^3 $\}\\
\hline
F10 & $v_3(b)=4,v_3(a)\geq 5$  & 26 & 10& \{1,$\theta,\theta^2/3,\theta^3/3^2,\theta^4/3^2,\theta^5/3^3 $\}\\
\hline
F11 & $v_3(b)\geq 5,v_3(a)=4$  & 24 & 4 & \{1,$\theta,\theta^2/3,\theta^3/3^2,\theta^4/3^3,\theta^5/3^4 $\}\\
\hline
F12 & $v_3(b)= 5,v_3(a)=5$  & 30 & 10 & \{1,$\theta,\theta^2/3,\theta^3/3^2,\theta^4/3^3,\theta^5/3^4 $\}\\
\hline
F13 & $v_3(b)= 5,v_3(a)\geq 6$  & 31 & 11 & \{1,$\theta,\theta^2/3,\theta^3/3^2,\theta^4/3^3,\theta^5/3^4 $\}\\
\hline
F14 & $v_3(a)=1$  & 6 & 6 & \{1,$\theta,\theta^2,\theta^3,\theta^4,\theta^5 $\}\\
&$b\equiv1~(mod~3)$ & & &\\
\hline
F15 & $v_3(a)\geq 2$  & 6 & 6 & \{1,$\theta,\theta^2,\theta^3,\theta^4,\theta^5 $\}\\
&$b\equiv4~or~7~(mod~9)$ & & &\\
&&&&\\
\hline
F16 & $v_3(a)\geq 2$  & 6 & 2 & \{1,$\theta,\theta^2,\theta^3,(1-\theta^2+\theta^4)/3,$\\
&$b\equiv1~(mod~9)$ & & &$(\theta-\theta^3+\theta^5)/3 $\}\\
\hline
F17 & $v_3(a)\geq 2$  & 6 & 6 & \{1,$\theta,\theta^2,\theta^3,\theta^4,\theta^5 $\}\\
&$b\equiv2 ~or~ 5~(mod~9)$ & & &\\
&&&&\\
\hline
F18 & $v_3(a)\geq 2$  & 6 & 2 & \{1,$\theta,\theta^2,\theta^3,(1+\theta^2+\theta^4)/3,$\\
&$b\equiv-1~(mod~9)$ & & &$(\theta+\theta^3+\theta^5)/3 $\}\\
\hline
F19 & $b\equiv2~(mod~9)$  & 7 & 5 & \{1,$\theta,\theta^2,\theta^3,\theta^4,$\\
& $a\equiv\pm 3~(mod~9)$& & & $(\epsilon+\theta+\epsilon\theta^2+\theta^3+\epsilon\theta^4+\theta^5)/3 $\},\\
&&& & $\epsilon$ is $-1$ or $1$ according as $a\equiv 3$ or $-3~(mod~9)$\\
\hline

F20 & $b\equiv-1~(mod~9)$  & 7 & 7 & \{1,$\theta,\theta^2,\theta^3,\theta^4,\theta^5\}$\\
& $a\equiv\pm3~(mod~9)$& & &\\
\hline
F21 & $b\equiv5~(mod~9)$  & 8 & 6 & \{1,$\theta,\theta^2,\theta^3,\theta^4,$\\
& $a\equiv\pm 3~(mod~9)$& & & $(\epsilon+\theta+\epsilon\theta^2+\theta^3+\epsilon\theta^4+\theta^5)/3 $\},\\
&$v_3(D)$ $=$ 8 & & &  $\epsilon$ is $-1$ or $1$ according as $a\equiv -3$ or $3~(mod~9)$ \\
\hline
F22 & $b\equiv5~(mod~9)$  & 9 & 3 & \{1,$\theta,\theta^2,\theta^3,$\\
& $a\equiv\pm 3~(mod~9)$& & & $(-1-\theta+\theta^3+\theta^4)/3$\\
&$v_3(D) =$ 9& & &$(-5x_1^5+x_1^4\theta+x_1^3\theta^2+x_1^2\theta^3+x_1\theta^4+\theta^5)/3^2 $\},\\
 & & & &   $5(\frac{a}{3})x_1+2b\equiv 0~(mod~9)$\\
\hline
F23 & $b\equiv5~(mod~9)$  & $v_3(D)\geq 10$ & 4 & \{1,$\theta,\theta^2,\theta^3,$\\
& $a\equiv\pm 3~(mod~9)$& even & & $(-1-\theta+\theta^3+\theta^4)/3$\\
&$v_3(D)\geq10$ and even & & &$(-5x_2^5+x_2^4\theta+x_2^3\theta^2+x_2^2\theta^3+x_2\theta^4+\theta^5)/3^{\frac{v_3(D)-6}{2}} $\},\\
& & & &  $5(\frac{a}{3})x_2+2b\equiv 0~(mod~3^{{\frac{v_3(D)-6}{2}}})$\\
\hline
F24 & $b\equiv5~(mod~9)$  & $v_3(D)\geq 11$ & 3 & \{1,$\theta,\theta^2,\theta^3,$\\
& $a\equiv\pm 3~(mod~9)$& odd& & $(-1-\theta+\theta^3+\theta^4)/3$\\
&$v_3(D)\geq11$ and odd & & &$(-5x_3^5+x_3^4\theta+x_3^3\theta^2+x_3^2\theta^3+x_3\theta^4+\theta^5)/3^{\frac{v_3(D)-5}{2}}$\},\\
& & & &  $5(\frac{a}{3})x_3+2b\equiv 0~(mod~3^{{\frac{v_3(D)-5}{2}}})$\\
\hline

F25 & $v_3(b)=3,v_3(a)=3$  & 18 & 6 & \{1,$\theta,\theta^2/3,\theta^3/3,\theta^4/3^2,\theta^5/3 ^2 $\}\\
\hline
F26 & $v_3(b)=3,v_3(a)\geq 4$  & 21 & 7 & \{1,$\theta,\theta^2/3,\theta^3/3,\theta^4/3^2,$\\
&$B=\frac{b}{3^3}$, $v_3(B^3-B)=1$ &  & &$(\theta^2+3B)^2\theta/3 ^3 $\}\\
\hline
F27 & $v_3(b)=3,v_3(a)\geq 4$  & 21 & 3 & \{1,$\theta,\theta^2/3,(\theta^2+3B)\theta/3 ^2,$\\
&$B=\frac{b}{3^3}$, $v_3(B^3-B)\geq2$ &  & &$(\theta^2+3B)^2/3 ^3,(\theta^2+3B)^2\theta/3 ^3 $\}\\
\hline

\end{tabular}
\end{adjustbox}
\end{table*}
\begin{table*}
\centering
\caption{Table III}
\begin{adjustbox}{width=\textwidth}
\begin{tabular}{|l|c|c|c|c|}
\hline
Case & Conditions & $v_5(D)$ & $v_5(d_K)$ & 5-Integral Basis\\
\hline
G1 & $v_5(b)=0$  & 0 & 0 & \{1,$\theta,\theta^2,\theta^3,\theta^4,\theta^5 $\}\\
\hline
G2 & $v_5(b)=1,v_5(a)=0$, & & & \\
 & $a^4\not\equiv 21$ $ mod$ $ 25$, & 5 & 5 & \{1,$\theta,\theta^2,\theta^3,\theta^4,\theta^5 $\}\\
  & $b \not\equiv a^2-a^6$ $mod$ $ 25$, & & & \\
& $a^2\not\equiv( b/5)$ $mod$ $5$ & & & \\
\hline  
G3 & $v_5(b)=1,v_5(a)=0$, & & & \\
 & $a^4\not\equiv 21$ $ mod$ $ 25$, & 6 & 6 & \{1,$\theta,\theta^2,\theta^3,\theta^4,\theta^5 $\}\\
  & $b \not\equiv a^2-a^6$ $mod$ $ 25$, & & & \\
& $a^2\equiv( b/5)$ $mod$ $5$ & & & \\
\hline 
G4 & $v_5(b)=1,v_5(a)=0$, & & &\{1,$\theta,\theta^2,\theta^3,\theta^4,$ \\
 & $a^4\not\equiv 21$ $ mod$ $ 25$, & 5 & 3 & $(\theta-a^3\theta^2+a^2\theta^3-a\theta^4+\theta^5)/5 $\}\\
  & $b \equiv a^2-a^6$ $mod$ $ 25$, & & &  \\
\hline
G5 & $v_5(b)=1,v_5(a)=0$, & & & \\
 & $a^4\equiv 21$ $ mod$ $ 25$, & 5 & 5 & \{1,$\theta,\theta^2,\theta^3,\theta^4,\theta^5 $\}\\
  & $b \not\equiv a^2-a^6$ $mod$ $ 25$ & & & \\
  \hline 
G6 & $v_5(b)=1,v_5(a)=0$, & & &\{1,$\theta,\theta^2,\theta^3,$ \\
 & $a^4\equiv 21$ $ mod$ $ 25$, & $v_5(D)\geq 7$ & 3 & $(-4a^3\theta+3a^2\theta^2-2a\theta^3+\theta^4)/5, $\\
  & $b \equiv a^2-a^6$ $mod$ $ 25$, & & &$(-5x_0^5+x_0^4\theta+x_0^3\theta^2+x_0^2\theta^3+x_0\theta^4+\theta^5)/5\frac{v_5(D)-5}{2}$\},  \\
   & $v_5(D)$ is odd & & &${a}x_0+6(\frac{b}{5})\equiv 0~(mod~5^{\frac{v_5(D)-5}{2}})$ \\
   \hline
G7 & $v_5(b)=1,v_5(a)=0$, & & &\{1,$\theta,\theta^2,\theta^3,$ \\
    & $a^4\equiv 21$ $ mod$ $ 25$, & $v_5(D)\geq 8$ & 2 & $(-4a^3\theta+3a^2\theta^2-2a\theta^3+\theta^4)/5, $\\
     & $b \equiv a^2-a^6$ $mod$ $ 25$, & & &$(-5x_1^5+x_1^4\theta+x_1^3\theta^2+x_1^2\theta^3+x_1\theta^4+\theta^5)/5\frac{v_5(D)-4}{2}$\},  \\
      & $v_5(D)$ is even & & & ${a}x_1+6(\frac{b}{5})\equiv 0~(mod~5^{\frac{v_5(D)-4}{2}})$\\
      \hline
G8 & $v_5(b)=1,v_5(a)\geq 1$  & 5 & 5 & \{1,$\theta,\theta^2,\theta^3,\theta^4,\theta^5 $\}\\
\hline
G9 & $v_5(b)\geq 2,v_5(a)=0$  & 5 & 5 & \{1,$\theta,\theta^2,\theta^3,\theta^4,\theta^5 $\}\\
 & $a^4 \not\equiv 1$ $mod$ $25$ & & & \\
\hline
G10 & $v_5(b)\geq 2,v_5(a)=0$, &5 &3 &\{1,$\theta,\theta^2,\theta^3,\theta^4,$ \\
 & $a^4\equiv 1$ $ mod$ $ 25$, &  &  & $(\theta-a^3\theta^2+a^2\theta^3-a\theta^4+\theta^5)/5 $\}\\
\hline 
G11 & $v_5(b)=2,v_5(a)= 1$ & 10 & 8 &  \{1,$\theta,\theta^2,\theta^3,\theta^4,\theta^5/5 $\}\\
\hline
G12 & $v_5(b)=2,v_5(a)\geq 2$ & 10 & 4 & \{1,$\theta,\theta^2,\theta^3/5,\theta^4/5,\theta^5/5 $\}\\
\hline
G13 & $v_5(b)\geq 3,v_5(a)=1$ & 11 & 9 & \{1,$\theta,\theta^2,\theta^3,\theta^4,\theta^5/5 $\}\\
\hline
G14 &  $v_5(b)=3,v_5(a)=2$ & 15 & 7 & \{1,$\theta,\theta^2,\theta^3/5,\theta^4/5,\theta^5/5^2 $\}\\
\hline
G15 & $v_5(b)=3,v_5(a)\geq 3$ & 15 & 3 & \{1,$\theta,\theta^2/5,\theta^3/5,\theta^4/5^2,\theta^5/5^2 $\}\\
\hline
G16 & $v_5(b)\geq 4,v_5(a)=2$ & 17 & 9 & \{1,$\theta,\theta^2,\theta^3/5,\theta^4/5,\theta^5/5^2 $\}\\
\hline
G17 & $v_5(b)=4,v_5(a)=3$ & 20 & 6 & \{1,$\theta,\theta^2/5,\theta^3/5,\theta^4/5^2,\theta^5/5^3 $\}\\
\hline
G18 &  $v_5(b)=4,v_5(a)\geq 4$ & 20 & 4 & \{1,$\theta,\theta^2/5,\theta^3/5^2,\theta^4/5^2,\theta^5/5^3 $\}\\
\hline
G19 & $v_5(b)\geq5,v_5(a)=3$ & 23 & 9 & \{1,$\theta,\theta^2/5,\theta^3/5,\theta^4/5^2,\theta^5/5^3 $\}\\
\hline
G20 & $v_5(b)=5,v_5(a)=4$ & 25 & 5 & \{1,$\theta,\theta^2/5,\theta^3/5^2,\theta^4/5^3,\theta^5/5^4 $\}\\
\hline
G21 & $v_5(b)=5,v_5(a)\geq5$ & 25 & 5 & \{1,$\theta,\theta^2/5,\theta^3/5^2,\theta^4/5^3,\theta^5/5^4 $\}\\
\hline
G22 & $v_5(b)\geq 6,v_5(a)=4$ & 29 & 9 & \{1,$\theta,\theta^2/5,\theta^3/5^2,\theta^4/5^3,\theta^5/5^4 $\}\\
\hline
\end{tabular}
\end{adjustbox}
\end{table*}
\begin{table*}
\centering
\caption{Table IV}
\label{Table IV}
\begin{adjustbox}{width=\textwidth}
\begin{tabular}{|l|c|c|c|c|}
\hline
Case & Conditions & $v_p(D)$ & $v_p(d_K)$ & $p$-Integral Basis\\
\hline
H1 & $v_p(b)=0,v_p(a)\geq 1$ or & 0 & 0 & \{1,$\theta,\theta^2,\theta^3,\theta^4,\theta^5 $\}\\
   &   $v_p(a)=0,v_p(b)\geq 1$ &  &  &  \\
\hline
H2 & $v_p(b)=1,v_p(a)\geq 1$ & 5 & 5 &  \{1,$\theta,\theta^2,\theta^3,\theta^4,\theta^5 $\}\\
\hline
H3 & $v_p(a)=1,v_p(b)\geq 2$ & 6 & 4 & \{1,$\theta,\theta^2,\theta^3,\theta^4,\theta^5/p $\}\\
\hline
H4 & $v_p(b)=2,v_p(a)\geq 2$ & 10 & 4 & \{1,$\theta,\theta^2,\theta^3/p,\theta^4/p,\theta^5/p $\}\\
\hline
H5 &  $v_p(a)=2,v_p(b)\geq 3$ & 12 & 4 & \{1,$\theta,\theta^2,\theta^3/p,\theta^4/p,\theta^5/p^2 $\}\\
\hline
H6 & $v_p(b)=3,v_p(a)\geq 3$ & 15 & 3 & \{1,$\theta,\theta^2/p,\theta^3/p,\theta^4/p^2,\theta^5/p^2 $\}\\
\hline
H7 & $v_p(a)=3,v_p(b)\geq 4$ & 18 & 4 & \{1,$\theta,\theta^2/p,\theta^3/p,\theta^4/p^2,\theta^5/p^3 $\}\\
\hline
H8 & $v_p(b)=4,v_p(a)\geq 4$ & 20 & 4 & \{1,$\theta,\theta^2/p,\theta^3/p^2,\theta^4/p^2,\theta^5/p^3 $\}\\
\hline
H9 &  $v_p(a)=4,v_p(b)\geq 5$ & 24 & 4 & \{1,$\theta,\theta^2/p,\theta^3/p^2,\theta^4/p^3,\theta^5/p^4 $\}\\
\hline
H10 & $v_p(b)=5,v_p(a)\geq 5$ & 25 & 5 & \{1,$\theta,\theta^2/p,\theta^3/p^2,\theta^4/p^3,\theta^5/p^4 $\}\\
\hline
H11 & $v_p(ab)=0$,& $v_p(D) $  & 0 & \{{1,$\theta,\theta^2,\theta^3,\theta^4,$}\\
 & $v_p(D) $ is even  & & & $({x+y\theta+z\theta^2+v\theta^3+w\theta^4+\theta^5})/p^{m}\}$, \\
 &&&& $m={v_p(D)/2}$,\\
 & & & &$6x\equiv 5a~(mod~p^m)$, \\
 &&&&  $(5a)^4y\equiv(6b)^4~(mod~p^m)$,\\
 & && &$(5a)^3z\equiv-(6b)^3~(mod~p^m),$ \\
 & & &&$(5a)^2v\equiv(6b)^2~(mod~p^m)$, \\
  &&& &            $5aw\equiv-6b~(mod~p^m)$.\\
\hline
H12 &  $v_p(ab)=0$,  & $v_p(D) $  & 1 & \{1,$\theta,\theta^2,\theta^3,\theta^4$, \\
 & $v_p(D) $ is odd & & & $(x+y\theta+z\theta^2+v\theta^3+w\theta^4+\theta^5)/p^{m}\}$,\\
 &&&&$ m={(v_p(D)-1)/2}$,\\
  & & & &$6x\equiv 5a~(mod~p^m)$,\\
  &&& &$(5a)^4y\equiv(6b)^4~(mod~p^m)$ ,\\
  & && & $(5a)^3z\equiv-(6b)^3~(mod~p^m),$\\
  & & && $(5a)^2v\equiv(6b)^2~(mod~p^m)$,\\
   &&& & $5aw\equiv-6b~(mod~p^m)$           .\\
\hline

\end{tabular}
\end{adjustbox}
\end{table*}

\newpage
Note that with $K=\Q(\theta)$ as in Theorem \ref{p1.1}, the $p$-integral basis of $K$ mentioned in Tables I-IV in each case  consists of elements $\alpha_0, \alpha_1, \alpha_2, \alpha_3, \alpha_4,\alpha_5$, where $\alpha_i$'$s$ are  of the type  as in $(\ref{equa 1})$.
In general if $L=\Q(\xi)$ is an algebraic number field of degree $n$ with $\xi$ an algebraic integer, then it is well known that there exists an integral basis $\mathcal B:=\{\beta_0,\cdots ,\beta_{n-1}\}$ of $L$ such that $\beta_0=1$, $$\beta_i=\frac{a_{i0}+a_{i1}\xi+\cdots +a_{ii-1}\xi^{i-1}+\xi^i}{d_i}$$ with $a_{ij},~d_i\in \Z$ and the positive integer $d_i$ dividing $d_{i+1}$ for $1\leq i\leq n-1$; morever the numbers $d_i$ are uniquely determined by $\xi$ and the index  $[A_L : \Z[\xi]]$ to be denoted by $ind~\xi$ equals $\displaystyle\prod_{i=1}^{n-1}d_i$  (cf. \cite[Chapter 2, Theorem 13]{Mar}). Fix a prime $p$ and let $l_i$ stand for $v_p(d_i)$, then $v_p(ind~\xi)=l_1+l_2+\cdots +l_{n-1}.$ Since $\mathcal B$ is a $p$-integral basis of $L$, so is $\mathcal B_p^*:=\{1,\frac{\beta_1d_1}{p^{l_1}},\cdots ,\frac{\beta_{n-1}d_{n-1}}{p^{l_{n-1}}}\}$. It can be easily seen that if $\mathcal C_p=\{1,\gamma_1,\cdots ,\gamma_{n-1}\}$ is another $p$-integral basis of $L$ where $\gamma_i$'$s$ are of the type $$\gamma_i=\frac{c_{i0}+c_{i1}\xi^i+\cdots +c_{ii-1}\xi^{i-1}+\xi^i}{p^{k_i}} $$ with $c_{ij},~k_i$ in $ \Z$ for $1\leq i\leq n-1$, then on writing each member of $\mathcal B_p^*$ as a $\Z_{(p)}$-linear combination of members of $\mathcal C_p$ and vice versa, we see that $l_i=k_i~ \forall~ i$ and hence 
\begin{equation}\label{eqn p2}
v_p(ind~\xi)=l_1+l_2+\cdots +l_{n-1}=k_1+k_2+\cdots +k_{n-1}.
\end{equation}In fact  a $p$-integral basis of an algebraic number field $L=\Q(\xi)$ of the type given by $\mathcal C_p$ is easier to construct than constructing an integral basis of $L$ and these $p$-integral bases of $L$ with $p$ running over all primes dividing $ind~\xi$, quickly lead to the construction of an integral basis of $L$ as asserted by the following theorem. 

\begin{theorem}\label{p1.2}
Let $L=\Q(\xi)$ be an algebraic number field of degree $n$ with $\xi$ an algebraic integer. Assume that $A_L\neq \Z[\xi]$. Let $p_1,\cdots, ~p_s $ be all the distinct primes dividing  $[A_L : \Z[\xi]]$. Let $\mathcal{B}_r=\{\alpha_{r0},\alpha_{r1},\cdots,\alpha_{r(n-1)}\}$ be a $p_r$-integral basis of $L$, $1 \leq r\leq s $ with $\alpha_{r0}=1$, $\alpha_{ri}=\frac{c_{i0}^{(r)}+c_{i1}^{(r)}\xi+\cdots+c_{ii-1}^{(r)}\xi^{i-1}+\xi^i}{p^{k_{i,r}}}$, $1\leq i \leq n-1$, where $c_{ij}^{(r)}$ and $ 0\leq k_{i,r}\leq k_{i+1,r}$ are integers. If  $c_{ij} \in \Z$ are such that $c_{ij} \equiv c_{ij}^{(r)}(mod$ $p_r^{k_{i,r}})$ for $1 \leq r\leq s $ and if $t_i$ stands for $\displaystyle \prod_{r=1}^{s} p_r^{k_{i,r}}$, then $\{\alpha_0,\alpha_1,\cdots,\alpha_{n-1}\}$ is an integral basis of $L$ where $\alpha_0=1$, $\alpha_i=\frac{c_{i0}+c_{i1}\xi+\cdots+c_{ii-1}\xi^{i-1}+\xi^i}{t_{i}}$ for $1\leq i\leq n-1$.\\
\end{theorem}

It may be pointed out that our method of proving Theorem \ref{p1.1} is completely different from the one used in \cite{AA},  \cite{Al}, \cite{AW1} for dealing with the analogous problem in cubic, quartic and quintic fields. In fact, we construct theoretical proofs without cumbersome calculations and do not use computer programming in the proof.

The result of the following corollary which is already known  in a slightly  more complicated form (cf. \cite{JKS2}) is an immediate consequence of Theorem \ref{p1.1} .
\begin{corollary}\label{p1.14}
Let $K=\Q(\theta)$ be an algebraic number field with $\theta$ satisfying an irreducible trinomial $x^6+b$ such that the integer $b$ is not divisible by the $6$-th power of any prime. For a prime $p$, let   $s_p$ stand for $6-$gcd $(6,v_p(b))$ and $b_q$ for  $\frac{b}{q^{v_{q}(b)}}$ when $q=2,~3$.
  Then the discriminant  of $K$ is  $sgn(-b)2^{r_1}3^{r_2}\displaystyle\prod_{p\neq 2,3}^{}p^{s_p},$  where
$r_1,~r_2$ are given by 
\[r_1= 
\begin{cases}
0 & \text{ if } b+1\equiv 0~(mod~4)\\
6 & \text{ if } b+1\equiv 2~(mod~4)\\
11 & \text{if} ~v_2(b)~ \text{is }~ 1~ \text{or} ~5\\
9 & \text{if} ~v_2(b)=3\\
4 &\text{if} ~v_2(b)~\text{is}~2 ~\text{or}~ 4~ \text{and}~ b_2\equiv 3~(mod~4)\\
8 &\text{if} ~v_2(b)~\text{is}~2~ \text{or}~ 4~ \text{and}~ b_2\equiv 1~(mod~4)\\
\end{cases}
\] 
\[r_2= 
\begin{cases}
2 & \text{ if } b\equiv 1~\text{or} ~-1~(mod~9)\\
6 & \text{ if } b\equiv \pm 2~\text{or}~\pm 4~(mod~9)\\
11 & \text{if} ~v_3(b)~\text{is} ~1~ \text{or}~ 5\\
10 & \text{if} ~v_3(b)~\text{is} ~2~ \text{or}~ 4\\
3 &\text{if} ~v_3(b)=3 ~ \text{and}~ (b_3)^2\equiv 1~(mod~9)\\
7 &\text{if} ~v_3(b)=3 ~ \text{and}~ (b_3)^2\not\equiv 1~(mod~9).\\
\end{cases}
\]

\end{corollary}

\hspace{6mm}

The following examples illustrate Theorems \ref{p1.1}, \ref{p1.2}.
\begin{example}\label{p1.3}
We compute the discriminant and an integral basis of $K=\Q(\theta)$ where $\theta$ is a root of the polynomial $f(x)=x^6+12$. Here $D=discr(f(x))=-2^{16}3^{11}$. By case E20 of Table I,  \{1,$\theta,\theta^2,(2+\theta^3)/2^2,(2\theta+\theta^4)/2^2,(2\theta^2+\theta^5)/2^2 $\} is a $2$-integral basis of $K$. Further by case F3 of Table II, \{1,$\theta,\theta^2,\theta^3,\theta^4,\theta^5 $\} is a $3$-integral basis of $K$. When $p\geq 5,$ then $p \nmid D$ and so \{1,$\theta,\theta^2,\theta^3,\theta^4,\theta^5 $\} is $p$-integral basis of $K$. Therefore by Theorem \ref{p1.2},   \{1,$\theta,\theta^2,(2+\theta^3)/2^2,(2\theta+\theta^4)/2^2,(2\theta^2+\theta^5)/2^2 $\}  is an integral basis of $K$ and $d_K=\frac{D}{(ind~\theta)^2}=-2^43^{11}.$
\end{example}
\begin{example}\label{p1.4}
We determine the discriminant and an integral basis  of $K=\Q(\theta)$ where $\theta$ is a root of the polynomial $f(x)=x^6+135$. Here $D=discr(f(x))=-2^63^{21}5^5.$ By case E17 of Table I, \{1,$\theta,\theta^2,(1+\theta^3)/2,(\theta+\theta^4)/2,(\theta^2+\theta^5)/2 $\} is a $2$-integral basis of $K$. Further by case F26 of Table II, \{$1,\theta,\theta^2/3,\theta^3/3,\theta^4/3^2,(9\theta+3\theta^3+\theta^5)/3^3$\} is a $3$-integral basis of $K$. By case G8 of Table III, \{1,$\theta,\theta^2,\theta^3,\theta^4,\theta^5 $\}  is $5$-integral basis of $K$. When  $p\geq 7,$ then $p \nmid D$ and so \{1,$\theta,\theta^2,\theta^3,\theta^4,\theta^5 $\} is $p$-integral basis of $K$.  Therefore by Theorem \ref{p1.2} , \{$1,\theta,\theta^2/3,(3+\theta^3)/6, (27\theta+\theta^4)/18,(36\theta+27\theta^2+30\theta^3+\theta^5)/54$\} is an integral basis of $K$ and $d_K=\frac{D}{(ind~\theta)^2}=-3^75^5$.
\end{example}
\begin{example}\label{p1.5}
We find the discriminant and an integral basis of $K=\Q(\theta)$ where $\theta$ is a root of the polynomial $f(x)=x^6+4x+4$. In view of a simple result\footnote{Let $f(x)= a_nx^n+\cdots +a_0$ be a polynomial with coefficients in $\mathbb Z$. Suppose that there exists a prime $p$ such that $v_p(a_n)=0,v_p(a_i)/(n-i)\geq v_p(a_0)/n$ for $0 \leq i \leq n-1$. If $d$ is the gcd of $(v_p(a_0),n)$, then  each irreducible factor of $f(x)$ over $\mathbb Q$ has degree at least $\frac{n}{d}$.} proved in \cite{Ja}, each irreducible factor of $f(x)$ over $\Q$ has degree at least $3$ from which it can be easily deduced that $f(x)$ is irreducible over $\Q$. Here $D=discr(f(x))=-2^{12}\times8539.$ By case E18 of Table I, $\{1,\theta,\theta^2,\theta^3/2,\theta^4/2,\theta^5/2 \}$ is a 2-integral basis of $K$. Since 8539 is a prime, it does not divide $ind~\theta$. So \{1,$\theta,\theta^2,\theta^3,\theta^4,\theta^5 $\} is a  $8539$-integral basis of $K$. Similar is the situation for a prime $p$ not dividing $D$. So by Theorem \ref{p1.2}, $\{1,\theta,\theta^2,\theta^3/2,\theta^4/2,\theta^5/2 \}$ is an integral basis of $K$ and $d_K=\frac{D}{(ind~\theta)^2}=-2^{6}\times8539.$
\end{example}

 \section{Proof of Theorem \ref{p1.2}}
 Clearly $\alpha_i$ is integral over $\Z_{(p)}$ for all primes $p$ not belonging to the set $\{p_1,\cdots ,p_{s}\}$ for each $i$. Also in view of the choice of $c_{ij}$ we see that each $\alpha_i$ is integral over $\Z_{(p)}$ for $p$ belonging to the set $\{p_1,p_2,\cdots ,p_s\}$. So each $\alpha_i$ is integral over $\Z$. Therefore if $\Gamma$ denotes the   subgroup of $\C$ defined by $\Gamma =\Z\alpha_0+\Z\alpha_1+\cdots +\Z\alpha_{n-1}$, then $\Z[\xi]\subseteq\Gamma \subseteq A_L$. Further by virtue of a basic result (cf. \cite[Chapter 2, Section 2, Theorem 2]{Bor} )  the index of the subgroup $\Z[\xi]$ in $\Gamma$ being the absolute value of the determinant of the transition matrix from $\{\alpha_0,\alpha_1,\cdots,\alpha_{n-1}\}$ to $\{1,\xi,\cdots,\xi^{n-1}\}$ equals $\displaystyle\prod_{i=1}^{n-1}t_i$. In view of (\ref{eqn p2}), we have
 $$v_{p_{r}}(ind~\xi)=k_{1,r}+\cdots +k_{n-1,r} $$ for $1\leq r\leq s$. By hypothesis $p_1,\cdots ,p_s$ are the only primes dividing $ind~\xi$. Therefore $$ind~\xi=\displaystyle\prod_{r=1}^{s}p_r^{k_{1,r}+\cdots +k_{n-1,r} }=\displaystyle\prod_{i=1}^{n-1}(\displaystyle\prod_{r=1}^{s}p_r^{k_{i,r}})= \displaystyle\prod_{i=1}^{n-1}t_i.$$ Since $\Gamma \subseteq A_L$ and $[\Gamma:\Z[\xi]]=[A_L:\Z[\xi]]$, it follows that $A_L=\Gamma$ and hence $\{\alpha_0,\alpha_1,\cdots,\alpha_{n-1}\}$ is an integral basis of $L$.
\section{Preliminary Results}
 The following proposition to be used in the sequel follows immediately from what has been said in the paragraph preceding Theorem \ref{p1.2}.\\
 
 \noindent\textbf{Proposition 3.A.}\label{pr p1.1} Let $L=\Q(\xi)$ be an algebraic number field of degree $n$ with $\xi$ an algebraic integer and $p$ be a prime number. Let  $\alpha_1,\alpha_2,\cdots ,\alpha_{n-1}$ be $p$-integral elements of $L$ of the type  $\alpha_i=\frac{c_{i0}+c_{i1}\xi+\cdots +c_{ii-1}\xi^{i-1}+\xi^i}{p^{k_i}}$ where $c_{ij},~k_i$ are in $\Z$ with $0\leq k_i\leq k_{i+1}$ for $1\leq i\leq n-1$. Then  $\{1,\alpha_1,\cdots ,\alpha_{n-1}\}$ is a $p$-integral basis of $L$ if and only if    $v_p(ind~\xi)=\displaystyle\sum_{i=1}^{n-1}k_i$, in which case  the integers $k_1,~\cdots,k_{n-1}$ are uniquely determined by  the prime $p$ and the element $\xi$ of $L$. Morever there always exist a $p$-integral basis of $L$ of the above type.\\
 
 As usual, by a valuation $v$ of a field $K$,   we shall mean a mapping $v:K \longrightarrow\mathbb{R}\cup \{\infty\}$ which satisfies the following properties for all $\alpha,~\beta$ in $ K$.\\
 $(i)$ $v(\alpha)= \infty$ if and only if $\alpha=0,$\\
 $(ii)$ $v(\alpha\beta)= v(\alpha)+ v(\beta)$,\\
 $(iii)$ $v(\alpha+\beta)\geq $ min$\{v(\alpha), v(\beta)\}.$\\The pair $(K, ~v)$ is called a valued field. The subgroup $v(K^{\times})$ of $(\R,+)$ is called the value group of $v$.  The subring $R_v$ of  $K$ defined by ${R}_v= \left\lbrace \alpha\in K\mid v(\alpha)\geq 0\right\rbrace $ is called the valuation ring of $v$. It has unique maximal ideal $\mathcal {M}_v=\left\lbrace \alpha\in K\mid v(\alpha)> 0\right\rbrace $ and ${R}_v/\mathcal{M}_v$ is called the residue field of $v$. A valuation $v'$ of an overfield $K'$ of $K$ is said to be an extension or a prolongation of $v$ to $K'$  if $v'$ coinsides with $v$ on $K$.\\  \\
\textbf{Notations:} In what follows,  for a prime power $q$, we shall denote by $\F_q$ the field with $q$ elements. For a given prime $p$, $v_p$ will stand for the $p$-adic valuation of $\Q$  defined by $v_p(p)=1 $. We shall also denote by $v_p$ its unique prolongation to the field $\Q_p$ of $p$-adic numbers and by $\Z_p$ the ring of $p$-adic integers. $\tilde{v}_{p}$ will stand for the unique prolongation of the valuation $v_p$ of $\Q_p$ to a fixed algebraic closure $\tilde{\Q}_p$  of $\Q_p$. For any $\beta$ in the valuation ring of $\tilde{v}_{p}$, $\overline{\beta} $ will denote its $\tilde{v}_{p}$- residue, i.e., the image of $\beta$ under the canonical homomorphism from the valuation ring of $\tilde{v}_{p}$ onto its residue field. If $v_1$ denotes the restriction of $\tilde{v}_{p}$ to a subfield $K_1$ of $\tilde{\Q}_p$, then for any  polynomial $g_1(x)$  with coefficients in the valuation ring of $v_1$, $\overline{g_1}(x)$ will stand for the polynomial obtained by  replacing each coefficient of $g_1(x)$ by its $v_1$-residue in the residue field of $v_1$.

For proving Theorem \ref{p1.1},  we will use the classical Theorem of Index of Ore (stated as Theorem 3.C below) in addition to several simplifications. For this, we introduce the notions of Gauss valuation, $\phi$-Newton polygon, $\phi$-index of a polynomial where $\phi(x)$  belonging to $\Z_p[x]$ is a monic polynomial with $\overline{\phi}(x)$ irreducible over $\F_p$ .

We shall denote by $v_p^x$ the Gauss valuation of the field $\Q_p(x)$ of rational functions in an indeterminate $x$ which extends the valuation $v_p$ of $\Q_p$ and is defined on $\Q_p[x]$ by \begin{equation}\label{Gau}
v_p^x(\displaystyle\sum_{i}c_ix^i)= \displaystyle\min_i\{v_p(c_i)\}, c_i\in \Q_p.
 \end{equation}
\begin{definition}\label{p1.6}
Let $\phi(x)\in\Z_p[x]$ be a monic polynomial which is irreducible modulo $p$ and $F(x)\in\Z_p[x]$ be a monic polynomial not divisible by $\phi(x)$. Let $\displaystyle\sum_{i=0}^{n}a_i(x)\phi(x)^i$ with degree $a_i(x)<$ degree $\phi(x)$, $a_n(x)\neq 0$ be the $\phi(x)$-expansion of $F(x)$ obtained on dividing it by successive powers of $\phi(x)$. Let $P_i$ stands for the point in the plane having coordinates $(i,v_p^x(a_{n-i}(x)))$ when $a_{n-i}(x)\neq 0$, $0\leq i\leq n$. Let $\mu_{ij}$ denote the slope of the line joining the point $P_i$ with $P_j$ if $a_{n-i}(x)a_{n-j}(x)\neq 0$.  Let $i_1$ be the largest positive index not exceeding $n$ such that\\
\centerline{$\mu_{0i_1}=\min\{~\mu_{0j}~|~0<j\leq n,~a_{n-j}(x)\neq0\}.$}

\noindent
If $i_1<n,$ let $i_2$ be the largest index such that $i_1<i_2\leq n$ and\\
\centerline{$\mu_{i_1i_2}=\min\{~\mu_{i_1j}~|~i_1<j\leq n,~a_{n-j}(x)\neq0\}$} 
 and so on. The $\phi$-Newton polygon of $F(x)$ with respect to the prime $p$ is the polygonal path having segments
$P_{0}P_{i_1},P_{i_1}P_{i_2},\ldots,P_{i_{k-1}}P_{i_k}$ with $i_k=n$. These segments are called the edges of the $\phi$-Newton polygon and their slopes form a strictly increasing sequence; these slopes are non-negative as $F(x)$ is a monic polynomial with coefficients in $\Z_p$.
\end{definition}


\begin{definition}\label{p1.7}
Let $\phi(x)$ and $F(x)$ be as in Definition \ref{p1.6}. Let $N$ denote the number of points with positive integral coordinates lying on or below the $\phi$-Newton polygon of $F(x)$ away from the vertical line passing through the last vertex of this polygon. As in \cite{Kh-Ku}, the $\phi$-index of $F$ (with respect to $p$)  is defined to be $N\deg \phi(x)$ and will be denoted by $i_{\phi}(F)$.
\end{definition}
\begin{example}
 Consider $\phi(x)=x^2+1$. We determine the $\phi$-Newton polygon of the polynomial $F(x)=(x^2+1)^3+(3x+9)(x^2+1)^2+(12x+9)(x^2+1)+9x+81$ with respect to the prime $3$. The $\phi$-Newton polygon of $F(x)$ with respect to $3$ being the convex hull of the points (0, 0), (1, 1), (2, 1) and (3, 2) has two edges; the first edge  is the line segment joining the point  (0, 0) with (2, 1) and the second edge is the line segment joining  (2, 1) with (3, 2) as shown in the following diagram. Also $i_{\phi}(F)=2$.
 
\begin{figure}[htbp]

\hspace{30mm}
\includegraphics[width=0.50\textwidth]{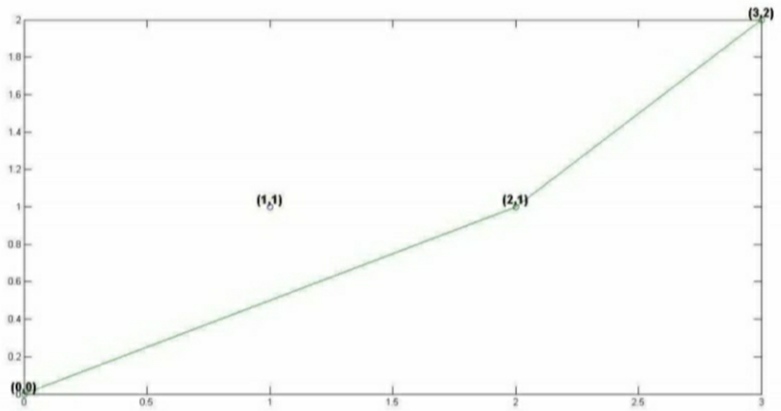}

\caption{$\phi$-Newton polygon of $F(x)$}

\end{figure}
\end{example}



With the above notation, we shall use the theorem stated below originally proved by Ore (see \cite[Theorem 1.2] {Kh-Ku}, \cite{Ore} ).\\

\noindent\textbf{Theorem 3.B.}\label{p1.9}
 {\it  Let $L=\Q(\xi)$ be an algebraic number field with $\xi$ satisfying  a monic  irreducible polynomial $g(x) \in \Z[x]$  and $p$ be a prime number. Let $ \overline{\phi}_{1}(x)^{e_1}\cdots\overline{\phi}_r{(x)}^{e_r}$ be the factorization of ${g}(x)$ modulo $p$ into a product of powers of distinct irreducible polynomials over $\F_p$ with each
$\phi_i(x) \neq g(x)$ belonging to $\Z[x]$ monic. Then }
\begin{equation*}\label{o1}\vspace*{-0.05in}
v_p(ind~\xi) \geq \sum\limits_{j=1}^{r}i_{\phi_j}(g).
\vspace*{-0.05in}\end{equation*} 

Ore also gave a sufficient condition so that the inequality in the above theorem becomes equality. For this he associated with each edge $S_{ij}$ having positive slope of the $\phi_i$- Newton polygon of $g(x)$ a polynomial $T_{ij}(Y)$ in an indeterminate $Y$ with coefficients from the field $\mathbb F_q$, $q=p^{deg\phi_i(x)}$ described in the following definitions.
\begin{definition}\label{p1.10}
Let $\phi(x) \in\Z_p[x]$  be a monic polynomial which is irreducible modulo $p$ having a root $\alpha$ in $\tilde{\Q}_{p}$. Let $F(x) \in \Z_p[x]$ be a monic polynomial not divisible by $\phi(x)$ having degree a multiple of $\deg \phi(x)$ with $\phi(x)$-expansion $\phi(x)^n + a_{n-1}(x)\phi(x)^{n-1} + \cdots + a_0(x)$. Suppose that the $\phi$-Newton polygon of $F(x)$  consists of a single edge, say $S$ having positive slope denoted by $\frac{d}{e}$ with $d, e$ coprime, i.e.,  $\min\{\frac{v_p^x(a_{n-i}(x))}{i}~\mid~1\leq i\leq n\} = \frac{v_p^x(a_0(x))}{n} = \frac{d}{e}$ so that $n$ is divisible by $e$, say $n=et$ and $v_p^x(a_{n-ej}(x)) \geq dj$ for $1\leq j\leq t$. Thus the polynomial $\frac{a_{n-ej}(x)}{p^{dj}}=b_j(x)$ (say)  has coefficients in $\Z_p$ and hence $b_j(\alpha)\in \Z_p[\alpha]$ for $1\leq j \leq t$.  The polynomial $T(Y)$ in an indeterminate $Y$ defined by   $T(Y) = Y^t + \sum\limits_{j=1}^{t} \overline{b_j}(\overline{\alpha})Y^{t-j}$ having coefficients in $\F_p[\overline{\alpha}]$ is said to be the polynomial associated to $F(x)$ with respect to $(\phi,S)$.
\end{definition}

With notations as in the above definition, note that a lattice point \footnote{A point in the plane having integer coordinates will be referred to as a lattice point}, $(i,v_p^x(a_{n-i}(x)))$ lies on $S$ if and only if $i$ is divisible by $e$, say $i=ej$ and $v_p^x(a_{n-i}(x))=\frac{di}{e}=dj$ i.e., $v_p^x(b_{j}(x))=0$ which is same as saying that  $\overline{b_j}(\overline{\alpha})\neq\overline{0}$  as $\overline{\phi}(x)$ is the minimal polynomial of $\overline{\alpha}$ over $\F_p$.\\ \\
\noindent\textbf{Example.} Let $F(x)=(x+5)^4-5$. Clearly $f(x)\equiv x^4~(mod~2)$. One can check that the $x$-Newton polygon of $F(x)$ with respect to the prime $2$ consists of only one edge $S$ joining the points (0, 0) and (4, 2) with the lattice point (2, 1) lying on it. With notations as in the above definition, we see that $e=2,~d=1$ and the polynomial associated to $F(x)$ with respect to $(x,S)$ is $T(Y)=Y^2+Y+\overline{1}$ belonging to $\F_2[Y]$.\\  

	The following definition extends the notion of associated polynomial when $F(x)$ is more general.
\begin{definition}\label{p1.11} Let $\phi(x), \alpha$ be as in Definition \ref{p1.10}.  Let $G(x)\in \Z_p[x]$ be a monic polynomial not divisible by $\phi(x)$ such that $\overline{G}(x)$ is a power of $\overline{\phi}(x)$. Let $\lambda_1 < \cdots < \lambda_k$ be the slopes of the edges of the $\phi$-Newton polygon of $G(x)$ and $S_i$ denote the edge with slope $\lambda_i$. In view of a classical result proved by Ore (cf. \cite[Theorem 1.5]{CMS}, \cite[Theorem 1.1]{SKS1}), we can write $G(x) = G_1(x)\cdots G_k(x)$, where the $\phi$-Newton polygon of $G_i(x) \in \Z_{{p}}[x]$ has a single edge, say $S_i'$ which is a translate of $S_i$. Let $T_i(Y)$ belonging to ${\F}_{p}[\overline{\alpha}][Y]$ denote the polynomial associated to $G_i(x)$ with respect to ($\phi,~S_i'$) described as in Definition \ref{p1.10}.  For convenience, the polynomial $T_i(Y)$  will be referred to as the polynomial  associated to $G(x)$ with respect to $(\phi,S_i)$. The polynomial $G(x)$ is said to be $p$-regular with respect to $\phi$ if none of the polynomials $T_i(Y)$  has a repeated root in the algebraic closure of $\F_p$, $1\leq i\leq k$. In general, if $F(x)$ belonging to $\Z_p[x]$ is a monic polynomial and $\overline{F}(x) = \overline{\phi}_{1}(x)^{e_1}\cdots\overline{\phi}_r{(x)}^{e_r}$ is its factorization modulo $p$ into irreducible polynomials with each $\phi_i(x)$ belonging to $\Z_p[x]$ monic and $e_i > 0$, then by Hensel's Lemma there exist monic polynomials $F_1(x), \cdots, F_r(x)$ belonging to $\Z_{{p}}[x]$ such that $F(x) = F_1(x)\cdots F_r(x)$ and $\overline{F}_i(x) = \overline{\phi}_i(x)^{e_i}$ for each $i$. The polynomial $F(x)$ is said to be $p$-regular (with respect to $\phi_1, \cdots, \phi_r$) if each $F_i(x)$ is ${p}$-regular with respect to $\phi_i$.
\end{definition}
With the above notations, we state below a celebrated theorem known as the Theorem of Index of Ore (cf.  \cite[Theorem 1.4]{Kh-Ku}, \cite{Ore}).\\

\noindent\textbf{Theorem 3.C.}\label{p1.12}
{\it Let $L=\Q(\xi),~g(x)$ and $\phi_1(x),\cdots ,\phi_r(x)$ be as in Theorem 3.B. If $g(x)$ is $p$-regular with respect to $\phi_1,\cdots ,\phi_r$, then \begin{equation*}\label{o1}\vspace*{-0.05in}
v_p(ind~\xi)=\sum\limits_{j=1}^{r}i_{\phi_j}(g).
\vspace*{-0.05in}\end{equation*}  }

Let $g(x),~p,~\phi_1(x),\cdots,\phi_r(x)$ be as in Theorem 3.B. Then by Hensel's Lemma, we can write $g(x)=g_1(x)\cdots g_r(x)$ where $g_i(x)\in \Z_p[x]$ is a monic polynomial with $\overline{g}_i(x)=\overline{\phi_i}(x)^{e_i}$. If $S$ is an edge of the $\phi_i$-Newton polygon of $g_i(x)$, then the polynomial associated to $g_i(x)$ with respect to $(\phi_i,~S)$ will be referred to as the polynomial associated to $g(x)$ with respect to $(\phi_i,~S')$ where $S'$ is the edge of the $\phi_i$-Newton polygon of $g(x)$ which is a translate of $S$, because the $\phi_i$ -Newton polygon of $g(x)$ can be obtained from the $\phi_i$ -Newton polygon of $g_i(x)$ by giving a horizontal translation in view of the following simple lemma proved in \cite[Proposition 1.2, Theorem 3.2]{CMS} and \cite[Corollary 2.5]{SKS1}.\\

\noindent\textbf{Lemma 3.D.}
Let $\phi(x)\in \Z_p[x]$ be a monic polynomial which is irreducible modulo $p$ and $F(x)$, $G(x)$ belonging to $\Z_p[x]$ be two monic polynomials not divisible by $\phi(x)$. If $\overline{\phi}(x)$ does not divide $\overline{G}(x)$, then the $\phi$-Newton polygon of $G(x)$ is a horizontal line segment and the $\phi$ -Newton polygon of $F(x)G(x)$ is obtained by adjoining this horizontal line segment to the translate of the  $\phi$-Newton polygon of $F(x)$.\\

As usual for a real number $\lambda$, $ \lfloor {\lambda} \rfloor$ stands for the largest integer not exceeding $\lambda$. The following proposition is proved in a more general set up in (\cite [ Proposition 2.2]{Kh-Ku}). Its proof is omitted.\\

\noindent\textbf{Proposition 3.E.}\label{1.13}
Let $\Q(\eta)$ be an algebraic number field, where $\eta$ is a root of a monic irreducible polynomial $h(x)$ belonging to $\Z_{(p)}[x]$, $\Z_{(p)}$ being the localisation of $\Z$ at a prime ideal $p\Z$. Let $\phi(x)\in~\Z_{(p)}[x]$ be a monic polynomial different from $h(x)$ which divides $h(x)$ modulo $p$ and is irreducible modulo $p$. Let $d$ denote the largest integer not exceeding $(deg~h$)/$deg~\phi$ and $q_j(x)$ the quotient obtained on dividing $h(x)$ by $\phi(x)^j$, $1\leq j\leq d$. If $y_{d-j}$ stands for the ordinate of the point with abscissa $d-j$ on the $\phi$-Newton polygon of $h(x)$ with respect to prime $p$, then $q_j(\eta)/p^{[y_{d-j}]}$ is integral over $\Z_{(p)}$.\\ \\
The elementary  lemma stated below is well known (cf. \cite[Problem 435]{Book}).\\

\noindent\textbf{Lemma 3.F.} Let $t,  n$ be positive integers. Let $\mathcal P$ denote the set of points in the plane with positive integer entries lying inside or on the triangle with vertices $(0,0), (n, 0), (n,t)$ which do not lie on the line $x = n$.  Then
 $\# \mathcal P = \sum\limits_{i=1}^{n-1} \Bigl\lfloor \dfrac{it}{n} \Bigr\rfloor  = \dfrac{1}{2} [(n-1)(t-1) + \gcd(t, n) -1].$

 \section{Proof of Theorem \ref{p1.1}}
 For any prime $p$, we shall denote by $S_{(p)}$ the integral closure of $\Z_{(p)}$ in $K$.\\
{ \bf Case E1 :} $v_2(a)
=0$. By virtue of (\ref{eq:p1.1}), we have $v_2(D)=0$. Using the relation $v_2(D)=2v_2(ind~\theta)+v_2(d_K)$, we see that $v_2(ind~\theta)=0.$ Therefore $\{1,\theta,\theta^2,\theta^3,\theta^4,\theta^5\}$ is a 2-integral basis of $K$.\\
{ \bf Case E2 :} $v_2(a)=v_2(b)=1$. In view of (\ref{eq:p1.1}), we have $v_2(D)=6$. In this case $f(x)\equiv x^6~(mod~2)$. The $x$-Newton polygon of $f(x)$ (with respect to prime 2) has a single edge, say $S$ joining the points (0, 0) and (6, 1) with slope $\frac{1}{6}$. The polynomial associated to $f(x)$ with respect to ($x,S)$ is linear. So Theorem 3.C is applicable. Applying this theorem, we see that $v_2(ind~\theta)$  equals the number of points with the positive integral coordinates lying on or below the $x$-Newton polygon of $f(x)$ which do not lie on the line $x=6$. So $v_2(ind~\theta)=0$ and hence $\{1,\theta,\theta^2,\theta^3,\theta^4,\theta^5\}$ is a 2-integral basis of $K$.\\
{ \bf Case E3 :} $v_2(a)\geq 2, ~v_2(b)=1$. It follows from (\ref{eq:p1.1}) that $v_2(D)=11$.  The $x$-Newton polygon of $f(x)$ has a single edge joining the points (0, 0) and (6, 1).  Arguing as in case E2, it can be seen that $v_2(ind~\theta)=0$ and hence $\{1,\theta,\theta^2,\theta^3,\theta^4,\theta^5\}$ is a 2-integral basis of $K$.\\
{ \bf Case E4 :} $v_2(a)=1, ~v_2(b)\geq 2$. By virtue of (\ref{eq:p1.1}), we have $v_2(D)=6.$ The $x$-Newton polygon of $f(x)$ has two edges . The first edge, say $S_1$  is the line segment joining the point (0, 0) with (5, 1) and the second edge, say $S_2$ is the line segment joining the point (5, 1) to  (6, $v_2(b))$.  The polynomial associated to $f(x)$ with respect to $(x,S_i)$ is linear for $i=1,2$. So Theorem 3.C is applicable. It follows from this theorem that $v_2(ind~\theta)=1$. Applying Proposition 3.E with $\phi(x)=x$ and $j=1$, we see that $\frac{\theta^5+a}{2}\in S_{(2)}$. So $\frac{\theta^5}{2}\in S_{(2)}$ as $v_2(a)=1$.  Since $v_2(ind~\theta)=1$, it follows from Proposition 3.A that $\{1, \theta,\theta^2,\theta^3,\theta^4,\frac{\theta^5}{2}\}$ is a 2-integral basis of $K$.\\
{ \bf Case E5 :} $v_2(a)=2, ~v_2(b)\geq 3.$ In this case  $v_2(D)=12$.  The $x$-Newton polygon of $f(x)$ has two edges . The first edge, say $S_1$  is the line segment joining the point (0, 0) with  (5, 2) and the second edge, say $S_2$ is the line segment joining the point (5, 2) with  (6, $v_2(b))$.  The polynomial associated to $f(x)$ with respect to $(x,S_i)$ is linear for $i=1,2$. Using Theorem 3.C and Lemma 3.F, it can be shown that $v_2(ind~\theta)=4$.  Applying Proposition 3.E with $\phi(x)=x$ and $j=1,2,3$, we see that $\frac{\theta^5+a}{2^2}$, $\frac{\theta^4}{2}$, $\frac{\theta^3}{2}$ are integral over $\Z_{(2)}$. By hypothesis, $v_2(a)=2$. So $\frac{\theta^5}{2^2}\in S_{(2)}.$ Since $v_2(ind~\theta)=4$, we conclude that $\{1, \theta,\theta^2,\frac{\theta^3}{2},\frac{\theta^4}{2},\frac{\theta^5}{2^2}\}$ is a 2-integral basis of $K$ keeping in mind Proposition 3.A. \\
{ \bf Case E6 :} $v_2(a)=v_2(b)=3$. In view of (\ref{eq:p1.1}), we have $v_2(D)=18$.  The $x$-Newton polygon of $f(x)$ has a single edge, say $S$ joining the points (0, 0) and (6, 3). The polynomial associated  to $f(x)$ with respect to ($x,S)$ is $Y^3+\overline{1}\in\F_2[Y]$ having no repeated roots. Using Theorem 3.C  and Lemma 3.F, we obtain $v_2(ind~\theta)=6$.  Applying Proposition 3.E with $\phi(x)=x$ and $j=1,2,3,4$, it can be seen that $\frac{\theta^5}{2^2}$, $\frac{\theta^4}{2^2}$, $\frac{\theta^3}{2}$, $\frac{\theta^2}{2}$ belongs to $ S_{(2)}$.  Since $v_2(ind~\theta)=6$, it now follows from Proposition 3.A that $\{1, \theta,\frac{\theta^2}{2},\frac{\theta^3}{2},\frac{\theta^4}{2^2},\frac{\theta^5}{2^2}\}$ is a 2-integral basis of $K$. \\
{ \bf Case E7 :} $v_2(a)\geq 4,~v_2(b)=3$. Keeping in mind (\ref{eq:p1.1}), we have $v_2(D)=21$.  The $x$-Newton polygon of $f(x)$ has a single edge joining the points (0, 0) and (6, 3). Arguing as in case E6, we see that $v_2(ind~\theta)=6$ and  $\{1, \theta,\frac{\theta^2}{2},\frac{\theta^3}{2},\frac{\theta^4}{2^2},\frac{\theta^5}{2^2}\}$ is a 2-integral basis of $K.$ \\
{ \bf Case E8 :} $v_2(a)=3, ~v_2(b)\geq 4.$ In this case  $v_2(D)=18$.  The $x$-Newton polygon of $f(x)$ has two edges . The first edge, say $S_1$  is the line segment joining the point (0, 0) with  (5, 3) and the second edge, say $S_2$ is the line segment joining the point (5, 3) with  (6, $v_2(b))$.  The polynomial associated to $f(x)$ with respect to $(x,S_i)$ is linear for $i=1,2$.  Using Theorem 3.C and Lemma 3.F, we obtain $v_2(ind~\theta)=7$.  Applying Proposition 3.E with $\phi(x)=x$ and $j=1,2,3,4$, we see that $\frac{\theta^5}{2^3}$, $\frac{\theta^4}{2^2}$, $\frac{\theta^3}{2}$, $\frac{\theta^2}{2}$ are integral over $\Z_{(2)}$.  Since $v_2(ind~\theta)=7$, it follows that $\{1, \theta,\frac{\theta^2}{2},\frac{\theta^3}{2},\frac{\theta^4}{2^2},\frac{\theta^5}{2^3}\}$ is a 2-integral basis of $K$ in view of Proposition 3.A. \\
{ \bf Case E9 :} $v_2(a)=4, ~v_2(b)\geq 5.$ In the present case $v_2(D)=24$.  The $x$-Newton polygon of $f(x)$ has two edges . The first edge  is the line segment joining the point (0, 0) with  (5, 4) and the second edge is the line segment joining the point (5, 4) with  (6, $v_2(b))$. Arguing as in case E8, we see that $v_2(ind~\theta)=10$ and $\{1, \theta,\frac{\theta^2}{2},\frac{\theta^3}{2^2},\frac{\theta^4}{2^3},\frac{\theta^5}{2^4}\}$ is a 2-integral basis of $K$. \\
{ \bf Case E10 :} $v_2(a)=v_2(b)=5$. We have $v_2(D)=30$.  The $x$-Newton polygon of $f(x)$ has a single edge, say $S$ joining the points (0, 0) and (6, 5). Note that the polynomial associated to $f(x)$ with respect to ($x,S)$ is linear.  Using  Theorem 3.C  and Lemma 3.F, we obtain $v_2(ind~\theta)=10$.  Applying Proposition 3.E with $\phi(x)=x$ and $j=1,2,3,4$, it can be  seen that $\frac{\theta^5}{2^4}$, $\frac{\theta^4}{2^3}$, $\frac{\theta^3}{2^2}$, $\frac{\theta^2}{2}$ are in $S_{(2)}$ and hence $\{1, \theta,\frac{\theta^2}{2},\frac{\theta^3}{2^2},\frac{\theta^4}{2^3},\frac{\theta^5}{2^4}\}$ is a 2-integral basis of $K$. \\
{ \bf Case E11 :} $v_2(a)\geq 6, ~v_2(b)=5.$ In this case $v_2(D)=31$.  The $x$-Newton polygon of $f(x)$ has a single edge joining the points (0, 0) and (6, 5). Arguing as in case E10, we see that $v_2(ind~\theta)=10$ and $\{1, \theta,\frac{\theta^2}{2},\frac{\theta^3}{2^2},\frac{\theta^4}{2^3},\frac{\theta^5}{2^4}\}$ is a 2-integral basis of $K$. \\
{ \bf Case E12 :} $v_2(a)=1,~b\equiv 3~(mod~4)$. By virtue of (\ref{eq:p1.1}) and the hypothesis, we have $v_2(D)=7$. In this case $f(x)\equiv (x^2+x+1)^2(x+1)^2~(mod~2)$. Set $\phi_1(x)=x^2+x+1$ and $\phi_2(x)=x+1$. The $\phi_1$-expansion of $f(x)$ is given by
\begin{equation}\label{eq p1.26}
f(x)=(x^2+x+1)^3-3x(x^2+x+1)^2+(2x-2)(x^2+x+1)+ax+b+1.
\end{equation}
If $v_2^x$ denotes the Gaussian valuation of $\Q(x)$  defined by (\ref{Gau}) extending the $2$-adic valuation of $\Q$, then  $v_2^x(ax+b+1)=min\{v_2(a),v_2(b+1)\}=1$. So the $\phi_1$-Newton polygon of $f(x)$ has a  single edge  of positive slope,  say $S$ joining the point (1, 0) with (3, 1). The polynomial associated to  $f(x)$ with respect to $(\phi_1,S)$ is linear and  $i_{\phi_{1}}(f)=0$. Write
\begin{equation}\label{eq F17 2}
f(x)=(x+1)^6-6(x+1)^5+15(x+1)^4-20(x+1)^3+15(x+1)^2+(a-6)(x+1)+(-a+b+1).
\end{equation}
It can be easily seen that the $\phi_2$-Newton polygon of $f(x)$ has a single edge of positive slope, say $S_1$   joining the point (4, 0) with (6, 1) and $i_{\phi_2}(f)=0$. Note that the  polynomial associated to $f(x)$ with respect to $(\phi_2,S_1)$ is also linear. Therefore $f(x)$ is 2-regular with respect to $\phi_1,~\phi_2$. Consequently it follows from  Theorem 3.C that $v_2(ind~\theta)= i_{\phi_{1}}(f)+i_{\phi_{2}}(f)=0$. Hence $\{1,\theta,\theta^2,\theta^3,\theta^4,\theta^5\}$ is a 2-integral basis of $K$.\\
{ \bf Case E13 :} $v_2(a)=1,~b\equiv 1~(mod~4),~v_2(D)$  odd. By virtue of (\ref{eq:p1.1}) and the hypothesis, we have $v_2(D)\geq 9$. Here $f(x)\equiv (x^2+x+1)^2(x+1)^2~(mod~2)$. Set $\phi_1(x)=x^2+x+1$. It can be seen  that the $\phi_1$-Newton polygon of $f$ is the same as in case E12 and  \begin{equation}\label{eq E12 1}
i_{\phi_{1}}(f)=0.
\end{equation} In the present case $a\equiv 2 ~(mod~4)$, $b\equiv 1 ~(mod~4)$ which implies that $v_2(a-6)\geq 2,~v_2(-a+b+1)\geq 2$. So the number of edges of the $(x+1)$-Newton polygon depends upon the value of $v_2(a-6)$ and $v_2(-a+b+1)$. To get rid of this uncertain situation, consider $\phi_2(x)=x-\beta$ where $\beta=\frac{-6b}{5a}$. Note that $v_2(\beta)=0$. The $\phi_2$-expansion of $f(x)$ can be written as
\begin{equation}\label{beta}
f(x)=(x-\beta)^6+6\beta(x-\beta)^5+15\beta^2(x-\beta)^4+20\beta^3(x-\beta)^3+15\beta^4(x-\beta)^2+f'(\beta)(x-\beta)+f(\beta).
\end{equation}
Denote $v_2(f(\beta)),~v_2(f'(\beta))$ by $s_0,~s_1$ respectively. On substituting for $\beta$ and keeping in mind (\ref{eq:p1.1}), we see that $f(\beta)=\beta^6+a\beta+b=\frac{-bD}{5^6a^6}$. Therefore 
\begin{equation}\label{eq:p1.20}
s_0=v_2(D)-6.
\end{equation}
Similarly in view of the fact that $f'(\beta)=6\beta^5+a=\frac{D}{5^5a^5}$, we have   $s_1=v_2(D)-5$. Therefore 
\begin{equation}\label{s}
s_1=s_0+1=v_2(D)-5.
\end{equation}
Since $v_2(D)\geq 9$, it follows from (\ref{eq:p1.20}) that  $s_0\geq 3$. Using this fact together with (\ref{beta}) and (\ref{s}), it can be easily seen that the $\phi_2$-Newton polygon of $f(x)$ has a single edge of positive slope, say $S_2$ joining the points (4, 0) and (6, $s_0)$ with slope $\frac{s_0}{2}$. In view of (\ref{eq:p1.20}) and the hypothesis $v_2(D)$  odd, we conclude that $s_0$ is odd. Therefore the  polynomial associated to $f(x)$ with respect to ($\phi_2,S_2)$ is linear; consequently $f(x)$ is 2-regular with respect to $\phi_1,~\phi_2$.   Using  Lemma 3.F and (\ref{eq:p1.20}), we have $$i_{\phi_2}(f)=\frac{s_0-1}{2}=\frac{v_2(D)-7}{2}.$$  Keeping in mind the above equality together with (\ref{eq E12 1}), it now follows from Theorem 3.C that \begin{equation}\label{eq E12 2}
v_2(ind~\theta)= i_{\phi_{1}}(f)+i_{\phi_{2}}(f)=\frac{v_2(D)-7}{2}.
\end{equation}   

Denote $\frac{v_2(D)-7}{2}$ by $k_0$. Applying Proposition 3.E with $\phi(x)$ replaced by $x-\beta$ and $j$ by  $1$, we see that the element  $\frac{1}{2^{k_0}}[\theta^5+\beta\theta^4+\beta^2\theta^3+\beta^3\theta^2+\beta^4\theta+\beta^5+a]$ belongs to  $S_{(2)}$. Since $v_2(6\beta^5+a)= s_1=v_2(D)-5>k_0$ in view of (\ref{s}), it now follows that the element $\frac{1}{2^{k_0}}[\theta^5+\beta\theta^4+\beta^2\theta^3+\beta^3\theta^2+\beta^4\theta-5\beta^5]=\eta~ (say)$ is in $S_{(2)}$. Now choose an integer $x_0$ such that 
$5(\frac{a}{2})x_0+3b\equiv 0~(mod~2^{k_0})$
; such an integer $x_0$ exists because $\frac{a}{2}$ is odd. Note that $2^{k_0}$ divides $\beta-x_0$  in $\Z_{(2)}$. Therefore the element $\eta_0$ of $K$ defined by  $$\eta_0=\frac{\theta^5+x_0\theta^4+x_0^2\theta^3+x_0^3\theta^2+x_0^4\theta-5x_0^5}{2^{k_0}}$$ is in $S_{(2)}$ because  $\eta-\eta_0 $ belongs to $\Z_{(2)}[\theta]\subseteq S_{(2)}$. Since $v_2(ind~\theta)=k_0=\frac{v_2(D)-7}{2}$ by virtue of (\ref {eq E12 2}), it follows that  $\{1,\theta,\theta^2,\theta^3,\theta^4,\eta_0\}$ is a 2-integral basis of $K$ in view of Proposition 3.A.\\
{ \bf Case E14 :} $v_2(a)=1,~b\equiv 1~(mod~4),~v_2(D)$  even, $D_2\equiv 1 ~ (mod~4)$.  Keeping in mind the hypothesis and  (\ref{eq:p1.1}), we have $v_2(D)\geq 8$. In this case $f(x)\equiv (x^2+x+1)^2(x+1)^2~(mod~2).$ Set $\phi_1(x)=x^2+x+1$, $\beta=\frac{-6b}{5a}$, $s_0=v_2(f(\beta))$ and $s_1=v_2(f'(\beta))$. Using (\ref{eq p1.26}), one can check that the $\phi_1$-Newton polygon of $f(x)$ has a single edge of positive slope,  say $S$ joining the point (1, 0) with (3, 1) and the polynomial associated to $f(x)$ with respect to $(\phi_1,S)$ is linear. Clearly \begin{equation}\label{eq E14 1}
 i_{\phi_1}(f)=0.
 \end{equation} Arguing as for the proof of (\ref{eq:p1.20}) and (\ref{s}), it can be shown that
 \begin{equation}\label{eq E14 2}
 s_0=v_2(D)-6,~s_1=s_0+1.
 \end{equation} Further the  $(x-\beta)$-Newton polygon of $f(x)$ has a single edge of positive slope, say $S_1$  joining the points (4, 0) and (6, $s_0)$ with slope $\frac{s_0}{2}$. It is clear from  (\ref{eq E14 2}) that $s_0\equiv v_2(D)~(mod~2)$. Keeping in mind the hypothesis $v_2(D)$ is even, we conclude that $s_0$ is even.  Therefore the  polynomial associated to $f(x)$ with respect to $(x-\beta,S_1)$ is $Y^2+\overline{1}$ which has a repeated root in the field $\F_2$. Thus $f(x)$ is not 2-regular with respect to $\phi_1,~x-\beta$. So we now consider another rational number $\delta$ with $v_2(\delta)=0$ for which $f(x)$ turns out to be 2-regular with respect to $\phi_1,~x-\delta$. Since $v_2(D)\geq 8,$ it follows that $s_0\geq 2$ in view of (\ref{eq E14 2}). Define an integer $u\geq 1$ by  
$u=\frac{s_0}{2}$. Consider the rational number $\delta=\frac{2^u-3b}{5a_2}$, where $a_2=\frac{a}{2}$. Note that $v_2(\delta)=0$. Set $\phi_2(x)=x-\delta$ and write 
\begin{equation}\label{ga}
f(x)=(x-\delta)^6+6\delta(x-\delta)^5+15\delta^2(x-\delta)^4+20\delta^3(x-\delta)^3+15\delta^4(x-\delta)^2+f'(\delta)(x-\delta)+f(\delta).
\end{equation}
We first prove  that \begin{equation}\label{eq E14}
v_2(f(\delta))\geq 2u+1,~v_2(f'(\delta))=u+1.
\end{equation} Using (\ref{eq E14}), we shall prove that  $f(x)$ is 2-regular with respect to $\phi_1,~\phi_2$.
 Substituting for $\delta$ in $f(\delta)=\delta^6+a\delta+b$, we have 
$$5^6a_2^6f(\delta)=(2^u-3b)^6+a(2^u-3b)(5a_2)^5+b(5a_2)^6;$$ 
the above equation on  applying binomial theorem and rearranging terms can be rewritten as 
\begin{equation}\label{lmm 1}
5^6a_2^6f(\delta)=2^{6u}-9.2^{5u+1}b+135b^22^{4u}-(3b)^3.5.2^{3u+2}+15(3b)^42^{2u}+(b-2^{u+1})(3^6b^5-5^5a_2^6).
\end{equation}
Recall that by virtue of (\ref{eq:p1.1}), $3^6b^5-5^5a_2^6=\frac{-D}{2^6}$. Therefore in view of the fact that $s_0=2u$, $2\nmid b$ and (\ref{eq E14 2}), we have $$v_2(3^6b^5-5^5a_2^6)=v_2(D)-6=s_0=2u, ~ v_2(b-2^{u+1})=0.$$ So the last two summands on the right hand side of (\ref{lmm 1}) have $v_2$-valuation $2u$ and clearly every other summand has valuation strictly greater than  $2u+1$. It now follows from (\ref{lmm 1}) that
$v_2(f(\delta))\geq 2u+1.$ On
 differentiating $f(x)$, the relation $6f(x)-xf'(x)=5ax+6b$ can be easily verified; substituting $x=\delta$ in this  equation and keeping in mind that $5a\delta+6b=2^{u+1}$ together with the fact $v_2(f(\delta))\geq 2u+1$, we see that $v_2(f'(\delta))=u+1$ which proves (\ref{eq E14}).
 
  We now show that \begin{equation}\label{lmm 3}
v_2(f(\delta))\geq 2u+2.
\end{equation}To prove (\ref{lmm 3}), we discuss two possibilities. 

Consider first the possibility when $v_2(D)$ is 8; then in view of (\ref{eq E14 2}), $s_0=2$ and $u=\frac{s_0}{2}=1$. Substituting $u=1$ in (\ref{lmm 1}) and dividing by $2^2$, we see that 
$$\frac{5^6a_2^6f(\delta)}{2^2}=2^4-2^4.9b+2^2.15.9b^2-2^3(3b)^35+15(3b)^4+(b-2^2)D_2.$$ On taking congruence modulo  $8$, the above equation shows that $$\frac{5^6a_2^6f(\delta)}{2^2}\equiv bD_2-1~(mod~8).$$ By hypothesis $b\equiv 1~(mod~4),~D_2\equiv 1 ~ (mod~4)$. So the above  congruence implies that $v_2(f(\delta))\geq 4=2u+2$ which proves (\ref{lmm 3}) in the present situation.

Consider now the possibility when $v_2(D)>8$. Then $s_0\geq 4$ and $u\geq 2$. Dividing (\ref{lmm 1}) by $2^{2u}$ and taking congruence modulo $8$, we have $$\frac{5^6a_2^6f(\delta)}{2^{2u}}\equiv bD_2-1~(mod~8).$$ Arguing as in the above paragraph, we see that, $v_2(f(\delta))\geq 2u+2$ which proves (\ref{lmm 3}).

Next we write the $\phi_2$-Newton polygon of $f(x)$ where $\phi_2(x)=x-\delta$ using the $\phi_2$-expansion of $f$ given by (\ref{ga}). Consider first the subcase when  $v_2(f(\delta))=2u+2$. Keeping in mind that $v_2(f'(\delta))=u+1$ proved in (\ref{eq E14}), it can be easily seen that in the present subacase  the $\phi_2 $-Newton polygon of $f$ has a single edge of positive slope, say $S_2 $ joining the points (4, 0) and $(6, 2u+2) $ with the lattice point $(5, u+1)$  lying on it. The polynomial associated to $f(x)$  with respect to $(\phi_2,S_2)$ is $Y^2+Y+\overline{1}$ which has no repeated roots.  So  $f(x)$ is 2-regular with respect to $\phi_1,~\phi_2$. In view of Theorem 3.C and  Lemma 3.F, we have $v_2(ind~\theta)=i_{\phi_1}(f)+i_{\phi_2}(f)= \frac{v_2(D)-4}{2}=k_1$ (say). Applying Proposition 3.E with $\phi(x)$ replaced by $\phi_2(x)$ and $j$ by $1$, we see that the element $\frac{1}{2^{k_1}}[\theta^5+\delta\theta^4+\delta^2\theta^3+\delta^3\theta^2+\delta^4\theta+\delta^5+a]$ is in $S_{(2)}$. Since $v_2(6\delta^5+a)=u+1=k_1$ in view of (\ref{eq E14 2}) and (\ref{eq E14}), it now follows that $\frac{1}{2^{k_1}}[\theta^5+\delta\theta^4+\delta^2\theta^3+\delta^3\theta^2+\delta^4\theta-5\delta^5]=\eta~ (say)$ is in $S_{(2)}$. Choose an integer $x_1$ such that 
$5(\frac{a}{2})x_1-2^u+3b\equiv 0~(mod~2^{k_1}).$ Consider the element  $\eta_1$ of $K$ defined by  \begin{equation}\label{eta}
\eta_1=\frac{\theta^5+x_1\theta^4+x_1^2\theta^3+x_1^3\theta^2+x_1^4\theta-5x_1^5}{2^{k_1}}.
\end{equation} Keeping in mind that $5(\frac{a}{2})\delta-2^u+3b=0$ by definition of $\delta$ and arguing as in case E13, we see that  $\eta_1\in S_{(2)}.$ Hence in view of Proposition 3.A, $\{1,\theta,\theta^2,\theta^3,\theta^4,\eta_1\}$ is a 2-integral basis of $K$. 

Consider now the subcase when $v_2(f(\delta))\geq 2u+3$.  Keeping in mind that $v_2(f'(\delta))=u+1$ proved in (\ref{eq E14}), it can be easily seen that the $\phi_2 $-Newton polygon of $f$ has two edges of positive slope, say $S_1'  $ and $S_2'$  joining the point (4, 0) with  (5, $u+1)$ and  the point (5,$ u+1$) with $(6, v_2(f(\delta)))$ respectively.  The polynomials associated with $f(x)$ with respect to both $(\phi_2,S_1')$ and $(\phi_2,S_2')$ are linear. So arguing as in the above subcase, we have  $v_2(ind~\theta)=\frac{v_2(D)-4}{2}$ and $\{1,\theta,\theta^2,\theta^3,\theta^4,\eta_1\}$ is a 2-integral basis of $K$ where $\eta_1$ is given by (\ref{eta}). Note that in both subcases $v_2(ind~\theta)=k_1=\frac{v_2(D)-4}{2}$.  \\
{ \bf Case E15 :} $v_2(a)=1,~b\equiv 1~(mod~4),~v_2(D)$  even, $D_2\equiv 3 ~ (mod~4)$.  By virtue of (\ref{eq:p1.1}) and the hypothesis, $v_2(D)\geq 8$ in this case. Here $f(x)\equiv (x^2+x+1)^2(x+1)^2~(mod~2).$ We take $\phi_1(x)=x^2+x+1$, $\beta=\frac{-6b}{5a}$, $v_2(f(\beta))=s_0$, $u=\frac{s_0}{2}$ and $\phi_2(x)=x-\delta$ where $\delta=\frac{2^u-3b}{5a_2}$ , $a_2=\frac{a}{2}$. Proceeding exactly in the case E14, we see that (\ref{eq E14}) holds, i.e., $v_2(f'(\delta))=u+1$ and $v_2(f(\delta))\geq 2u+1$. Note that in this case $v_2(f(\delta))=2u+1$ because $\frac{5^6a_2^6f(\delta)}{2^{2u}}\equiv bD_2-1~(mod~8)$ and $ bD_2-1\equiv 2~(mod~4)$  by hypothesis. Keeping in mind the values of $v_2(f(\delta)),~v_2(f'(\delta))$  and using (\ref{ga}), it can be easily seen that the  $\phi_2 $-Newton polygon of $f(x)$ has a single edge of positive slope, say $S' $  joining the point (4, 0) with (6, $2u+1)$ and the polynomial associated to $f(x)$ with respect to $(\phi_2,S')$ is linear. So $f(x)$ is 2-regular with respect to $\phi_1,~\phi_2$. In view of Theorem 3.C, we see that $v_2(ind~\theta)=i_{\phi_1}(f)+i_{\phi_2}(f)= \frac{v_2(D)-6}{2}=k_2$ (say).  Applying Proposition 3.E with $\phi(x)$ replaced by $\phi_2(x)$, $j$ by $1$ and arguing as in case E14, we see that the element  $\frac{1}{2^{k_2}}{[\theta^5+\delta\theta^4+\delta^2\theta^3+\delta^3\theta^2+\delta^4\theta-5\delta^5}]=\eta~ (say)$ is  2-integral. Choose an integer $x_2$ such that $
5(\frac{a}{2})x_2+3b\equiv 0~(mod~2^{k_2}).$
Consider the element $\eta_2$ of $K$ defined by  $$\eta_2=\frac{\theta^5+x_2\theta^4+x_2^2\theta^3+x_2^3\theta^2+x_2^4\theta-5x_2^5}{2^{k_2}}.$$ Arguing as in case E13, we see that  $\eta_2\in S_{(2)}$ and hence  $\{1,\theta,\theta^2,\theta^3,\theta^4,\eta_2\}$ is a 2-integral basis of $K$. \\
{ \bf Case E16 :} $v_2(a)\geq 2,~b\equiv 1~(mod~4)$. In this case, $v_2(D)=6$.  Set $\phi_1(x)=x^2+x+1$ and $\phi_2(x)=x+1$.  Using the $\phi_1$-expansion of $f$ given by (\ref{eq p1.26}) and the hypothesis, we see that the $\phi_1$-Newton polygon of $f(x)$ has a  single edge   of positive slope, say $S$  joining the point $(1, 0)$ with $(3, 1)$ and the  polynomial associated to $f(x)$ with respect to $(\phi_1,S)$ is linear. Similarly the $\phi_2$-expansion of $f$ given by (\ref{eq F17 2}) shows that the  $\phi_2$-Newton polygon of $f$ has a  single edge of positive slope, say $S_1$   joining the point (4, 0) with (6, 1) and the polynomial associated to $f(x)$ with respect to $(\phi_2,S_1)$ is linear. It now follows from Theorem 3.C that $v_2(ind~\theta)= i_{\phi_{1}}(f)+i_{\phi_{2}}(f)=0$. Hence $\{1,\theta,\theta^2,\theta^3,\theta^4,\theta^5\}$ is a 2-integral basis of $K$.\\
{ \bf Case E17 :} $v_2(a)\geq 2,~b\equiv 3~(mod~4)$. By virtue of (\ref{eq:p1.1}), we have $v_2(D)=6$. Set $\phi_1(x)=x^2+x+1$ and $\phi_2(x)=x+1$.  It is clear from the $\phi_1$-expansion of $f(x)$ given by (\ref{eq p1.26}) and the hypothesis that the $\phi_1$-Newton polygon of $f(x)$ is the convex hull of points (0, 0), (1, 0), (2, 1) and (3, min$\{v_2(a),v_2(b+1)\})$. Note that min$\{v_2(a),v_2(b+1)\}\geq 2$. In view of  $\phi_2$-expansion of $f(x)$ given by (\ref{eq F17 2}), we see that the $\phi_2$-Newton polygon of $f(x)$ is the convex hull of points (0, 0), (1, 1), (2, 0), (3, 2), (4, 0), (5, 1), (6, $v_2(-a+b+1)$). By virtue of hypothesis, we have $v_2(-a+b+1)\geq 2.$ The shapes of the $\phi_1$-Newton polygon and the $\phi_2$-Newton polygon  of $f(x)$ depend upon the values min$\{v_2(a),v_2(b+1)\}$ and $v_2(-a+b+1)$. The proof is split into four subcases depending upon these values.

Consider first the subcase when  min$\{v_2(a),v_2(b+1)\}=2$, $v_2(-a+b+1)=2$. Then the $\phi_1$-Newton polygon of $f(x)$ has a single edge of positive slope, say $S$  joining the points (1, 0) and (3, 2) with the lattice point (2, 1) lying on it. So the polynomial associated to $f(x)$ with respect to $(\phi_1,S)$ is $Y^2+Y+\overline{1}$. The supposition  $v_2(-a+b+1)=2$ implies that the $\phi_2$-Newton polygon of $f(x)$ has a single edge of positive slope, say $S'$  joining the points (4, 0) and (6, 2) with the lattice point (5, 1) lying on it. The polynomial associated to $f(x)$  with respect to $(\phi_1,S')$ is $Y^2+Y+\overline{1}$. So $f(x)$ is 2-regular with respect to $\phi_1,~\phi_2$. Therefore by Theorem 3.C, $v_2(ind~\theta)=i_{\phi_1}(f)+i_{\phi_2}(f)=3$. Applying Proposition 3.E with $\phi(x)$ replaced by $\phi_1(x)$ and $j$ by $1$, we see that $\frac{\theta^4-\theta^3+\theta-1}{2}=\gamma_1$ (say) belongs to $S_{(2)}$.  Again applying this proposition with $\phi(x)$ replaced by $\phi_2(x)=x+1$ and $j$ by $1$, we see that $\frac{\theta^5-\theta^4+\theta^3-\theta^2+\theta-1}{2}=\gamma_2~ (say)$ belongs to $S_{(2)}$. So $\gamma_2-\theta\gamma_1+\theta^2-\theta+1=\frac{\theta^3+1}{2}$ also belongs to $S_{(2)}$. Hence the set $\mathcal B:=  \{1,\theta,\theta^2,\frac{\theta^3+1}{2},~\frac{\theta^4+\theta}{2},~\frac{\theta^5+\theta^2}{2}\}$ is contained in $S_{(2)}$. Since $v_2(ind~\theta)=3$, it follows from Proposition 3.A that $\mathcal{B}$ is a 2-integral basis of $K$. 

Consider now the subcase when min$\{v_2(a),v_2(b+1)\}=2$, $v_2(-a+b+1)\geq 3$. In this subcase  the $\phi_1$-Newton polygon of  $f$ has a single edge of positive slope, say $S$  joining the points (1, 0) and (3, 2) with the lattice point (2, 1) lying on it. So the polynomial associated to $f(x)$ with respect to $(\phi_1,S)$ is $Y^2+Y+\overline{1}$. The $\phi_2$-Newton polygon of $f(x)$ has two edges of positive slope, say $S_1'  $ and $S_2'$  joining the point (4, 0) with  (5, 1) and the point (5, 1) with (6, $v_2(-a+b+1)) $ respectively. The polynomial associated to $f(x)$ with respect to $(\phi_2,S_i')$  is linear for $i=1,2$. So $f$ is 2-regular with respect to $\phi_1,~\phi_2$. Arguing as in the above subcase, it can be easily seen that $v_2(ind~\theta)=3$ and $\{1,\theta,\theta^2,\frac{\theta^3+1}{2},\frac{\theta^4+\theta}{2},\frac{\theta^5+\theta^2}{2}\}$ is a 2-integral basis of $K$.

Consider now the subcase when min$\{v_2(a),v_2(b+1)\}\geq 3$, $v_2(-a+b+1)=2$. The $\phi_1$-Newton polygon of $f$ has two edges of positive slope, say $S_1  $ and $S_2$  joining the point (1, 0) with  (2, 1) and the point  (2, 1) with (3, min$\{v_2(a),v_2(b+1)\}) $ respectively. The $\phi_2$-Newton polygon of $f(x)$ has a single edge of positive slope, say $S'$  joining the points (4, 0) and (6, 2) with the lattice  point (5, 1) lying on it. Arguing as in the first subcase, it can be easily seen that  $v_2(ind~\theta)=i_{\phi_1}(f)+i_{\phi_2}(f)=3$ and $\{1,\theta,\theta^2,\frac{\theta^3+1}{2},\frac{\theta^4+\theta}{2},\frac{\theta^5+\theta^2}{2}\}$ is a 2-integral basis of $K$. 

Finally consider the subcase when min$\{v_2(a),v_2(b+1)\}\geq 3$, $v_2(-a+b+1)\geq 3$. The $\phi_1$-Newton polygon of $f$ has two edges of positive slope, say $S_1  $ and $S_2$  joining the point (1, 0) with  (2, 1) and the point (2, 1) with (3, min$\{v_2(a),v_2(b+1)\}) $. The $\phi_2$-Newton polygon of $f$ has two edges of positive slope, say $S_1'  $ and $S_2'$  joining the point (4, 0) with  (5, 1) and the point (5, 1) with (6, $v_2(-a+b+1)) $ respectively. Arguing as in the first subcase it can be easily seen that  $v_2(ind~\theta)=3$ and $\{1,\theta,\theta^2,\frac{\theta^3+1}{2},\frac{\theta^4+\theta}{2},\frac{\theta^5+\theta^2}{2}\}$ is a 2-integral basis of $K$. Note that $v_2(ind~\theta)=3$ in each  subcase. This completes the proof in case E17.\\

It remains to deal with the situation when $v_2(b)=2,~v_2(a)\geq 2$ or $v_2(b)=4,~v_2(a)\geq 4$. We first show that in these situations there does not exist any $\beta\in\Q$ with $v_2(\beta)>0$ such that $f(x)$ is 2-regular with respect to $x-\beta$. 	Let $\beta$ be a rational number with $v_2(\beta)>0$. On expanding $f(x)$ as in (\ref{beta}),
 we see that the $(x-\beta)$-Newton polygon of $f(x)$ is lower convex hull of the set $$\mathcal P=\{(0, ~0),~ (1, ~1+v_2(\beta)),~ (2,~2v_2(\beta)),~ (3,~2+3v_2(\beta)), ~(4,~4v_2(\beta)),~ (5,~v_2(f'(\beta))),~ (6,~v_2(f(\beta)))\}.$$ Keeping in view the hypothesis $v_2(b)=2,~v_2(a)\geq 2$ and $v_2(\beta)>0$, it can be easily seen that in this situation the $(x-\beta)$-Newton polygon of $f(x)$ has a single edge, say $S$ joining the points (0, 0) and (6, 2) with no other point of $\mathcal P$ lying on it; consequently the polynomial associated to $f(x)$ with respect to $((x-\beta),~S)$ is $Y^2+\overline{1}$ which shows that $f(x)$ is not 2-regular with respect to ($x-\beta)$. Taking $\beta=0$, it follows from a well known result (cf. \cite[Chapter 2, Proposition 6.3]{Neu}, \cite[Proposition 3.6]{CMS} ) that $V_2(\theta)=\frac{1}{3}$ for any valuation $V_2$ of $K$ extending the $2$-adic valuation of $\Q$.  Therefore the factorisation of $2A_K$ into a product of prime ideals of $A_K$ will be one of the following types: either $2A_K=\wp^3,~N_{K/\mathbb{Q}}(\wp)=2^2$ or $2A_K=\wp^3\wp_1^3,~N_{K/\mathbb{Q}}(\wp)=N_{K/\mathbb{Q}}(\wp_1)=2$ or $2A_K=\wp^6,~N_{K/\mathbb{Q}}(\wp)=2$. Using a basic result (\cite[Theorem 4.24]{Nar}), we see that $2^4$ divides $d_K$ and that $2^6$ divides $d_K$ if and only if 2 is fully ramified in $K$.  It can be verified that the  situation is similar  if $v_2(b)=4,~v_2(a)\geq 4$. With these observations, we discuss the next case.\\
{ \bf Case E18 :} $v_2(a)=v_2(b)=2$. By virtue of (\ref{eq:p1.1}),  $v_2(D)=12$. The $x$-Newton polygon of $f(x)$ is the  line segment joining (0, 0) with (6, 2). There are exactly three points with positive integer coordinates  which lie on or below this Newton polygon and are away from the line $x=6$. So by Theorem 3.B., $v_2(ind~\theta)\geq 3$; consequently $v_2(d_K)\leq 6$. It will be shown in the following paragraph that  2 is fully ramified in $K$ which would imply in view of \cite[Theorem 4.24]{Nar} that $v_2(d_K)\geq 6$. Thus we would conclude that $v_2(d_K)=6$ and $v_2(ind~\theta)=3$. It remains to be shown that 2 is fully ramified in $K$.

Let $V_2$ be any prolongation of $v_2$ to $K$. Then by a well-known result (\cite[Chapter 2, Proposition 6.3]{Neu}), $V_2(\theta)=\frac{1}{3}$. Since $\theta^6+a\theta+b=0$, we have $$V_2\Big(\frac{\theta^6+b}{2 ^2}\Big)=V_2\Big(\frac{-a\theta}{2^2}\Big)=\frac{1}{3}.$$ If  $\frac{\theta^3}{2}$  and $\frac{b}{2^2}$ are denoted respectively by $\eta$ and $B$, then the  above equation can be rewritten as $$V_2(\eta^2+B)= V_2((\eta+B)^2-2\eta B+B-B^2)=\frac{1}{3}$$ which in view of the strong triangle law\footnote{It states that if $ \alpha,~\beta$ are elements of a valued field $(K,~v)$ with $v(\alpha)\neq v(\beta), $ then $v(\alpha+\beta)= min \left\lbrace v(\alpha), v(\beta)\right\rbrace$.} shows that $V_2((\eta+B)^2)=\frac{1}{3}$.  So $\frac{1}{6}$ belongs to the value group of $V_2$ and hence 2 is fully ramified in K. Applying Proposition 3.E taking $\phi(x)=x$ and $j=1,2,3$, we see that $\frac{\theta^5}{2},~\frac{\theta^4}{2},~\frac{\theta^3}{2}$ are 2-integral. Since $v_2(ind~\theta)=3$, it now  follows that $\{1,~ \theta,~ \theta^2~\frac{\theta^3}{2},~\frac{\theta^4}{2},~\frac{\theta^5}{2}\}$ is a 2-integral basis of $K$ in view of Proposition 3.A. This completes the proof in case E18.

Before taking up the remaining cases of Table I, we prove a simple lemma.
\begin{lemma}{\label{lemma p1}}
Let $\theta$ be a root of an irreducible trinomial $f(x)=x^6+ax+b $ belonging to $\Z[x]$. The following hold:\\
(i) When $v_2(b)=2,~v_2(a)\geq 3$, then $\frac{(\theta^3+2)\theta^2}{2^2}\in S_{(2)}.$\\
(ii) When $v_2(b)=2,~\frac{b}{4}\equiv 3~ (mod~4),~v_2(a)\geq 3+i$, then $\frac{(\theta^3+2)\theta^{1-i}}{2^2}\in S_{(2)}$ for $i=0,~1.$\\
(iii) When $v_2(b)=4,~v_2(a)\geq 4$, then $\frac{\theta^4+4\theta}{2^3}\in S_{(2)}.$\\
(iv) When $v_2(a)=v_2(b)=4$ and $\frac{b}{16}\equiv 1~ (mod~ 4)$, then $\frac{\theta^5+4\theta^2+8\theta}{2^4}\in S_{(2)}.$\\
(v) When $v_2(b)=4,~v_2(a)\geq 5$ and $\frac{b}{16}\equiv 3~ (mod~ 4)$, then $\frac{\theta^5+4\theta^2}{2^4}\in S_{(2)}.$\\
(vi
) When $v_2(b)=4,~v_2(a)\geq 6$ and $\frac{b}{16}\equiv 3~ (mod~ 4)$, then $\frac{\theta^3+4}{2^3}\in S_{(2)}.$\\
\end{lemma}
\begin{proof} 
For proving assertion (i), it is enough to prove that $V_2((\theta^3+2)\theta^2)\geq 2$ for each valuation $V_2$ of $K$ extending the 2-adic valuation $v_2$ of $\Q$ in view of a basic result (cf. \cite[Chapter 3, Section 4, Theorem 6]{Bor}). Let $V_2$ be such a valuation of $K$. Since the $x$-Newton polygon of $f(x)$ is the segment joining (0, 0) with (6, 2) , we have $V_2(\theta)=\frac{1}{3}$ by virtue of  \cite[Chapter 2, Proposition 6.3]{Neu}. In the proof of this lemma, $\eta$ will stand for $\frac{\theta^3}{2}$. Then 
\begin{equation}\label{lm p1}
V_2\Big(\eta^2+\frac{b}{4}\Big)=V_2\Big(\frac{\theta^6}{4}+\frac{b}{4}\Big)=V_2\Big(\frac{-a\theta}{4}\Big)=v_2(a)-2+\frac{1}{3}.
\end{equation}
To verify the desired inequality, we discuss two cases.

 Consider first the possibility when $\frac{b}{4}\equiv 1~ (mod~4)$, $v_2(a)\geq 3 $. Then (\ref{lm p1}) together with the hypothesis $v_2(a)\geq 3 $ shows that $V_2(\eta^2+1)\geq \frac{4}{3}$, i.e., $V_2((\eta+1)^2-2\eta)\geq \frac{4}{3}$. Hence in view of the strong triangle law, we have  $V_2((\eta+1)^2)=V_2(2\eta)=1$. Therefore 
\begin{equation}\label{lm p2}
V_2\Big(\frac{\theta^3}{2}+1\Big)=V_2(\eta+1)=\frac{1}{2};
\end{equation}
consequently $V_2((\theta^3+2)\theta^2)=2+\frac{1}{6}$ 	which proves the desired inequality.

Consider now the situation when $\frac{b}{4}\equiv 3~ (mod~4),~v_2(a)\geq 3$. Then (\ref{lm p1}) together with the hypothesis $v_2(a)\geq 3 $ gives  $V_2(\eta^2+3)\geq \frac{4}{3}$ which implies that
\begin{equation}\label{lm p3}
V_2((\eta+1)^2-2(\eta+1))\geq \frac{4}{3}.
\end{equation}
The above inequality shows that $V_2(\eta+1)\geq \frac{1}{3}$, for otherwise $V_2((\eta+1)^2)\neq V_2(2(\eta+1))$; consequently in view of the strong triangle law, we would have $$V_2((\eta+1)^2-2(\eta+1))=min\{V_2((\eta+1)^2),1+V_2(\eta+1)\}<\frac{4}{3}$$ which contradicts (\ref{lm p3}). Thus $V_2(\eta+1)\geq \frac{1}{3}$ and hence $V_2((\theta^3+2)\theta^2)\geq 2$.\\

The proof of assertion (ii) is split into two cases. Consider first the case when $v_2(a)=3,~\frac{b}{4}\equiv 3~ (mod~4).$ Let $V_2$ be any prolongation of $v_2$ to $K$. Then using (\ref{lm p1}) and arguing as in the above paragraph, we have \begin{equation}\label{lm p4}
V_2((\eta+1)^2-2(\eta+1))= \frac{4}{3}.
\end{equation}
Note that $V_2((\eta+1)^2)\neq V_2(2(\eta+1))$, for otherwise $V_2(\eta+1)=1$ and hence we would have, $V_2((\eta+1)^2-2(\eta+1))\geq 2$ which is not so. It now follows from (\ref{lm p4}) that $$\frac{4}{3}=min~\{2V_2(\eta+1),~1+ V_2(\eta+1)\};$$ in fact this minimum equals $2V_2(\eta+1)$, because otherwise $V_2(\eta+1)=\frac{1}{3}$ which again contradicts (\ref{lm p4}). Thus
 $V_2(\eta+1)=\frac{2}{3} $
 and hence by the strong triangle law
 \begin{equation}\label{lm p5}
 V_2(\frac{\theta^3}{2}+1)=\frac{2}{3}=V_2(\frac{\theta^3}{2}-1)
 \end{equation}
 which proves that  $V_2(\theta^4+2\theta)= 2$ in this case.\\
 
 Consider now the case when $v_2(a)\geq 4,~\frac{b}{4}\equiv 3~ (mod~4).$  Then using (\ref{lm p1}) and arguing as for the proof of (\ref{lm p2}) , we see that \begin{equation}\label{lm p6}
 V_2((\eta+1)^2-2(\eta+1))\geq 2.
 \end{equation} 
 The above inequality implies that $V_2(\eta+1)\geq 1$,  for otherwise $V_2((\eta+1)^2)\neq V_2(2(\eta+1))$;  consequently we would have $$V_2((\eta+1)^2-2(\eta+1))=min~\{2V_2(\eta+1),~1+ V_2(\eta+1)\}<2$$ contrary to (\ref{lm p6}). Thus $V_2(\eta+1)\geq 1$ and hence $V_2(\theta^3+2)\geq 2$ as desired.\\
 
 We now prove assertion (iii). In what follows, we denote $\frac{\theta^3}{4}$ by $\psi$. Since  the $x$-Newton polygon of $f(x)$ is the line segment joining (0, 0) with (6, 4) , we have  $V_2(\theta)=\frac{2}{3}$  for each valuation $V_2$ of $K$ extending the valuation $v_2$ of $\Q$ in view of \cite[Chapter 2, Proposition 6.3]{Neu}. Note that 
 \begin{equation}\label{lm p7}
 V_2\Big(\psi^2+\frac{b}{16}\Big)=V_2\Big(\frac{\theta^6}{16}+\frac{b}{16}\Big)=V_2\Big(\frac{-a\theta}{16}\Big)=v_2(a)-4+\frac{2}{3}.
 \end{equation}
 By hypothesis $v_2(a)\geq 4$. So the above equation implies that  $V_2\Big(\psi^2+\frac{b}{16}\Big)\geq \frac{2}{3}$. Using this inequality and arguing exactly as for the proof of assertion (i), one can prove this assertion splitting the proof into two subcases viz when $\frac{b}{16}\equiv 1$ or $3~ (mod~4)$.\\
 
 Next we prove assertion (iv). Let $V_2$ be as above. In view of (\ref{lm p7}) and the hypothesis $v_2(a)=4,~\frac{b}{16}\equiv 1 ~ (mod ~4)$, we see that $V_2(\psi^2+1)=\frac{2}{3}$, i.e., $V_2(\theta^6+16)=\frac{14}{3}$. Keeping in mind the strong triangle law, the last equality immediately gives $V_2(\theta^6-16)=\frac{14}{3}$ and hence \begin{equation}\label{lm p8}
  V_2(\theta^3-4)=V_2(\theta^3+4)=\frac{7}{3}.
 \end{equation}
 For proving assertion (iv), it is to be shown that \begin{equation}\label{lm p9}
 V_2(\theta^4+4\theta+8)\geq \frac{10}{3}.
 \end{equation}
 Denote $\theta^4+4\theta+8$ by $\xi$. Keeping in mind that the value group $G_2$ (say) of $V_2$ is either $\frac{1}{3}\Z$ or $\frac{1}{6}\Z$, the desired inequality (\ref{lm p9}) is proved as soon as it is shown that $V_2(\xi^2)\geq \frac{19}{3}$ and that $V_2(\xi^2)\geq \frac{20}{3}$ when $G_2=\frac{1}{6}\Z$, because $V_2(\xi^2)=\frac{19}{3}$ is possible only when $G_2=\frac{1}{6}\Z$. Since $\xi^2=\theta^8+8\theta^5+16\theta^4+16\theta^2+64\theta+64$. On replacing $\theta^8$ by $-a\theta^3-b\theta^2$, we have  $$\xi^2=16\theta^2~\big(\frac{-b}{16}+1\big)+(64\theta+16\theta^4)+16~\big(\frac{-a}{16}\theta^3+4\big)+8\theta^5.$$ By hypothesis $\frac{b}{16}\equiv 1 ~ (mod ~4)$. So $V_2(16\theta^2~\big(\frac{-b}{16}+1\big))\geq \frac{22}{3}.$ By virtue of (\ref{lm p8}), $V_2(64\theta+16\theta^4)={7}$.  Using (\ref{lm p8}) and the hypothesis that $\frac{a}{16}$ is odd, we have $V_2(16~\big(\frac{-a}{16}\theta^3+4\big))=V_2(16(\theta^3+ 4))=\frac{19}{3};$ also $V_2(8\theta^5)=\frac{19}{3}$. It now follows that 
$V_2(\xi^2)\geq\frac{19}{3};$
 further $V_2(\xi^2)>\frac{19}{3} $ when $2$ is fully ramified in $K$ because in this situation the residue field of $V_2$ is $\F_2$ and hence for any pair of elements $\gamma,~\gamma'$ in $K$ with $V_2(\gamma)= V_2(\gamma')$, one has $V_2(\gamma+\gamma')>V_2(\gamma')$. Thus we have shown that $V_2(\xi^2)\geq \frac{20}{3}$ whether the value group  of $V_2$ is $\frac{1}{3}\Z$ or $\frac{1}{6}\Z$. This proves (\ref{lm p9}) and hence assertion (iv).\\
 
 We now prove assertion (v). Let $V_2$ be any valuation of $K$ extending $v_2$. As in proof of assertion (iii), let $\psi$ stand for $\frac{\theta^3}{4}$. Using the hypothesis $v_2(a)\geq 5, ~\frac{b}{16}\equiv 3 ~(mod~4)$, it follows from (\ref{lm p7}) that $V_2(\psi^2+3)\geq \frac{5}{3}$, which implies that \begin{equation}\label{lm p11}
 V_2((\psi+1)^2-2(\psi+1))\geq\frac{5}{3}.
 \end{equation}  
 We infer from the above inequality that $V_2(\psi+1)\geq \frac{2}{3}$, for otherwise $V_2((\psi+1)^2)\neq V_2(2(\psi+1))$ and consequently $V_2((\psi+1)^2-2(\psi+1))=min~\{V_2((\psi+1)^2),~ 1+V_2(\psi+1)\}<\frac{5}{3}$ which contradicts (\ref{lm p11}). This contadiction proves that $V_2(\psi+1)\geq \frac{2}{3}$ i.e., $V_2(\frac{\theta^5+4\theta^4}{2^4})\geq 0$ as desired. 
 
  Finally we prove assertion (vi). By virtue of (\ref{lm p7}) and the hypothesis, we have $V_2(\theta^6-16)\geq 6$. Hence by the strong triangle law $V_2(\theta^3+4)=V_2(\theta^3-4)\geq 3$. This completes the proof of the lemma.
\end{proof}
We now discuss the  remaining cases of Table I.\\
{ \bf Case E19 :} $v_2(a)=3,~v_2(b)=2,~\frac{b}{4}\equiv 3~(mod~4)$. In view of (\ref{eq:p1.1}), $v_2(D)=16$.  Since the $x$-Newton polygon of $f(x)$ is the line  segment joining the points (0, 0) and (6, 2), $V_2(\theta)=\frac{1}{3}$ for each valuation $V_2$ of $K$ extending $v_2$. In view of Proposition 3.E and Lemma \ref{lemma p1}, the set $\mathcal B:=\{1,\theta,\theta^2,\frac{\theta^3}{2},\frac{\theta^4+2
\theta}{2^2},\frac{\theta^5+2\theta^2}{2^2}\}$ is contained in $S_{(2)}$. So $v_2(ind~\theta)\geq 5$ and consequently $v_2(d_K)\leq 6.$ 	We will prove that $2$ is fully ramified in $K$ which  implies that $v_2(d_K)\geq 6$. So we would conclude that $v_2(d_K)=6,~v_2(ind~\theta)=5$ and $\mathcal B$ would be a 2-integral basis of $K$  by virtue of Proposition 3.A. Let $V_2$ be as above. It will be shown that \begin{equation}\label{eq E19}
V_2(4+2\theta+\theta^4)=\frac{13}{6}. \end{equation}
  Denote $4+2\theta+\theta^4$ by $\eta_1$. On taking square and writing $\theta^8$ as $-a\theta^3-b\theta^2$,~$\frac{b}{4}$ as $3+4u_1$, we see that  $$\eta_1^2=4\theta^2(\theta^3-2)-8(\frac{a}{8}\theta^3-2)+(8\theta^4+16
\theta)-16u_1\theta^2.$$ Using (\ref{lm p5}) and the fact that $\frac{a}{8}$ is odd, we see that $V_2(4\theta^2(\theta^3-2))=\frac{13}{3}$,$~V_2(8(\frac{a}{8}\theta^3-2))=\frac{14}{3}$ and $V_2(8\theta^4+16
\theta)=5$. Since $V_2(16\theta^2)=\frac{14}{3},$ it now follows from the strong triangle law that $V_2(\eta_1^2)=\frac{13}{3}$ which proves (\ref{eq E19}).\\
{ \bf Case E20 :} $v_2(b)=2,~v_2(a)\geq 4,~\frac{b}{4}\equiv 3~(mod~4)$. It follows from (\ref{eq:p1.1}) that $v_2(D)=16$. By Lemma \ref{lemma p1}, the set $\mathcal B:=\{1,\theta,\theta^2,\frac{\theta^3+2}{2^2},\frac{(\theta^3+2
)\theta}{2^2},\frac{(\theta^3+2
)\theta^2}{2^2}\}$ is contained in $S_{(2)}$. Consequently $v_2(ind~\theta)\geq 6$ and hence $v_2(d_K)\leq 4$. But as pointed out in the paragraph preceding case E18, we have $v_2(d_K)\geq 4$. Thus $v_2(d_K)= 4$, $v_2(ind~\theta)= 6$ and $\mathcal B$ is a 2-integral basis of $K$ in view of Proposition 3.A.\\
{ \bf Case E21 :} $v_2(b)=2,~v_2(a)\geq 3,~\frac{b}{4}\equiv 1~(mod~4)$. Let $V_2$ be any valuation of $K$ extending $v_2$. It has been shown in (\ref{lm p2}) that in the present case we have $V_2(\frac{\theta^3}{2}+1)=\frac{1}{2}$. Since $V_2(\theta)=\frac{1}{3}$, it follows that the value group of $V_2$ is $\frac{1}{6}\Z$. Thus $2$ is fully ramified in $K$ and hence $v_2(d_K)\geq 6$. By virtue of  (\ref{eq:p1.1}), we have $v_2(D)=16$. It now follows that  $v_2(ind~\theta)\leq 5$. Using Proposition 3.E and Lemma \ref{lemma p1}, we see that  the set $\mathcal B:=\{1,\theta,\theta^2,\frac{\theta^3}{2},\frac{\theta^4
}{2},\frac{(\theta^3+2
)\theta^2}{2^2}\}$ is contained in $S_{(2)}$ which shows that $v_2(ind~\theta)\geq 4$. Claim is that $v_2(ind~\theta)=4$ and hence $\mathcal B$ will be a 2-integral basis of $K$. Keeping in mind Proposition 3.A, it is now clear that the claim is proved once we prove the following assertions.\\
(i) There do not exist any integers $x,~y,~z,~v$ such that $\frac{1}{2^2}[x+y\theta+z\theta^2+v\theta^3+\theta^4]$ belongs to $S_{(2)}.$\\
(ii) There do not exist any integers $x,~y,~z,~v,~w$ for which $\frac{1}{2^3}[x+y\theta+z\theta^2+v\theta^3+w\theta^4+\theta^5]$ belongs to $S_{(2)}.$

To prove (i) suppose to the contrary that there exists  $\xi_1=\frac{1}{4}[x+y\theta+z\theta^2+v\theta^3+\theta^4]$ in $S_{(2)}$ with  $x,~y,~z,~v$ in $\Z$. Then we will have $v_2(ind~\theta)=5$ and the family $\mathcal C_1:=\{1,\theta,\theta^2,\frac{\theta^3}{2},\xi_1,\frac{(\theta^3+2
)\theta^2}{2^2}\}$ will be a 2-integral basis of $K$ in view of Proposition 3.A. So the element $2\xi_1-\frac{\theta^4}{2}=\frac{1}{2}[x+y\theta+z\theta^2+v\theta^3]$ of $S_{(2)}$ would be expressible as linear combination of members of $\mathcal C_1$ with coefficients from $\Z_{(2)}$. This implies that $x,~y,~z$ are all even, say $x=2x_1,~y=2y_1,~z=2z_1$. So \begin{equation}\label{eq E21 1}
\xi_1=\frac{1}{2}(x_1+y_1\theta+z_1\theta^2)+\Big(\frac{v}{2}\Big)\frac{\theta^3}{2}+\frac{\theta^4}{2^2}.
\end{equation}
The proof of the assertion that $\xi_1\not\in S_{(2)}$ is split into two subcases.

Consider first the subcase when $v$ is even  say $v=2v_1$. On rewriting (\ref{eq E21 1}) as $$\xi_1-v_1\Big(\frac{\theta^3}{2}\Big)=\frac{1}{2}(x_1+y_1\theta+z_1\theta^2)+\frac{\theta^4}{2^2},$$ we see that  $S_{(2)}$ contains an element of the type $\frac{1}{2^2}[2x'+2y'\theta+2z'\theta^2+\theta^4]$ for some $x',~y',~z'\in \{0,1\}$.  Keeping in mind (\ref{lm p2}), it can be easily checked that for  $x',~y',~z'\in\{0,1\}$, $V_2(2x'+2y'\theta+2z'\theta^2+\theta^4)$ equals $1$ or $\frac{4}{3}$ or $\frac{10}{6}$ or $\frac{11}{6}$ according as $x'=1$ or the tuple $(x',~y')=(0,~0)$ or $(x',~y',~z')=(0,~1,~1)$ or $(x',~y',~z')=(0,~1,~0)$ .   So no element of the type $\frac{1}{2^2}(2x'+2y'\theta+2z'\theta^2+\theta^4)$ belongs to $S_{(2)}$ and hence assertion (i) is proved in this situation.

Consider now the subcase when $v=2v_1'+1$ (say) is odd. On substituting for $v $ in (\ref{eq E21 1}), we see that $\xi_1-v_1'\Big(\frac{\theta^3}{2}\Big)=\frac{1}{2}(x_1+y_1\theta+z_1\theta^2)+\frac{\theta^3}{2^2}+\frac{\theta^4}{2^2}$  is an element of $S_{(2)}$. Thus $S_{(2)}$ contains an element of the type $\frac{1}{2^2}[2x'+2y'\theta+2z'\theta^2+\theta^3+\theta^4]$ for some $x',~y',~z'\in \{0,1\}$. It can be easily seen that $x'$ can not be 0 because in that situation  $V_2(2y'\theta+2z'\theta^2+\theta^3+\theta^4)=1$ as $\theta^3$ is the only term with minimum valuation. If $x'=1$, then in view of (\ref{lm p2}), $V_2(2+2y'\theta+2z'\theta^2+\theta^3+\theta^4)$ equals $\frac{4}{3}$ or $\frac{3}{2}$ according as $y'=0$ or $y'=1$. So no element of the type $\frac{1}{2^2}[2x'+2y'\theta+2z'\theta^2+\theta^3+\theta^4]$  can belong to $S_{(2)}$. This completes the proof of assertion (i).

To prove asserion (ii) towards the claim, suppose to the contrary that there exists integers $x,~y,~z,~v,~w$ such that the element  $\frac{1}{2^3}[x+y\theta+z\theta^2+v\theta^3+w\theta^4+\theta^5]=\xi$ (say) belongs to $S_{(2)}.$ So the family $\mathcal C:=\{1,\theta,\theta^2,\frac{\theta^3}{2},\frac{\theta^4}{2},\xi\}$ will be  a 2-integral basis of $K$. Consider the element $2\xi-\big(\frac{\theta^5+2\theta^2}{2^2}\big)$ of $S_{(2)}$ and write it as a $Z_{(2)}$-linear combination of elements of $\mathcal C$. Since $$2\xi-\Big(\frac{\theta^5+2\theta^2}{2^2}\Big)=\frac{x}{4}+\Big(\frac{y}{4}\Big)\theta+\Big(\frac{z-2}{4}\Big)\theta^2+\Big(\frac{v}{2}\Big)\frac{\theta^3}{2}+\Big(\frac{w}{2}\Big)\frac{\theta^4}{2},$$ it follows that 4 divides $x,~y,~z-2$ and $v,~w$ are even. Write $x=4x_1,~y=4y_1,~z=4m+2,~v=2v_1,~w=2w_1$. Thus\begin{equation}\label{eq E21 3}
\xi=\frac{1}{2^3}[4x_1+4y_1\theta+(4m+2)\theta^2+2v_1\theta^3+2w_1\theta^4+\theta^5].
\end{equation}
We split the proof of this assertion into four subcases.

Consider first the subcase when $v_1,~w_1$ are both even say $v_1=2v_1',~w_1=2w_2 $. Substituting for $v_1,~w_1$, we can write (\ref{eq E21 3}) as $$\xi-v_1'\Big(\frac{\theta^3}{2}\Big)-w_2\Big(\frac{\theta^4}{2}\Big)=\frac{1}{8}[4x_1+4y_1\theta+(4m+2)\theta^2+\theta^5].$$ So $S_{(2)}$ contains an element of the type $\frac{1}{8}[4x'+4y'\theta\pm 2\theta^2+\theta^5]$ for some $x',~y'\in\{0,1\}$. Such an element does not belong to  $S_{(2)}$ because by virtue of (\ref{lm p2}), $V_2(\theta^3-2)=V_2(\theta^3+2)=\frac{3}{2}$ and hence $V_2(4x'+4y'\theta\pm 2\theta^2+\theta^5)$ equals $2$ or $\frac{13}{6}$ according as $x'=1$ or $x'=0$ . This contradiction proves assertion (ii) in the present subcase.

 Consider now the subcase when $v_1$ is even and $w_1$ is odd, say $v_1=2v_1',~w_1=2k+1 $. Substituting for $v_1,~w_1$, we can rewrite (\ref{eq E21 3})  as $$\xi-v_1'\Big(\frac{\theta^3}{2}\Big)-k\Big(\frac{\theta^4}{2}\Big)=\frac{1}{8}[4x_1+4y_1\theta+(4m+2)\theta^2+2\theta^4+\theta^5].$$ The above equation shows that $S_{(2)}$ contains an element of the type $\frac{1}{8}[4x'+4y'\theta\pm 2\theta^2+2\theta^4+\theta^5]$ for some $x',~y'\in\{0,1\}$. Keeping in mind that $V_2(\theta^3\pm 2)=\frac{3}{2}$ in view of  (\ref{lm p2}), it can be seen that  $V_2(4x'+4y'\theta+2\theta^4\pm 2\theta^2+\theta^5)$ equals $2$ or $\frac{13}{6}$ according as $x'=1$ or $x'=0.$ This contradiction proves assertion (ii) in the present subcase.
 
  Consider now the subcase when $v_1$ is odd and $w_1$ is even, say $v_1=2v_1''+1,~w_1=2w_2$. Now we write (\ref{eq E21 3})  as $$\xi-v_1''\Big(\frac{\theta^3}{2}\Big)-w_2\Big(\frac{\theta^4}{2}\Big)=\frac{1}{8}[4x_1+4y_1\theta+(4m+2)\theta^2+2\theta^3+\theta^5]$$ which shows that $S_{(2)}$ contains an element of the type  $\frac{1}{8}[4x'+4y'\theta\pm 2\theta^2+2\theta^3+\theta^5]$ for some $x',~y'\in\{0,1\}$. Using (\ref{lm p2}), one can check that $V_2(4x'+2\theta^3+4y'\theta\pm 2\theta^2+\theta^5)$ equals $2$ or $\frac{13}{6}$ according as $x'=0$ or $x'=1$ and hence assertion (ii) is proved in this subcase.
  
   Finally we consider the subcase when $v_1,~w_1$ are both odd, say $v_1=2v_1''+1,~w_1=2k+1$. In this situation (\ref{eq E21 3}) will imply that  $S_{(2)}$ contains an element of the type $\frac{1}{8}[4x'+4y'\theta\pm 2\theta^2+2\theta^3+2\theta^4+\theta^5]$ for some $x',~y'\in\{0,1\}$. Keeping in view  (\ref{lm p2}), it can be seen that $V_2(4x'+2\theta^3+4y'\theta+2\theta^4\pm 2\theta^2+\theta^5)$ equals $2$ or $\frac{13}{6}$ according as $x'=0$ or $x'=1$. This contradiction completes the proof of the assertion (ii) and hence the result is proved in case E21.\\
{ \bf Case E22 :} $v_2(a)=v_2(b)=4,~\frac{b}{16}\equiv 1~(mod~4)$. It is clear from (\ref{eq:p1.1}) that $v_2(D)=24$. Note that the $x$-Newton polygon of $f(x)$ is the line segment joining the points (0, 0) and (6, 4). On applying Proposition 3.E and Lemma  \ref{lemma p1}, we see that the set $\mathcal B:=\{1,\theta,\frac{\theta^2}{2},\frac{\theta^3}{2^2},\frac{\theta^4+4
\theta}{2^3},\frac{\theta^5+4\theta^2+8\theta}{2^4}\}$ is contained in $S_{(2)}$. So $v_2(ind~\theta)\geq 10$. Since $v_2(d_K)\geq 4$, we have $v_2(ind~\theta)=\frac{v_2(D)-v_2(d_K)}{2}\leq 10$. So we conclude that $v_2(ind~\theta)=10$ and $\mathcal B$ is a 2-integral basis of $K$.\\
{ \bf Case E23 :} $v_2(a)=v_2(b)=4,~\frac{b}{16}\equiv 3~(mod~4)$. It is immediate from (\ref{eq:p1.1}) that $v_2(D)=24$. Since the $x$-Newton polygon of $f(x)$ is the line  segment joining the points (0, 0) and (6, 4), $V_2(\theta)=\frac{2}{3}$ for each valuation $V_2$ of $K$ extending  $v_2$. It follows from Proposition 3.E and Lemma  \ref{lemma p1} that the set $\mathcal B:=\{1,\theta,\frac{\theta^2}{2},\frac{\theta^3}{2^2},\frac{\theta^4+4
\theta}{2^3},\frac{\theta^5+4\theta^2}{2^3}\}$ is contained in $S_{(2)}$. So $v_2(ind~\theta)\geq 9$ and hence $v_2(d_K)\leq 6$. We shall prove that $2$ is fully ramified in $K$ which would imply that $v_2(d_K)\geq 6$. Thus we would conclude that $v_2(d_K)= 6$,  $v_2(ind~\theta)=9$ and in view of  Proposition 3.A, $\mathcal B$ is a 2-integral basis of $K$.

  It will be shown that for any valuation $V_2$ of $K$ extending $v_2$, we have   \begin{equation}\label{eqq E19 2}
V_2(4+2\theta^2+\theta^3)=\frac{5}{2}; \end{equation}
this together with the fact that $V_2(\theta)=\frac{2}{3}$ will imply that the value group of $V_2$ is $\frac{1}{6}\Z$. We first show that  \begin{equation}\label{lm pp8}
  V_2(\theta^3-4)=V_2(\theta^3+4)=\frac{7}{3}.
 \end{equation} Using (\ref{lm p7}) and the hypothesis $v_2(a)=4,~\frac{b}{16}\equiv 3 ~ (mod ~4)$, we see that $V_2(\frac{\theta^6}{16}-1)=\frac{2}{3}$, i.e., $V_2(\theta^6-16)=\frac{14}{3}$. The last equality by virtue  of the strong triangle law immediately gives (\ref{lm pp8}).

Denote $4+2\theta^2+\theta^3$ by $\eta_2$. On taking square and writing $\theta^6$ as $-a\theta-b$, we see that $$\eta_2^2=4\theta^2(\theta^3+4)+4\theta(\theta^3-\frac{a}{4})+8(\theta^3-\frac{b}{8}+2).$$ Write $\frac{b}{16}$ as $4u_1+3$ and  $\frac{a}{16}$ as $2A+1$ with $u_1,~A$ in $\Z$. Using (\ref{lm pp8}), we have $V_2(4\theta^2(\theta^3+4))=\frac{17}{3}, $ $V_2(4\theta(\theta^3-\frac{a}{4}))=\frac{8}{3}+V_2(\theta^3-4-8A)=5$ and $ V_2(8(\theta^3-\frac{b}{8}+2)=3+V_2(\theta^3-8u_1-4)=\frac{16}{3}.$ It now follows using the strong triangle law that $V_2(\eta_2^2)=5$ which  proves (\ref{eqq E19 2}) and hence  the result in this case.\\
 { \bf Case E24 :} $v_2(a)=5,~v_2(b)=4,~\frac{b}{16}\equiv 3~(mod~4)$. It is immediate  from (\ref{eq:p1.1}) that $v_2(D)=26$. As in case E23, $V_2(\theta)=\frac{2}{3}$ for each valuation $V_2$ of $K$ extending  $v_2$. On applying Proposition 3.E and Lemma  \ref{lemma p1}, we see that the set $\mathcal B:=\{1,\theta,\frac{\theta^2}{2},\frac{\theta^3}{2^2},\frac{\theta^4+4
 \theta}{2^3},\frac{\theta^5+4\theta^2}{2^4}\}$ is contained in $S_{(2)}$. So $v_2(ind~\theta)\geq 10$ which implies $v_2(d_K)\leq 6$. We will prove that $2$ is fully ramified in $K$ which would imply that $v_2(d_K)\geq 6$. Thus we would conclude that $v_2(d_K)= 6$,  $v_2(ind~\theta)=10$ and by Proposition 3.A, $\mathcal B$ is a 2-integral basis of $K$.

 Let $V_2$ be a valuation of $K$ extending $v_2$.  Since $\theta^6+a\theta+b=0$, denoting $\frac{\theta^3}{4}$ by $\psi$, we have
 \begin{equation*}\label{lm pp7}
  V_2\Big(\psi^2+\frac{b}{16}\Big)=V_2\Big(\frac{\theta^6}{16}+\frac{b}{16}\Big)=V_2\Big(\frac{-a\theta}{16}\Big)=\frac{5}{3}.
  \end{equation*}
  By hypothesis $\frac{b}{16}\equiv 3~(mod~4)$. Therefore the above equation implies that $V_2(\psi^2-1)=\frac{5}{3}$, i.e., \begin{equation}\label{lm pp4}
V_2((\psi+1)^2-2(\psi+1))= \frac{5}{3}.
\end{equation}
Note that $V_2((\psi+1)^2)\neq V_2(2(\psi+1))$, for otherwise $V_2(\psi+1)=1$ which would imply $V_2((\psi+1)^2-2(\psi+1))\geq 2$ contradicting (\ref{lm pp4}). It now follows from (\ref{lm pp4}) and the strong triangle law that $\frac{5}{3}=min~\{V_2((\psi+1)^2),~1+ V_2(\psi+1)\};$ in fact this minimum equals $V_2((\psi+1)^2)$, because otherwise $V_2(\psi+1)=\frac{2}{3}$ which again contradicts (\ref{lm pp4}). Thus $V_2((\psi+1)^2-2(\psi+1))=V_2((\psi+1)^2)= \frac{5}{3}$ and hence $\frac{5}{6}$ belongs to the value group of $V_2$ as asserted.\\
{ \bf Case E25 :} $v_2(a)\geq 6,~v_2(b)=4,~\frac{b}{16}\equiv 3~(mod~4)$. By virtue of (\ref{eq:p1.1}), we have $v_2(D)=26$. On applying Proposition 3.E and Lemma  \ref{lemma p1}, we see that the set $\mathcal B:=\{1,\theta,\frac{\theta^2}{2},\frac{\theta^3+4}{2^3},\frac{\theta^4+4
\theta}{2^3},\frac{\theta^5+4\theta^2}{2^4}\}$ is contained in $S_{(2)}$. So $v_2(ind~\theta)\geq 11$. Since $v_2(d_K)\geq 4$, we conclude that $v_2(ind~\theta)\leq 11$. Therefore we have $v_2(ind~\theta)=11$ and by Proposition 3.A, $\mathcal B$ is a 2-integral basis of $K$.\\
{ \bf Case E26 :} $v_2(a)\geq 5,~v_2(b)=4,~\frac{b}{16}\equiv 1~(mod~4)$. It follows from (\ref{eq:p1.1}) that $v_2(D)=26$. As in case E23, $V_2(\theta)=\frac{2}{3}$ for each valuation $V_2$ of $K$ extending $v_2$.   In view of Proposition 3.E and Lemma \ref{lemma p1}, the set $\mathcal B:=\{1,\theta,\frac{\theta^2}{2},\frac{\theta^3}{2^2},\frac{\theta^4+4
\theta}{2^3},\frac{\theta^5}{2^3}\}$ is contained in $S_{(2)}$. So $v_2(ind~\theta)\geq 9$. Claim is that $v_2(ind~\theta)=9$. To prove the claim we first show that $2$ is fully ramified in $K$ i.e., the value group of $V_2$ is $\frac{1}{6}\Z$. For the last assertion,   it is enough to show that \begin{equation}\label{eqq E26}
V_2(\theta^3-4)=V_2(\theta^3+4)=\frac{5}{2}. \end{equation}
 In view of (\ref{lm p7}), we have $V_2(\frac{\theta^6}{16}+\frac{b}{16})\geq \frac{5}{3}$. By hypothesis $\frac{b}{16}\equiv 1~(mod~4)$, so $V_2({\theta^6}+16)\geq \frac{17}{3}$. Therefore using the strong triangle law, we see that  $V_2(\theta^6-16)=5$ and hence  $V_2(\theta^3-4)=V_2(\theta^3+4)=\frac{5}{2}$. This proves (\ref{eqq E26}) and thus $2$ is fully ramified in $K$ with $v_2(d_K)\geq 6$. So we have shown that \begin{equation}
\label{eqqqq E26}
9\leq v_2(ind~\theta) \leq 10.
\end{equation}

Suppose now that the claim  is false. Then by virtue of (\ref{eqqqq E26}), $v_2(ind~\theta)$ is $10$. Keeping in view Proposition 3.A and the fact that $\mathcal B$ $\subseteq~S_{(2)}$, we see that there exist $x,~y,~z,~v,~w$ in $\Z$ for which $\frac{1}{2^4}[x+y\theta+z\theta^2+v\theta^3+w\theta^4+\theta^5]=\xi~(say)$ is in $S_{(2)}$ and hence the set $\mathcal C:=\{1,\theta,\frac{\theta^2}{2},\frac{\theta^3}{2^2},\frac{\theta^4+4
\theta}{2^3},\xi\}$ will be a 2-integral basis of $K$.  Keeping in mind that $Tr_{K/\mathbb{Q}}(\theta^i)=0$ for $1\leq i\leq 4$ and $Tr_{K/\mathbb{Q}}(\theta^5)=-5a$, we see that $Tr_{K/\mathbb{Q}}(\xi)=\frac{6x}{16}-\frac{5a}{16}$ which must be in $\Z_{(2)}$ and hence $8 \mid x$ as $v_2(a)\geq 5$ by hypothesis. Write $x=8x_1$, so that $\xi=\frac{1}{2^4}[{8x_1+y\theta+z\theta^2+v\theta^3+w\theta^4+\theta^5}].$ On expressing the element $2\xi-x_1-(\frac{\theta^5}{2^3})$ of $S_{(2)}$ as a linear combination of members of $\mathcal C$, we have $$2\xi-x_1-\Big(\frac{\theta^5}{2^3}\Big)=\Big(\frac{y-4w}{8}\Big)\theta+\Big(\frac{z}{4}\Big)\frac{\theta^2}{2}+\Big(\frac{v}{2}\Big)\frac{\theta^3}{2^2}+w\Big(\frac{\theta^4+4\theta}{2^3}\Big).$$ Hence 8 divides $y-4w,~4$ divides $z$ and $v$ is even. Write $y=8l+4w,~z=4z_1,~v=2v_1$ with $l,z_1,v\in\Z$. Thus\begin{equation}\label{eqq E19 1}
\xi=\frac{1}{2^4}[{8x_1+(8l+4w)\theta+4z_1\theta^2+2v_1\theta^3+w\theta^4+\theta^5}].
\end{equation}
The proof is  split into four subcases.

Consider first the subcase when $v_1,~w$ are both even say $v_1=2v_1',~w=2w_1 $. Substituting for $v_1,~w$, we can write (\ref{eqq E19 1}) as $$\xi-v_1'\Big(\frac{\theta^3}{2^2}\Big)-w_1\Big(\frac{\theta^4+4\theta}{2^3}\Big)=\frac{1}{16}[8x_1+8l\theta+4z_1\theta^2+\theta^5].$$ So $S_{(2)}$ contains an element of the type $\frac{1}{16}[8x_1+8l\theta+4z_1\theta^2+\theta^5]$. Without loss of generality, we may assume that $x_1,~l$ belong to $ \{0,1\}$ and $z_1\in\{0,1,-1,2\}$.  Observe that $x_1$ cannot be $1$ because $V_2(8+8l\theta+4z_1\theta^2+\theta^5)=3$ in view of the strong triangle law as $8$ is the only term having minimum valuation. So $x_1$ has to be zero. Therefore $\frac{1}{16}[8l\theta+4z_1\theta^2+\theta^5]\in$ $S_{(2)}$. This is not possible because by virtue of (\ref{eqq E26}), we have  $V_2(\theta^3-4)=V_2(\theta^3+4)=\frac{5}{2}$ and hence $V_2(8l\theta+4z_1\theta^2+\theta^5)$ equals $\frac{10}{3}$ or $\frac{11}{3}$ or $\frac{23}{6}$ according as the pair  $(z_1,~ l)$ belongs to $\{0,2\}\times \{0,1\}$ or to $\{1,-1\}\times\{1\}$ or to  $\{1,-1\}\times \{0\}$. This contradiction proves the claim in the present subcase.

Now consider the subcase when $v_1$ is even  and $w$ is odd, say $v_1=2v_1',~w=2k+1$, then  (\ref{eqq E19 1}) can be rewritten as $$\xi-v_1'\big(\frac{\theta^3}{2^2}\big)-k\big(\frac{\theta^4+4\theta}{2^3}\big)=\frac{1}{16}[8x_1+(8l+4)\theta+4z_1\theta^2+\theta^4+\theta^5]$$ and hence $S_{(2)}$ contains an element of the type $\frac{1}{16}[8x_1+(8l+4)\theta+4z_1\theta^2+\theta^4+\theta^5]$ where $x_1,~l$ belong to $ \{0,1\}$ and $z_1\in\{0,1,-1,2\}$.  Observe that $x_1$ cannot be $1$ because $V_2(8+(8l+4)\theta+4z_1\theta^2+\theta^4+\theta^5)=3$ in view of (\ref{eqq E26}) and the strong triangle law as $8$ is the only term having minimum valuation. So $x_1$ has to be zero. Therefore $\frac{1}{16}[(8l+4)\theta+4z_1\theta^2+\theta^4+\theta^5]\in$ $S_{(2)}$. This is not possible because $V_2((8l+4)\theta+4z_1\theta^2+\theta^4+\theta^5)=V_2(4\theta+\theta^4)=\frac{19}{6}$ keeping in view (\ref{eqq E26}) and the strong triangle law. This contradiction proves the claim in the present subcase.

Next consider the subcase when $v_1$ is odd  and $w$ is even, say $v_1=2v_1''+1,~w=2w_1$, then  (\ref{eqq E19 1}) can be rewritten as $$\xi-v_1''\big(\frac{\theta^3}{2^2}\big)-w_1\big(\frac{\theta^4+4\theta}{2^3}\big)=\frac{1}{16}[8x_1+8l\theta+4z_1\theta^2+2\theta^3+\theta^5]$$ and hence $S_{(2)}$ contains an element of the type $\frac{1}{16}[8x_1+8l\theta+4z_1\theta^2+2\theta^3+\theta^5]$ where $x_1,~l$ belong to $ \{0,1\}$ and $z_1\in\{0, 1,-1,2\}$.  Observe that $x_1$ cannot be $0$ because $V_2(8l\theta+4z_1\theta^2+2\theta^3+\theta^5)=3$ in view of  the strong triangle law as $2\theta^3$ is the only term with minimum valuation. So $x_1$ has to be $1$. Therefore $\frac{1}{16}[8+8l\theta+4z_1\theta^2+2\theta^3+\theta^5]\in S_{(2)}.$  This is impossible because using (\ref{eqq E26}), it can be seen that $V_2(8+2\theta^3+8l\theta+4z_1\theta^2+\theta^5)$ equals $\frac{10}{3}$ or $\frac{21}{6}$ according as $z_1\in\{0,2\}$ or $z_1\in\{1,-1\}$. This contradiction proves the claim in the present subcase.

Finally consider the  subcase when $v_1,~w$ are both odd, say $v_1=2v_1''+1,~w=2k+1$, then  (\ref{eqq E19 1}) can be rewritten as $$\xi-v_1''\big(\frac{\theta^3}{2^2}\big)-k\big(\frac{\theta^4+4\theta}{2^3}\big)=\frac{1}{16}[8x_1+(8l+4)\theta+4z_1\theta^2+2\theta^3+\theta^4+\theta^5]$$ and hence $S_{(2)}$ contains an element of the type $\frac{1}{16}[8x_1+(8l+4)\theta+4z_1\theta^2+2\theta^3+\theta^4+\theta^5]$ where $x_1,~l$ belong to $ \{0,1\}$ and $z_1\in\{0,1,-1,2\}$. Using  (\ref{eqq E26}), it can be seen that $V_2(8x_1+2\theta^3+4z_1\theta^2+(8l+4)\theta+\theta^4+\theta^5)$ equals $3$ or $\frac{19}{6}$ according as $x_1$ is $0$ or $x_1$ is $1$. This contradiction proves the desired result in the present subcase and hence the claim is proved. This completes the proof of Table I.

We now take up Table II.\\
{ \bf Case F1 :} $v_3(a)
=0$. By virtue of (\ref{eq:p1.1}), we have $v_3(D)=0$ and hence  $\{1,\theta,\theta^2,\theta^3,\theta^4,\theta^5\}$ is a 3-integral basis of $K$.\\
{ \bf Case F2 :} $v_3(a)=v_3(b)=1$. In view of (\ref{eq:p1.1}), we have $v_3(D)=6$. The $x$-Newton polygon of $f(x)$ has a single edge joining the points (0, 0) and (6, 1). Arguing exactly as in case E2, we see that $v_3(ind~\theta)=0$ and $\{1,\theta,\theta^2,\theta^3,\theta^4,\theta^5\}$ is a 3-integral basis of $K$.\\
{ \bf Case F3 :} $v_3(a)\geq 2, ~v_3(b)=1$. It follows from (\ref{eq:p1.1}) that $v_3(D)=11$.  The $x$-Newton polygon of $f(x)$ has a single edge joining the points (0, 0) and (6, 1). Proceeding as in case E2 we see that, $v_3(ind~\theta)=0$ and  $\{1,\theta,\theta^2,\theta^3,\theta^4,\theta^5\}$ is a 3-integral basis of $K$.\\
{ \bf Case F4 :} $v_3(a)=1, ~v_3(b)\geq 2$. By virtue of (\ref{eq:p1.1}), we have $v_3(D)=6.$. The $x$-Newton polygon of $f(x)$ has two edges . The first edge is the line segment joining the point (0, 0) with (5, 1) and the second edge is the line segment joining the point (5, 1) with (6, $v_3(b))$. Arguing as in proof of E$4$, we see that  $v_3(ind~\theta)=1$ and  $\{1, \theta,\theta^2,\theta^3,\theta^4,{\theta^5}/{3}\}$ is a 3-integral basis of $K$.\\
{ \bf Case F5 :} $v_3(a)=v_3(b)=2$. It follows from (\ref{eq:p1.1}) that $v_3(D)=12$.  The $x$-Newton polygon of $f(x)$ has a single edge, say $S$ joining the points (0, 0) and (6, 2) with slope $\frac{1}{3}$. The polynomial belonging to $\F_3[Y]$ associated to $f(x)$ with respect to ($x,S)$ is  $Y^2+\overline{\big(\frac{b}{3^2}\big)}$ which has no repeated roots.  Using Theorem 3.C and Lemma 3.F, we conclude that $v_3(ind~\theta)=3$. Applying Proposition 3.E taking $\phi(x)=x$ and $j=1,~2,~3$, it can be seen that $\frac{\theta^5}{3},~\frac{\theta^4}{3},~\frac{\theta^3}{3}$ are integral over $\Z_{(3)}$. Since $v_3(ind~\theta)=3$, it follows that $\{1, \theta,\theta^2,\frac{\theta^3}{3},\frac{\theta^4}{3},\frac{\theta^5}{3}\}$ is a 3-integral basis of $K$ by virtue of Proposition 3.A.\\
{ \bf Case F6 :} $v_3(a)\geq 3, ~v_3(b)=2$. It is immediate from (\ref{eq:p1.1}) that $v_3(D)=16$.  The $x$-Newton polygon of $f(x)$ has a single edge, say $S$ joining the points (0, 0) and (6, 2) with slope $\frac{1}{3}$. The polynomial belonging to $\F_3[Y]$ associated to $f(x)$ with respect to ($x,S)$ is  $Y^2+\overline{\big(\frac{b}{3^2}\big)}$. Arguing as in case F5, it can be easily shown that $v_3(ind~\theta)=3$ and $\{1, \theta,\theta^2,\frac{\theta^3}{3},\frac{\theta^4}{3},\frac{\theta^5}{3}\}$ is a 3-integral basis of $K$.\\
{ \bf Case F7 :} $v_3(a)=2, ~v_3(b)\geq 3.$ By virtue of (\ref{eq:p1.1}), we have  $v_3(D)=12$.  The $x$-Newton polygon of $f(x)$ has two edges . The first edge is the line segment joining the point (0, 0) with (5, 2) and the second edge is the line segment joining the point (5, 2) with (6, $v_3(b))$.  Proceeding as in the proof of case E5, we see that $v_3(ind~\theta)=4$ and $\{1, \theta,\theta^2,\frac{\theta^3}{3},\frac{\theta^4}{3},\frac{\theta^5}{3^2}\}$ is a 3-integral basis of $K$.\\
{ \bf Case F8 :} $v_3(a)=3, ~v_3(b)\geq 4.$ It follows from (\ref{eq:p1.1}) that $v_3(D)=18$.  The $x$-Newton polygon of $f(x)$ has two edges. The first edge is the line segment joining the point (0, 0) with (5, 3) and the second edge is the line segment joining the point (5, 3) with (6, $v_3(b))$.   Arguing as in case E8, we see that $v_3(ind~\theta)=7$ and $\{1, \theta,\frac{\theta^2}{3},\frac{\theta^3}{3},\frac{\theta^4}{3^2},\frac{\theta^5}{3^3}\}$ is a 3-integral basis of $K$. \\
{ \bf Case F9 :} $v_3(a)=v_3(b)=4$. In this case, we have $v_3(D)=24$ .  The $x$-Newton polygon of $f(x)$ has a single edge, say $S$ joining the points (0, 0) and (6, 4). The polynomial associated  to $f(x)$ with respect to ($x,S)$ is  $Y^2+\overline{\big(\frac{b}{3^4}\big)}$. Arguing exactly as in case F5,  we see that $v_3(ind~\theta)=8$ and $\{1, \theta,\frac{\theta^2}{3},\frac{\theta^3}{3^2},\frac{\theta^4}{3^2},\frac{\theta^5}{3^3}\}$ is a 3-integral basis of $K$.\\
{ \bf Case F10 :} $v_3(a)\geq 5, ~v_3(b)=4$. It follows from (\ref{eq:p1.1}) that $v_3(D)=26$.  The $x$-Newton polygon of $f(x)$ has a single edge, say $S$ joining the points (0, 0) and (6, 4). The polynomial associated  to $f(x)$ with respect to ($x,S)$ is  $Y^2+\overline{\big(\frac{b}{3^4}\big)}$. Arguing as in case F5, it can be easily shown that $v_3(ind~\theta)=8$ and  $\{1, \theta,\frac{\theta^2}{3},\frac{\theta^3}{3^2},\frac{\theta^4}{3^2},\frac{\theta^5}{3^3}\}$ is a 3-integral basis of $K$.\\
{ \bf Case F11 :} $v_3(a)=4, ~v_3(b)\geq 5.$ By virtue of (\ref{eq:p1.1}), we see that $v_3(D)=24$.  The $x$-Newton polygon of $f(x)$ has two edges . The first edge  is the line segment joining the point (0, 0) with (5, 4) and the second edge is the line segment joining the point (5, 4) with (6, $v_3(b))$. Proceeding as in case E8, we see that $v_3(ind~\theta)=10$ and $\{1, \theta,\frac{\theta^2}{3},\frac{\theta^3}{3^2},\frac{\theta^4}{3^3},\frac{\theta^5}{3^4}\}$ is a 3-integral basis of $K$. \\
{ \bf Case F12 :} $v_3(a)=v_3(b)=5$. In view of (\ref{eq:p1.1}), we have $v_3(D)=30$.  The $x$-Newton polygon of $f(x)$ has a single edge joining (0, 0) with (6, 5).  Arguing as in the proof of case E10, we see that $v_3(ind~\theta)=10$ and $\{1, \theta,\frac{\theta^2}{3},\frac{\theta^3}{3^2},\frac{\theta^4}{3^3},\frac{\theta^5}{3^4}\}$ is a 3-integral basis of $K$. \\
{ \bf Case F13 :} $v_3(a)\geq 6, ~v_3(b)=5.$ Keeping in mind (\ref{eq:p1.1}), we see that $v_3(D)=31$.  The $x$-Newton polygon of $f(x)$ has a single edge joining the points (0, 0) and (6, 5). Proceeding in the proof of case E10, we see that $v_3(ind~\theta)=10$ and $\{1, \theta,\frac{\theta^2}{3},\frac{\theta^3}{3^2},\frac{\theta^4}{3^3},\frac{\theta^5}{3^4}\}$ is a 3-integral basis of $K$. \\
{ \bf Case F14 :} $v_3(a)=1,~b\equiv 1~(mod~3)$. By virtue of (\ref{eq:p1.1}) and the hypothesis, we have $v_3(D)=6$. In this case $f(x)\equiv (x^2+1)^3~(mod~3)$. Set $\phi_1(x)=x^2+1$. The $(x^2+1)$-expansion of $f(x)$ can be written as
\begin{equation}\label{eqq p1.26}
f(x)=(x^2+1)^3-3(x^2+1)^2+3(x^2+1)+ax+b-1.
\end{equation}
If $v_3^x$ denotes the Gaussian prolongation of $v_3$ to $\Q(x)$ defined by (\ref{Gau}), then $v_3^x(ax+b-1)=min\{v_3(a),v_3(b-1)\}=1$.
So the $\phi_1$-Newton polygon of $f(x)$ has a  single edge  of positive slope,  say $S$ joining the point (0, 0) with (3, 1) and the polynomial associated to  $f(x)$ with respect to $(\phi_1,S)$ is linear. With notations as in Theorem 3.C, it follows that $v_3(ind~\theta)= i_{\phi_{1}}(f)=0$. Hence $\{1,\theta,\theta^2,\theta^3,\theta^4,\theta^5\}$ is a 3-integral basis of $K$.\\
{ \bf Case F15 :} $v_3(a)\geq 2,~b\equiv 4$ or $7~(mod~9)$. By virtue of (\ref{eq:p1.1}), we have $v_3(D)=6$.  Set $\phi_1(x)=x^2+1$. Note that in this case $v_3^x(ax+b-1)=min\{v_3(a),v_3(b-1)\}=1$.  So the $\phi_1$-Newton polygon of $f$ has a  single edge  of positive slope joining the point (0, 0) with (3, 1). Arguing as in case F14, we see that $v_3(ind~\theta)= 0$ and $\{1,\theta,\theta^2,\theta^3,\theta^4,\theta^5\}$ is a 3-integral basis of $K$.\\
{ \bf Case F16 :} $v_3(a)\geq 2,~b\equiv 1~(mod~9)$. By virtue of (\ref{eq:p1.1}), $v_3(D)=6$ in this case.  Set $\phi_1(x)=x^2+1$. In view of the hypothesis,  $v_3^x(ax+b-1)=min\{v_3(a),v_3(b-1)\}\geq 2.$ Using (\ref{eqq p1.26}), it can be seen that the $\phi_1$-Newton polygon of $f$ has two  edges. The first edge, say $S_1$  is the line segment joining the point (0, 0) with (2, 1) and the second edge, say $S_2$ is the line segment joining the point (2, 1) with (3, $min\{v_3(a),v_3(b-1)\} )$. The polynomial associated to $f(x)$ with respect to $(x,S_i)$ is linear for $i=1,~2$. By virtue of Theorem 3.C, we have  $v_3(ind~\theta)= i_{\phi_{1}}(f)=2$. Applying Proposition 3.E with $\phi(x)$ replaced by $\phi_1(x)$ and $j$ by $1$, we see that $\frac{\theta^4-\theta^2+1}{3}$ is 3-integral. Hence the set $\{1,\theta,\theta^2,\theta^3,\frac{\theta^4-\theta^2+1}{3},\frac{\theta^5-\theta^3+\theta}{3}\}$ is contained in $S_{(3)}$. Since $v_3(ind~\theta)=2$, it follows from Proposition 3.A that $\{1,\theta,\theta^2,\theta^3,\frac{\theta^4-\theta^2+1}{3},\frac{\theta^5-\theta^3+\theta}{3}\}$ is a 3-integral basis of $K $.\\
{ \bf Case F17 :} $v_3(a)\geq 2,$  $b\equiv 2$ or $5~(mod~9)$. We have $v_3(D)=6$. In this case $f(x)\equiv (x-1)^3 (x+1)^3~(mod~3)$. Set $\phi_1(x)=x-1$ and $\phi_2(x)=x+1$. The $\phi_1$-expansion of $f(x)$ is given by
\begin{equation}\label{eqq F17}
f(x)=(x-1)^6+6(x-1)^5+15(x-1)^4+20(x-1)^3+15(x-1)^2+(a+6)(x-1)+(a+b+1).
\end{equation}
Using the hypothesis and the above expansion, it can be easily seen that the $\phi_1$-Newton polygon of $f$ has a  single edge  of positive slope,  say $S$ joining the points (3, 0) and (6, 1). Note that the polynomial associated to  $f(x)$ with respect to $(\phi_1,S)$ is linear and $i_{\phi_{1}}(f)=0$.
Keeping in mind the $\phi_2$-expansion of $f(x)$ given by (\ref{eq F17 2}) and using conditions on $a,~b$, one can show that the $\phi_2$-Newton polygon of $f(x)$ has a single edge of positive slope, say $S_1$   joining the point (3, 0) with (6, 1) and $i_{\phi_2}(f)=0$. Note that the  polynomial associated to $f(x)$ with respect to $(\phi_2,S_1)$ is also linear. Therefore $f(x)$ is 3-regular with respect to $\phi_1,~\phi_2$. Consequently it follows from  Theorem 3.C that $v_3(ind~\theta)= i_{\phi_{1}}(f)+i_{\phi_{2}}(f)=0$. Hence $\{1,\theta,\theta^2,\theta^3,\theta^4,\theta^5\}$ is a 3-integral basis of $K$.\\
{ \bf Case F18 :} $v_3(a)\geq 2,~b\equiv -1~(mod~9)$. By virtue of (\ref{eq:p1.1}), we have $v_3(D)=6$. In this case, $f(x)\equiv (x-1)^3 (x+1)^3~(mod~3)$. Set $\phi_1(x)=x-1$ and $\phi_2(x)=x+1$. By hypothesis $v_3(a)\geq 2,~b\equiv -1~(mod~9)$; so $a-6\equiv 6~(mod~9)$ and $a+b+1\equiv 0~(mod~9)$. Therefore using (\ref{eqq F17}), it can be easily seen that the $\phi_1$-Newton polygon of $f(x)$ has two  edges  of positive slope. The first edge, say $S_1$  is the line segment joining the point (3, 0) with (5, 1) and the second edge, say $S_2$ is the line segment joining the point (5, 1) with (6, $v_3(a+b+1)$). The polynomial associated to $f(x)$ with respect to $(x,S_i)$ is linear for $i=1,~2$ and $ i_{\phi_{1}}(f)=1$. 

 Using the $\phi_2$-expansion of $f$ given by (\ref{eq F17 2}) and arguing as above, one can verify that the $\phi_2$-Newton polygon of $f(x)$ has two  edges  of positive slope. The first edge, say $S_1'$  is the line segment joining the point (3, 0) with (5, 1) and the second edge, say $S_2'$ is the line segment joining the point (5, 1) with (6, $v_3(-a+b+1)$). The polynomial associated to $f(x)$ with respect to $(x,S_i')$ is linear for $i=1,~2$.  Therefore $f(x)$ is 3-regular with respect to $\phi_1,~\phi_2$. It now follows from Theorem 3.C that  $v_3(ind~\theta)= i_{\phi_{1}}(f)+i_{\phi_{2}}(f)=2$.  Applying Proposition 3.E with $\phi(x)$ replaced by $\phi_1(x)$ and $j$ by $1$, we see that $\frac{\theta^5+\theta^4+\theta^3+\theta^2+\theta+1}{3}=\gamma_1$ (say) belongs to $S_{(3)}$.  Again applying this proposition taking $\phi(x)=x+1$ and $j=1$, we obtain an element   $\frac{\theta^5-\theta^4+\theta^3-\theta^2+\theta-1}{3}=\gamma_2~ (say)$ in  $S_{(3)}$. So $\gamma_2-\gamma_1+\theta^4+\theta^2+1
 =\frac{\theta^4+\theta^2+1}{3}$ also belongs to $S_{(3)}$. Hence the set $\mathcal B:=  \{1,\theta,\theta^2,\theta^3,\frac{\theta^4+\theta^2+1}{3},\frac{\theta^5+\theta^3+\theta}{3}\}$ is contained in $S_{(3)}$. Since $v_3(ind~\theta)=2$, it follows from Proposition 3.A that $\mathcal{B}$ is a 3-integral basis of $K$. \\
  { \bf Case F19 :} $a\equiv \pm 3~(mod~9),~b\equiv 2~(mod~9)$. In this case $v_3(D)=7$ and $f(x)\equiv (x-1)^3 (x+1)^3~(mod~3)$. Consider first the subcase when $a\equiv 3~(mod~9)$. Set $\phi_1(x)=x-1$ and $\phi_2(x)=x+1$. One can check that the $\phi_1$-Newton polygon of $f$ has a  single edge  of positive slope,  say $S$ joining the point (3, 0) with (6, 1) and the polynomial associated to  $f$ with respect to $(\phi_1,S)$ is linear.  Similarly using  the $\phi_2$-expansion of $f$ given by  (\ref{eq F17 2}), it can be easily seen that the $\phi_2$-Newton polygon of $f$ has two  edges  of positive slope. The first edge, say $S_1$  is the line segment joining the point (3, 0) with (5, 1) and the second edge, say $S_2$ is the line segment joining the point (5, 1) with (6, $v_3(-a+b+1)$). The polynomial associated to $f$ with respect to $(x,S_i)$ is linear for $i=1,~2$.  Thus $f(x)$ is 3-regular with respect to $\phi_1,~\phi_2$. So by Theorem 3.C,  $v_3(ind~\theta)= i_{\phi_{1}}(f)+i_{\phi_{2}}(f)=1$.  Applying Proposition 3.E with $\phi(x)=x+1$ and $j=1$, we see that $\frac{\theta^5-\theta^4+\theta^3-\theta^2+\theta-1}{3}$ belongs to $S_{(3)}$. Since $v_3(ind~\theta)= 1$,  it now follows from Proposition 3.A that $\{1,\theta,\theta^2,\theta^3,\theta^4,\frac{\theta^5-\theta^4+\theta^3-\theta^2+\theta-1}{3}\}$ is a 3-integral basis of $K$. Now consider the subcase when $a\equiv -3~(mod~9)$. Since $-\theta$ is a root of $x^6-ax+b$, it is immediate from the above subcase that $v_3(ind~\theta)= 1$ and  $\{1,\theta,\theta^2,\theta^3,\theta^4,\frac{\theta^5+\theta^4+\theta^3+\theta^2+\theta+1}{3}\}$ is a 3-integral basis of $K$.\\
  { \bf Case F20 :} $a\equiv \pm 3~(mod~9),~b\equiv -1~(mod~9)$. In view of (\ref{eq:p1.1}), $v_3(D)=7$. Here $f(x)\equiv (x-1)^3 (x+1)^3~(mod~3)$. Set $\phi_1(x)=x-1$ and $\phi_2(x)=x+1$. Consider first the subcase when $a\equiv -3~(mod~9)$. One can verify that the $\phi_1$-Newton polygon of $f$ has a  single edge  of positive slope,  say $S_1$ joining the point (3, 0) with (6, 1) and the polynomial associated to  $f$ with respect to $(\phi_1,S_1)$ is linear. It can also be shown that the $\phi_2$-Newton polygon of $f$ has a single edge of positive slope, say $S_2$   joining the point (3, 0) with (6, 1) and  the  polynomial associated to $f$ with respect to $(\phi_2,S_2)$ is also linear. Therefore $f$ is 3-regular with respect to $\phi_1,~\phi_2$. Consequently it follows from  Theorem 3.C that $v_3(ind~\theta)= i_{\phi_{1}}(f)+i_{\phi_{2}}(f)=0$. Hence $\{1,\theta,\theta^2,\theta^3,\theta^4,\theta^5\}$ is a 3-integral basis of $K$. Now consider the subcase when $a\equiv 3~(mod~9)$. Changing $\theta$ to $-\theta$, we infer from the above subcase that $v_3(ind~\theta)=0$ and $\{1,\theta,\theta^2,\theta^3,\theta^4,\theta^5\}$ is a 3-integral basis of $K$.\\
  { \bf Case F21 :} $a\equiv \pm 3 ~(mod~9),~b\equiv 5~(mod~9),~v_3(D)=$ 8.  Here $f(x)\equiv (x-1)^3 (x+1)^3~(mod~3)$. Consider first the subcase when $a\equiv -3~(mod~9)$. Set $\phi_1(x)=x-1$. It can be shown that the $\phi_1$-Newton polygon of $f(x)$ has a single edge of positive slope, say $S$ joining the point (3, 0) with (6, 1) and the polynomial associated to $f(x)$ with respect to $(\phi_1,S)$ is linear with $
 i_{\phi_{1}}(f)=0.$
The hypothesis of this subcase implies that $v_2(a-6)\geq 2,~v_2(-a+b+1)\geq 2$. So the number of edges of the $(x+1)$-Newton polygon depends upon the value of $v_3(a-6)$ and $v_3(-a+b+1)$. To get rid of this uncertain situation, consider $\phi_2(x)=x-\beta$ where $\beta=\frac{-6b}{5a}\equiv -1 ~(mod~3)$ in view of the hypothesis. Denote $v_3(f(\beta)),~v_3(f'(\beta))$ by $s_0,~s_1$ respectively. On substituting for $\beta$ and keeping in mind (\ref{eq:p1.1}), we see that $f(\beta)=\beta^6+a\beta+b=\frac{-bD}{5^6a^6}$. Therefore 
 \begin{equation}\label{eq F23 B}
 s_0=v_3(D)-6.
 \end{equation}
 Similarly in view of the fact that $f'(\beta)=6\beta^5+a=\frac{D}{5^5a^5}$, we have
 \begin{equation}\label{eqq}
 s_1=s_0+1=v_3(D)-5.
 \end{equation}
 Since $v_3(D)=8$ by hypothesis, it follows from (\ref{eq F23 B}) that  $s_0=2$. Therefore using the $(x-\beta)$-expansion of $f(x)$ given by (\ref{beta}), it can be easily seen that the $\phi_2$-Newton polygon of $f(x)$ has a single edge of positive slope, say $S_2$ joining the point (3, 0) with (6, 2) and the  polynomial associated to $f$ with respect to ($\phi_2,S_2)$ is linear; consequently $f(x)$ is 3-regular with respect to $\phi_1,~\phi_2$ having $ i_{\phi_{2}}(f)=1$. It now follows from Theorem 3.C that $v_3(ind~\theta)= i_{\phi_{1}}(f)+i_{\phi_{2}}(f)=1$. Applying Proposition 3.E with $\phi(x)$ replaced by $x-\beta$ and $j$ by $1$, we see that the element  $\frac{\theta^5+\beta\theta^4+\beta^2\theta^3+\beta^3\theta^2+\beta^4\theta+\beta^5+a}{3}$ belongs to $S_{(3)}$. Keeping in mind that $\beta\equiv -1~(mod~3)$ and $a\equiv 0~(mod~3)$, we conclude the element $\frac{\theta^5-\theta^4+\theta^3-\theta^2+\theta-1}{3}=\eta~(say)$ is  in $S_{(3)}$. Since $v_3(ind~\theta)=1$, it follows  that $\{1,\theta,\theta^2,\theta^3,\theta^4,\eta\}$ is 3-integral basis of $K$ in view of Proposition 3.A. 

In the present case consider now  the situation when $a\equiv 3~(mod~9)$. Since $-\theta$ is a root of $x^6-ax+b$, it follows from the above subcase that $v_3(ind~\theta)=1$ and $\{1,\theta,\theta^2,\theta^3,\theta^4,\eta'\}$ is a 3-integral basis of $K$ where $\eta'=\frac{\theta^5+\theta^4+\theta^3+\theta^2+\theta+1}{3}$.\\
  { \bf Case F22 :} $a\equiv \pm 3~(mod~9),~b\equiv 5~(mod~9),~v_3(D)=$  9.  Here $f(x)\equiv (x-1)^3 (x+1)^3~(mod~3)$.  Consider first the subcase when $a\equiv -3~(mod~9)$. Set $\beta=\frac{-6b}{5a}$, $\phi_1(x)=x-1$ and $\phi_2(x)=x-\beta$,  $s_0=v_3(f(\beta)),~s_1=v_3(f'(\beta))$. Note that $\beta\equiv -1 ~(mod~3)$. The $\phi_1$-Newton polygon of $f$ has a single edge of positive slope, say $S_1$ joining the points (3, 0) and (6, 1) with  $i_{\phi_1}(f)=0$. The polynomial associated to $f$ with respect to $(\phi_1,S_1)$ is linear. Arguing exactly as for the proof of (\ref{eq F23 B}) and (\ref{eqq}), one can check that  $s_1=s_0+1=v_3(D)-5=4$ in view of the hypothesis on $v_3(D)$. Therefore using the $(x-\beta)$-expansion of $f(x)$ given by (\ref{beta}), it can be easily seen that the the $\phi_2$-Newton polygon  $f(x)$ has a single edge of positive slope, say $S_2$ joining the points (3, 0) and (6, 3) with the lattice point (4, 1) lying on it. The  polynomial $T[Y]$ (say) associated to $f(x)$ with respect to ($\phi_2,S_2)$ over the field $\F_3$ is of the type $Y^3+\overline{c_1}Y^2+\overline{c_2}$ with $\overline{c_1}$$~\overline{c_2}\neq \overline{0}$ and hence $T[Y]$ no repeated roots. Consequently $f$ is 3-regular with respect to $\phi_1,~\phi_2$.   It now follows from Theorem 3.C that $v_3(ind~\theta)= i_{\phi_{1}}(f)+i_{\phi_{2}}(f)=3$. Applying Proposition 3.E with  $\phi(x)$ replaced by $x-\beta$ and $j$ by $1,~2$, we see that the elements 
  $\frac{\theta^5+\beta\theta^4+\beta^2\theta^3+\beta^3\theta^2+\beta^4\theta+\beta^5+a}{9},~\frac{\theta^4+2\beta\theta^3+3\beta^2\theta^2+4\beta^3\theta+5\beta^4}{3}$ are in $S_{(3)}$. Since $v_3(6\beta^5+a)=s_1=4$, we conclude that  the elements $\eta_5=\frac{\theta^5+\beta\theta^4+\beta^2\theta^3+\beta^3\theta^2+\beta^4\theta-5\beta^5}{9},~\eta_4=\frac{\theta^4-\beta\theta^3+\beta\theta-1}{3}$ are in $S_{(3)}$. Now choose an integer $x_1$ such that 
   $5(\frac{a}{3})x_1+2b\equiv 0~(mod~9)$. Note that $9$ divides $\beta-x_1$  in $\Z_{(3)}$ and $\beta\equiv -1~(mod~3)$. Therefore the elements $\eta_4',~\eta_5'$ defined by
   \begin{equation}\label{43} \eta_4'=\frac{\theta^4+\theta^3-\theta-1}{3},~\eta_5'=\frac{\theta^5+x_1\theta^4+x_1^2\theta^3+x_1^3\theta^2+x_1^4\theta-5x_1^5}{9}
   \end{equation} are in $S_{(3)}$ because $\eta_4-\eta_4',~\eta_5-\eta_5'$  belong to $\Z_{(3)}[\theta]\subseteq S_{(3)}$. Since  $v_3(ind~\theta)=3$, it follows that  $\{1,\theta,\theta^2,\theta^3,\eta_4',\eta_5'\}$ is a 3-integral basis of $K$ in view of Proposition 3.A.  
   
   In the present case consider now the subcase when $a\equiv 3~(mod~9)$. Since $-\theta$ is a root of $x^6-ax+b$, on replacing $\theta$ by $-\theta$ and $\beta$ by $-\beta$ in $\eta_4,~\eta_5$, we infer from the above subcase that $v_3(ind~\theta)=3$ and $\{1,\theta,\theta^2,\theta^3,\eta_4',\eta_5'\}$ is a 3-integral basis of $K$ where $\eta_4',\eta_5'$ are given by (\ref{43}) .\\
    { \bf Case F23 :} $a\equiv \pm 3~(mod~9),~b\equiv 5~(mod~9),~v_3(D)\geq 10$ and even.
     In this case $f(x)\equiv (x-1)^3 (x+1)^3~(mod~3)$. Consider first the subcase when $a\equiv -3 ~(mod~9)$. Set $\phi_1(x)=x-1$ and $\phi_2(x)=x-\beta$ with $\beta=\frac{-6b}{5a}$, $s_0=v_3(f(\beta)),~s_1=v_3(f'(\beta))$. It can be easily seen that the $\phi_1$-Newton polygon of $f$ is the same as in case F22 with  $i_{\phi_1}(f)=0$. Arguing exactly as for the proof of (\ref{eq F23 B}) and (\ref{eqq}), one can check that $ s_1=s_0+1=v_3(D)-5.$  So   $s_0$ is even and $s_0\geq 4$ in view of the hypothesis on $v_3(D)$.  Therefore using the $(x-\beta)$-expansion of $f(x)$ given by (\ref{beta}), it can be easily seen that the the $\phi_2$-Newton polygon  $f(x)$ has $2$ edges.  The first edge, say $S_1$  is the line segment joining the point (3, 0) with (4, 1) and the second edge, say $S_2$ is the line segment joining the point (4, 1) with (6, $s_0$). The slopes of these edges are $1$ and $\frac{s_0-1}{2}$ respectively. The polynomial associated to $f$ with respect to $(\phi_2,S_1)$ is clearly linear and the one associated to $f$ with respect to $(\phi_2,S_2)$ is linear as $s_0$ is even.  Therefore $f(x)$ is 3-regular with respect to $\phi_1,~\phi_2$. So using  Lemma 3.F and the equality $s_0=v_3(D)-6$, we have $i_{\phi_2}(f)=\frac{s_0+2}{2}=\frac{v_3(D)-4}{2}$; consequently by Theorem 3.C,  $v_3(ind~\theta)= i_{\phi_{1}}(f)+i_{\phi_{2}}(f)=\frac{v_3(D)-4}{2}$.  Applying Proposition 3.E taking  $\phi(x)=x-\beta$ and $j=1,~2$, we see that the elements  $\frac{\theta^5+\beta\theta^4+\beta^2\theta^3+\beta^3\theta^2+\beta^4\theta+\beta^5+a}{3^{k_0}},~\frac{\theta^4+2\beta\theta^3+3\beta^2\theta^2+4\beta^3\theta+5\beta^4}{3}$ are in $S_{(3)}$,  where $k_0=\frac{v_3(D)-6}{2}$. Since $v_3(6\beta^5+a)=s_1=v_3(D)-5>k_0$, it follows that the elements $\eta_5=\frac{\theta^5+\beta\theta^4+\beta^2\theta^3+\beta^3\theta^2+\beta^4\theta-5\beta^5}{3^{k_0}},~\eta_4=\frac{\theta^4-\beta\theta^3+\beta\theta-1}{3}$ are in $S_{(3)}$.  Choose an integer $x_2$ such that 
        $5(\frac{a}{3})x_2+2b\equiv 0~(mod~3^{k_0})$. Note that $3^{k_0}$ divides $\beta-x_2$  in $\Z_{(3)}$  and $\beta\equiv -1~(mod~3)$. So the elements $\eta_4',~\eta_5'$ defined by 
        \begin{equation}\label{44}
         \eta_4'=\frac{\theta^4+\theta^3-\theta-1}{3},~\eta_5'=\frac{\theta^5+x_2\theta^4+x_2^2\theta^3+x_2^3\theta^2+x_2^4\theta-5x_2^5}{3^{k_0}}
         \end{equation} are in $S_{(3)}$. Since  $v_3(ind~\theta)=\frac{v_3(D)-4}{2}$, it follows that  $\{1,\theta,\theta^2,\theta^3,\eta_4',\eta_5'\}$ is a 3-integral basis of $K$ in view of Proposition 3.A. 
         
         Consider now the subcase when $a\equiv 3~(mod~9)$. Keeping in view that $-\theta$ is a root of $x^6-ax+b$, we infer from the above subcase that $v_3(ind~\theta)=\frac{v_3(D)-4}{2}$ and $\{1,\theta,\theta^2,\theta^3,\eta_4',\eta_5'\}$ is a 3-integral basis of $K$ where  $\eta_4',\eta_5'$ are given by (\ref{44}) .\\
    { \bf Case F24 :} $a\equiv \pm 3~(mod~9),~b\equiv 5~(mod~9),~v_3(D)\geq 11$ and odd.
        Here $f(x)\equiv (x-1)^3 (x+1)^3~(mod~3)$. Consider first the subcase when $a\equiv - 3~(mod~9)$. Set $\phi_1(x)=x-1$ and $\phi_2(x)=x-\beta$ with $\beta=\frac{-6b}{5a}$, $s_0=v_3(f(\beta)),~s_1=v_3(f'(\beta))$. It can be shown that the $\phi_1$-Newton polygon of $f$ is the same as in case F22. Arguing exactly as for the proof of (\ref{eq F23 B}) and (\ref{eqq}), one can check that  $ s_1=s_0+1=v_3(D)-5$. So $s_0$ is odd and $s_0\geq 5$ in view of the hypothesis. Therefore using the $(x-\beta)$-expansion of $f(x)$ given by (\ref{beta}), it can be easily seen that the the $\phi_2$-Newton polygon  $f(x)$ has $2$ edges.  The first edge, say $S_1$  is the line segment joining the point (3, 0) with (4, 1) and the second edge, say $S_2$ is the line segment joining the point (4, 1) with (6, $s_0$). The polynomial associated to $f$ with respect to $(\phi_2,S_1)$ is linear  and the one associated to $f$ with respect to $(\phi_2,S_2)$ is a quadratic polynomial of the type $Y^2+\overline{c}$  because $s_0$ is odd. Therefore $f(x)$ is 3-regular with respect to $\phi_1,~\phi_2$. So keeping in mind  Theorem 3.C , Lemma 3.F and the equality $s_0=v_3(D)-6$, we see that $v_3(ind~\theta)= i_{\phi_{1}}(f)+i_{\phi_{2}}(f)=\frac{v_3(D)-3}{2}$.  Applying Proposition 3.E with  $\phi(x)=x-\beta$ and $j=1,~2$, we see that the elements  $\frac{\theta^5+\beta\theta^4+\beta^2\theta^3+\beta^3\theta^2+\beta^4\theta+\beta^5+a}{3^{k_1}}$,   $\frac{\theta^4+2\beta\theta^3+3\beta^2\theta^2+4\beta^3\theta+5\beta^4}{3}$  are in $S_{(3)}$, where $k_1=\frac{v_3(D)-5}{2}$.  Since $v_3(6\beta^5+a)=s_1>k_1$, it follows that the elements $\eta_5=\frac{\theta^5+\beta\theta^4+\beta^2\theta^3+\beta^3\theta^2+\beta^4\theta-5\beta^5}{3^{k_1}}$  and $\eta_4=\frac{\theta^4-\beta\theta^3+\beta\theta-1}{3}$ are in $S_{(3)}$.  Now choose an integer $x_3$ such that 
           $5(\frac{a}{3})x_3+2b\equiv 0~(mod~3^{k_1})$. Note that $3^{k_1}$ divides $\beta-x_3$  in $\Z_{(3)}$ and $\beta\equiv -1~(mod~3)$. Therefore the elements
           \begin{equation}\label{45}
            \eta_4'=\frac{\theta^4+\theta^3-\theta-1}{3},~\eta_5'=\frac{\theta^5+x_3\theta^4+x_3^2\theta^3+x_3^3\theta^2+x_3^4\theta-5x_3^5}{3^{k_1}}
            \end{equation} are $3$-integral. Since  $v_3(ind~\theta)=\frac{v_3(D)-3}{2}$, it follows that  $\{1,\theta,\theta^2,\theta^3,\eta_4',\eta_5'\}$ is a 3-integral basis of $K$ in view of Proposition 3.A. 
            
             Consider now the subcase when $a\equiv 3~(mod~9)$. Since $-\theta$ is a root of $x^6-ax+b$, we conclude in view of the above subcase that $v_3(ind~\theta)=\frac{v_3(D)-3}{2}$ and $\{1,\theta,\theta^2,\theta^3,\eta_4',\eta_5'\}$ is a 3-integral basis of $K$ where $\eta_4',\eta_5'$ are given by (\ref{45}). This completes the proof in case F24.\\

  In Table II, it remains to deal with the cases when $v_3(b)=3,~v_3(a)\geq 3$. In this situation the $x$-Newton polygon of $f(x)$ is a line segment  joining the points (0, 0) and (6, 3).  In what follows, $\xi$, $B$ will stand respectively for $\frac{\theta^2}{3}$, $\frac{b}{3^3}$ and $V_3$ for any prolongation of $v_3$ to $K$. Note that   $V_3(\theta)=\frac{1}{2}$ in view of \cite[Chapter 2, Proposition 6.3]{Neu} and 
 
 \begin{equation}\label{F31}
 V_3(\xi^3+B)= V_3\big(\frac{\theta^6+b}{3^3}\big)=V_3\big(\frac{-a\theta}{3^3}\big)=v_3(a)-3+\frac{1}{2} .
  \end{equation}

   With the above notations, we  take up the remaining cases of Table II.\\
  { \bf Case F25 :} $v_3(a)=v_3(b)=3.$ It is immediate from (\ref{eq:p1.1}) that $v_3(D)=18$.  Applying Proposition 3.E taking $\phi(x)=x$ and $j=1,~2,~3,~4$, we see that the set $\mathcal B:=\{1,\theta,\frac{\theta^2}{3},\frac{\theta^3}{3},\frac{\theta^4}{3^2},\frac{\theta^5}{3^2}\}$ is contained in $S_{(3)}$. So $v_3(ind~\theta)\geq 6$ which implies that  $v_3(d_K)=v_3(D)-2v_3(ind~\theta)\leq 6$. In this case we will prove that $v_3(d_K)= 6$ and hence $\mathcal B$ will be a 3-integral basis of $K$. In view of the hypothesis, (\ref{F31}) gives $V_3(\xi^3+B)=\frac{1}{2}$. Also by little Fermat's Theorem $V_3(B^3-B)=v_3(B^3-B)\geq 1$. So by the strong triangle law $$V_3(\xi^3+B^3)=min~\{V_3(\xi^3+B),V_3(B^3-B)\}=\frac{1}{2}$$ i.e., $V_3((\xi+B)^3-3\xi B(\xi+B))=\frac{1}{2}$. It now follows in view of the strong triangle law that $V_3((\xi+B)^3)=\frac{1}{2}$ and hence $V_3(\xi+B)=\frac{1}{6}$. So $3 $ is totally ramified in $K$; consequently $v_3(d_K)\geq 6$ in view of a basic result (cf.\cite[Theorem 4.24] {Nar}).\\
  { \bf Case F26 :} $v_3(a)\geq 4,~v_3(b)=3,$ $v_3(B^3-B)=1$.   Using  (\ref{F31}), hypothesis and the strong triangle law, we have $V_3(\xi^3+B^3)=min~\{V_3(\xi^3+B),V_3(B^3-B)\}=1.$  In view of the strong triangle law, the equality $V_3((\xi+B)^3-3\xi B(\xi+B))=1$  shows that $V_3((\xi+B)^3)=1.$ Since the $x$-Newton polygon of $f(x)$ has a single edge with slope $\frac{1}{2}$, we have $V_3(\theta)=\frac{1}{2}$. So the value group of $V_3$ is $\frac{1}{6}\Z$. Consequently $v_3(d_K)\geq 6 $. By virtue of (\ref{eq:p1.1}), we have $v_3(D)=21$. Therefore $v_3(d_K)$ is odd and it exceeds 6. Since for any prolongation $V_3$ of $v_3$ to $K$, we have $V_3((\xi+B)^2)\theta)=\frac{2}{3}+\frac{1}{2}>1$. So $\frac{(\xi+B)^2\theta}{3}=$$\frac{(\theta^2+3B)^2\theta}{3^3}$  belongs to $S_{(3)}$ in view of \cite[Chapter 3, Section 4, Theorem 6]{Bor}. Applying Proposition 3.E with $\phi(x)=x$ and $j=4$, we see that $\frac{\theta^2}{3}\in S_{(3)}$ and consequently $\frac{\theta^3}{3},~\frac{\theta^4}{3^2}\in S_{(3)}.$ Therefore the set $\mathcal B:=\{1,\theta,\frac{\theta^2}{3},\frac{\theta^3}{3},\frac{\theta^4}{3^2},\frac{(\theta^2+3B)^2\theta}{3^3}\}$ is contained in $S_{(3)}$. Hence $v_3(ind~\theta)\geq 7$ and $v_3(d_K)\leq 7$. It now follows that  $v_3(d_K)=7$ and $\mathcal B$ is a 3-integral basis of $K$.\\
  { \bf Case F27 :} $v_3(a)\geq 4,~v_3(b)=3,$ $v_3(B^3-B)\geq 2$. Keeping in mind (\ref{F31}) and the hypothesis, we see that for any prolongation $V_3$ of $v_3$ to $K$,  $V_3(\xi^3+B^3)\geq \frac{3}{2}$ i.e., 
  \begin{equation}\label{F33}
  V_3((\xi+B)^3-3\xi B(\xi+B))\geq \frac{3}{2}.
  \end{equation} With $V_3$ as above, we first show that  $V_3(\xi+B)\geq \frac{1}{2}.$ Note that when  $V_3((\xi+B)^3)=V_3(3\xi B(\xi+B))$, then clearly $V_3(\xi+B)=\frac{1}{2}$ and when $V_3((\xi+B)^3)\neq V_3(3\xi B(\xi+B))$ then in view of (\ref{F33}), $min~\{(\xi+B)^3,1+V_3(\xi+B)\}\geq \frac{3}{2}$ which implies that $V_3(\xi+B)\geq \frac{1}{2}$. As $V_3(\theta)=\frac{1}{2}$  and $V_3(\xi+B)\geq\frac{1}{2}$, it follows that $\frac{(\xi+B)\theta}{3}=\frac{(\theta^2+3B)\theta}{3^2}\in S_{(3)}$ and that $\frac{(\theta^2+3B)^2}{3^3}=\frac{(\xi+B)^2}{3}\in S_{(3)}$; consequently $\frac{(\theta^2+3B)^2\theta}{3^3}$ belongs to $S_{(3)}$. By Proposition 3.E applied to $\phi(x)=x$ and $j=4$, we see that $\frac{\theta^2}{3}$ belongs to $S_{(3)}$. Thus we have shown that the set $\mathcal B:=\{1,\theta,\frac{\theta^2}{3},\frac{(\theta^2+3B)\theta}{3^2},\frac{(\theta^2+3B)^2}{3^3},\frac{(\theta^2+3B)^2\theta}{3^3}\}$ is contained in $S_{(3)}$ which implies $v_3(ind~\theta)\geq 9$ and hence $v_3(d_K)\leq 3$. Keeping in mind that the index of ramification of each prime ideal of $A_K$ lying over $3$ is divisible by $2$ and taking into consideration the factorisation\footnote{In fact the factorisation of $3A_K$ will be either $\wp_0^2$ or $\wp^2\wp_1^2$ or $\wp_1^2\wp_2^2\wp_3^2$ or $\wp_1^2\wp_2^4$ or $\wp_1^6$ where  $N_{K/\mathbb{Q}}(\wp_0)=3^3$, $N_{K/\mathbb {Q}}(\wp_i)=3$ for each $i\geq 1$ and $N_{K/\mathbb{Q}}(\wp)=3^2$.} of $3A_K$ into product of prime ideals of $A_K$, it can be easily verified that  $v_3(d_K)\geq 3$ . Therefore $v_3(d_K)=3,~v_3(ind~\theta)=9$ and $\mathcal B$ is a 3-integral basis of $K$. This completes the proof of Table II.\\
  
    In the proof of Table III, we shall use the $(x+a)$-expansion of $f(x)$ given by 
     \begin{equation}\label{G2 1}
        (x+a)^6-6a(x+a)^5+15a^2(x+a)^4-20a^3(x+a)^3+15a^4(x+a)^2+a(1-6a^4)(x+a)+(b+a^6-a^2)
        \end{equation} and shall denote  $v_5(b+a^6-a^2),~v_5(a(1-6a^4))$ by $r_0,~r_1$ respectively. Note that when  $v_5(b)
            =1,~v_5(a)=0,$ then $r_0=1$ if and only if $b\not\equiv a^2-a^6~(mod~25)$ and $r_1=1$ if and only if $a^4\not\equiv 21~(mod~25)$. With this fact in mind, we take up the various cases of Table III. \\
    { \bf Case G1 :} $v_5(b)
    =0$. In this case, we have $v_5(D)=0$ and hence $\{1,\theta,\theta^2,\theta^3,\theta^4,\theta^5\}$ is a 5-integral basis of $K$.\\
  { \bf Case G2 :} $v_5(b)
          =1,~v_5(a)=0,~r_0=1=r_1$, $a^2\not\equiv \frac{b}{5}~(mod~5)$. Keeping in mind (\ref{eq:p1.1}) and the hypothesis $a^2\not\equiv \frac{b}{5}~(mod~5)$, it can be easily verified  that $v_5(D)=5$. Here $f(x)\equiv x(x+a)^5~(mod~5)$. Set $\phi_1(x)=x$ and $\phi_2(x)=x+a$. The $\phi_1$-Newton polygon of $f$ has a single edge of positive slope joining the points (5, 0) and (6, 1) with $i_{\phi_1}(f)=0$.  Using (\ref{G2 1}), it can be easily seen that the $\phi_2$-Newton polygon of $f$ has a single edge of positive slope, say $S_1$ joining the points (1, 0) and (6, 1). Note that the polynomial associated to $f$ with respect to $(\phi_2,S_1)$ is linear and $i_{\phi_2}(f)=0$. Using Theorem 3.C, we have $v_5(ind~\theta)=i_{\phi_1}(f)+i_{\phi_2}(f)=0$ and hence $\{1,\theta,\theta^2,\theta^3,\theta^4,\theta^5\}$ is a 5-integral basis of $K$.\\
   { \bf Case G3 :} $v_5(b)
           =1,~v_5(a)=0,~r_0=1=r_1$, $a^2\equiv \frac{b}{5}~(mod~5)$. Using (\ref{eq:p1.1}) and  the hypothesis, a simple calculation gives $v_5(D)=6$.   Set $\phi_1(x)=x$ and $\phi_2(x)=x+a$. One can check that the $\phi_1$-Newton polygon and the $\phi_2$-Newton polygon of $f$ are the same as in case G2.  Arguing as in case G2, we see that $v_5(ind~\theta)=0$ and $\{1,\theta,\theta^2,\theta^3,\theta^4,\theta^5\}$ is a 5-integral basis of $K$.\\ 
    { \bf Case G4 :} $v_5(b) =1,~v_5(a)=0,~r_0\geq 2,~r_1=1$. By virtue of (\ref{eq:p1.1}) and  the hypothesis, it can be easily seen that $v_5(D)=5$ in this case.  Set $\phi_1(x)=x$ and $\phi_2(x)=x+a$. The $\phi_1$-Newton polygon of $f$ has a single edge of positive slope joining the points (5, 0) and (6, 1) with $i_{\phi_1}(f)=0$. Using (\ref{G2 1}), it can be easily seen that the $\phi_2$-Newton polygon of $f(x)$ has two  edges  of positive slope. The first edge, say $S_1$  is the line segment joining the point (1, 0) with (5, 1) and the second edge, say $S_2$ is the line segment joining the point (5, 1) with (6, $r_0$) with $i_{\phi_2}(f)=1$. The  polynomial associated to $f(x)$ with respect to $(\phi_2,S_i)$ is linear for $i=1,~2$. In view of Theorem 3.C, we have $v_5(ind~\theta)=i_{\phi_1}(f)+i_{\phi_2}(f)=1$.  Applying Proposition 3.E with $\phi(x)$ replaced by $\phi_2(x)$ and $j$ by $1$, we see that $\frac{\theta^5-a\theta^4+a^2\theta^3-a^3\theta^2+a^4\theta-a^5+a}{5}$ belongs to $S_{(5)}$ and hence $\frac{\theta^5-a\theta^4+a^2\theta^3-a^3\theta^2+\theta}{5}=\eta~(say)$ belongs to $S_{(5)}$. Since $v_5(ind~\theta)= 1$,  it now follows from Proposition 3.A that $\{1,\theta,\theta^2,\theta^3,\theta^4,\eta\}$ is a 5-integral basis of $K$.\\
      { \bf Case G5 :} $v_5(b)
        =1,v_5(a)=0,~r_0=1,~r_1\geq 2$. Keeping in mind the hypothesis and (\ref{eq:p1.1}), a simple calculation shows that $v_5(D)=5$. Here $f(x)\equiv x(x+a)^5~(mod~5)$. Set $\phi_1(x)=x$ and $\phi_2(x)=x+a$. One can check that the $\phi_i$-Newton polygon of $f$ is the same as in case G2 for $i=1,2$. Proceding as in case G2, we see that $v_5(ind~\theta)=0$ and  $\{1,\theta,\theta^2,\theta^3,\theta^4,\theta^5\}$ is a 5-integral basis of $K$.\\
   { \bf Case G6 :} $v_5(b)=1,~v_5(a)=0,~r_0\geq 2,~r_1\geq 2,v_5(D)$ is odd. By virtue of (\ref{eq:p1.1}) and the hypothesis, we have $v_5(D)\geq 7$. Recall that $f(x)\equiv x(x+a)^5~(mod~5)$. Set $\phi_1(x)=x$. The $\phi_1$-Newton polygon of $f$ has a single edge of positive slope joining the points (5, 0) and (6, 1) with $i_{\phi_1}(f)=0$. Set $\beta=\frac{-6b}{5a}$ and $\phi_2(x)=x-\beta$. Keeping in view the hypothesis, one can check that $\beta\equiv -a ~(mod~5)$. Denote $v_5(f(\beta)),~v_5(f'(\beta))$ by $s_0,~s_1$ respectively. On substituting for $\beta$ and keeping in mind (\ref{eq:p1.1}), we see that $f(\beta)=\frac{-bD}{5^6a^6}$. Therefore 
   \begin{equation}\label{eq G6 1}
    s_0=v_5(D)-5.
   \end{equation}
   Similarly it can be seen that  $s_1=v_5(6\beta^5+a)=v_5(\frac{D}{5^5a^5})=v_5(D)-5$. Therefore  \begin{equation}\label{eqq G}
                                       s_1=s_0=v_5(D)-5.
    \end{equation}
      By virtue of (\ref{eq G6 1}) and the hypothesis,  the number $s_0$ is even and positive.  So using the $(x-\beta)$-expansion of $f(x)$ given by (\ref{beta}), it can be easily seen that the $\phi_2$-Newton polygon of $f(x)$ has  two  edges  of positive slope. The first edge, say $S_1$  is the line segment joining the point (1, 0) with (4, 1) and the second edge, say $S_2$ is the line segment joining the point (4, 1) with (6, $s_0$). The slopes of these edges are $\frac{1}{3}$ and $\frac{s_0-1}{2}$ respectively. The polynomial associated to $f(x)$ with respect to $(\phi_2,S_1)$ is clearly linear and the one associated to $f$ with respect to $(\phi_2,S_2)$ is linear because $s_0$ is even. Therefore $f$ is 5-regular with respect to $\phi_1,~\phi_2$. Keeping in mind (\ref{eq G6 1}), it now follows from Lemma 3.F that $i_{\phi_2}(f)=\frac{s_0+2}{2}=\frac{v_5(D)-3}{2}$ and consequently by Theorem 3.C, $v_5(ind~\theta)=i_{\phi_1
      }(f)+i_{\phi_2}(f)=\frac{v_5(D)-3}{2}$. Set $k_1=\frac{v_5(D)-5}{2}$.  Applying Proposition 3.E with $\phi(x)$ replaced  by  $\phi_2(x)$  and $j$ by $1,~2$, we see that the elements $\frac{\theta^5+\beta\theta^4+\beta^2\theta^3+\beta^3\theta^2+\beta^4\theta+\beta^5+a}{5^{k_1}}$  and $\frac{\theta^4+2\beta\theta^3+3\beta^2\theta^2+4\beta^3\theta+5\beta^4}{5}=\eta_0~(say)$ are in $S_{(5)}$. Since $v_5(6\beta^5+a)=s_1=v_5(D)-5>k_1$ in view of (\ref{eqq G}), it follows that the element $\eta_1=\frac{\theta^5+\beta\theta^4+\beta^2\theta^3+\beta^3\theta^2+\beta^4\theta-5\beta^5}{5^{k_1}}$ is in $S_{(5)}$. Choose an integer $x_0$ such that  ${a}x_0+6(\frac{b}{5})\equiv 0~(mod~5^{k_1})$. Keeping in mind that $5^{k_1}$ divides $\beta-x_0$  in $\Z_{(5)}$ and $\beta\equiv -a~(mod~5)$, we see that the elements $\eta_0',~\eta_1'$ defined by 
       $$\eta_0'=\frac{\theta^4-2a\theta^3+3a^2\theta^2-4a^3\theta}{5},~\eta_1'=\frac{\theta^5+x_0\theta^4+x_0^2\theta^3+x_0^3\theta^2+x_0^4\theta-5x_0^5}{5^{k_1}}$$ are in $S_{(5)}$ because $\eta_0-\eta_0',~\eta_1-\eta_1'$  belong to $\Z_{(5)}[\theta]\subseteq S_{(5)}$. Since  $v_5(ind~\theta)=\frac{v_5(D)-3}{2}$,we conclude that  $\{1,\theta,\theta^2,\theta^3,\eta_0',\eta_1'\}$ is a 5-integral basis of $K$ in view of Proposition 3.A.\\
  { \bf Case G7 :} $v_5(b)=1,~v_5(a)=0,~r_0\geq 2,~r_1\geq 2,~v_5(D)$ is even. It follows from (\ref{eq:p1.1}) and the hypothesis that  $v_5(D)\geq 8$.   Set $\phi_1(x)=x$, $\phi_2(x)=x-\beta$ with $\beta=\frac{-6b}{5a}$, $s_0=v_5(f(\beta)),~s_1=v_5(f'(\beta))$. The $\phi_1$-Newton polygon of $f$ has a single edge of positive slope joining the points (5, 0) and (6, 1) with $i_{\phi_1}(f)=0$. Arguing exactly as for the proof of (\ref{eq G6 1}) and (\ref{eqq G}), one can check that $s_0=s_1=v_5(D)-5.$ So $s_0$ is odd  and $s_0\geq 3$ in view of the hypothesis.  Therefore using the $(x-\beta)$-expansion of $f(x)$ given by (\ref{beta}), it can be easily seen that the the $\phi_2$-Newton polygon  $f(x)$ has  two  edges  of positive slope. The first edge, say $S_1$  is the line segment joining the point (1, 0) with (4, 1) and the second edge, say $S_2$ is the line segment joining the point (4, 1) with (6, $s_0$). The slopes of these edges are $\frac{1}{3}$ and $\frac{s_0-1}{2}$ respectively.  The polynomial associated to $f(x)$ with respect to $(\phi_2,S_1)$ is linear and the one associated to $f$ with respect to $(\phi_2, S_2)$ is of the type $Y^2+\overline{c}$ because $s_0$ is odd. So $f$ is 5-regular with respect to $\phi_1,~\phi_2$. Keeping in mind that $s_0=v_5(D)-5$, it now follows from Lemma 3.F that $i_{\phi_2}(f)=\frac{s_0+3}{2}=\frac{v_5(D)-2}{2}$ and consequently by  virtue of Theorem 3.C, $v_5(ind~\theta)= i_{\phi_1}(f)+i_{\phi_2}(f)=\frac{v_5(D)-2}{2}$. Set $k_2=\frac{v_5(D)-4}{2}$. Applying Proposition 3.E taking $\phi(x)=x-\beta$ and  $j=1,~2$, we see that the elements $\frac{\theta^5+\beta\theta^4+\beta^2\theta^3+\beta^3\theta^2+\beta^4\theta+\beta^5+a}{5^{k_2}}$ and $\frac{\theta^4+2\beta\theta^3+3\beta^2\theta^2+4\beta^3\theta+5\beta^4}{5}=\eta_0~(say)$ are in $S_{(5)}$. Recall that $v_5(6\beta^5+a)=s_1=v_5(D)-5>k_2$. So the element $\eta_1=\frac{\theta^5+\beta\theta^4+\beta^2\theta^3+\beta^3\theta^2+\beta^4\theta-5\beta^5}{5^{k_2}}$ is in $S_{(5)}$. Choose an integer $x_1$ such that  ${a}x_1+6(\frac{b}{5})\equiv 0~(mod~5^{k_2})$.  Then $5^{k_2}$ divides $\beta-x_1$  in $\Z_{(5)}$. Note that $\beta\equiv-a~(mod~5)$ by virtue of the hypothesis. Therefore the elements $\eta_0',~\eta_1'$ defined by $$\eta_0'=\frac{\theta^4-2a\theta^3+3a^2\theta^2-4a^3\theta}{5},~\eta_1'=\frac{\theta^5+x_1\theta^4+x_1^2\theta^3+x_1^3\theta^2+x_1^4\theta-5x_1^5}{5^{k_2}}$$ are in $ S_{(5)}$. Since  $v_5(ind~\theta)=\frac{v_5(D)-2}{2}$, it follows that  $\{1,\theta,\theta^2,\theta^3,\eta_0',\eta_1'\}$ is a 5-integral basis of $K$ in view of Proposition 3.A.\\
  { \bf Case G8 :} $v_5(b)=1,~v_5(a)\geq 1$. In view of (\ref{eq:p1.1}), we have $v_5(D)=5$.  The $x$-Newton polygon of $f(x)$ has a single edge joining the points (0, 0) and (6, 1). Arguing as in case E2, we see that $v_5(ind~\theta)=0$ and $\{1,\theta,\theta^2,\theta^3,\theta^4,\theta^5\}$ is a 5-integral basis of $K$.\\
  { \bf Case G9 :} $v_5(b)\geq 2,~v_5(a)=0,~a^4\not\equiv 1~(mod~25)$. In this case  $v_5(D)=5$, $r_0=1$ and $r_1\geq 1$. Set $\phi_1(x)=x$ and $\phi_2(x)=x+a$.  It can be easily seen that the $\phi_1$-Newton polygon and the $\phi_2$-Newton polygon of $f$ are the same as in case G2. Arguing as in case G2, we see that $v_5(ind~\theta)=0$ and $\{1,\theta,\theta^2,\theta^3,\theta^4,\theta^5\}$ is a 5-integral basis of $K$.\\ 
  { \bf Case G10 :} $v_5(b)\geq 2,~v_5(a)=0,~a^4\equiv 1~(mod~25)$. In this case $v_5(D)=5$, $r_0\geq 2$ and  $r_1=1$. Set $\phi_1(x)=x$ and $\phi_2(x)=x+a$. One can check that the $\phi_i$-Newton polygon of $f$ is the same as in case G4 for $i=1,~2$. Proceding exactly as in case G4, we see that $v_5(ind~\theta)= 1$ and $\{1,\theta,\theta^2,\theta^3,\theta^4,\frac{\theta^5-a\theta^4+a^2\theta^3-a^3\theta^2+\theta}{5}\}$ is a 5-integral basis of $K$.\\ 
   { \bf Case G11 :} $v_5(b)=2, ~v_5(a)= 1$. One can check that $v_5(D)=10.$ The $x$-Newton polygon of $f(x)$ has two edges . The first edge  is the line segment joining the point (0, 0) with (5, 1) and the second edge is the line segment joining the point (5, 1) with (6, $2)$. Proceeding as in case E4, we see that $v_5(ind~\theta)=1$ and $\{1, \theta,\theta^2,\theta^3,\theta^4,\frac{\theta^5}{5}\}$ is a 5-integral basis of $K$.\\
     { \bf Case G12 :} $v_5(b)=2,~v_5(a)\geq 2$. In view of (\ref{eq:p1.1}), we have $v_5(D)=10$. The $x$-Newton polygon of $f(x)$ has a single edge, say $S$ joining the points (0, 0) and (6, 2). The polynomial associated with $f(x)$ with respect to $(x,S)$ is $Y^2+\overline{c}$ where $\overline{c}\neq \overline{0}$. Arguing exactly as in case F5, we see that $v_5(ind~\theta)=3$ and $\{1, \theta,\theta^2,\frac{\theta^3}{5},\frac{\theta^4}{5},\frac{\theta^5}{5}\}$ is a 5-integral basis of $K$.\\
      { \bf Case G13 :} $v_5(b)\geq 3, ~v_5(a)= 1$. By virtue of (\ref{eq:p1.1}), we have $v_5(D)=11.$. The $x$-Newton polygon of $f(x)$ has two edges . The first edge  is the line segment joining the point (0, 0) with (5, 1) and the second edge is the line segment joining the point (5, 1) with (6, $v_5(b))$. Arguing as in Case E4, we see that $v_5(ind~\theta)=1$ and $\{1, \theta,\theta^2,\theta^3,\theta^4,\frac{\theta^5}{5}\}$ is a 5-integral basis of $K$.\\
      { \bf Case G14 :} $v_5(b)= 3,~v_5(a)= 2$. In view of (\ref{eq:p1.1}), we have $v_5(D)=15$.  The $x$-Newton polygon of $f(x)$ has two edges . The first edge is the line segment joining the point (0, 0) with  (5, 2) and the second edge is the line segment joining the point (5, 2) with  (6, 3). Arguing as in case E5, we see that $v_5(ind~\theta)=4$ and $\{1, \theta,\theta^2,\frac{\theta^3}{5},\frac{\theta^4}{5},\frac{\theta^5}{5^2}\}$ is a 5-integral basis of $K$. \\
         { \bf Case G15 :} $v_5(b)= 3, ~v_5(a)\geq 3$. By virtue of (\ref{eq:p1.1}), we have $v_5(D)=15.$  The $x$-Newton polygon of $f(x)$ has a single edge, say $S$ joining the points (0, 0) and (6, 3) with slope $\frac{1}{2}$. The polynomial associated over the field $\F_5$ to $f(x)$ with respect to ($x,S)$ is $Y^3+\overline{c}$ where  $\overline{c}\neq \overline{0}$. Arguing as in case E6, we see that $v_5(ind~\theta)=6$ and $\{1, \theta,\frac{\theta^2}{5},\frac{\theta^3}{5},\frac{\theta^4}{5^2},\frac{\theta^5}{5^2}\}$ is a 5-integral basis of $K$. \\
     { \bf Case G16 :} $v_5(b)\geq 4,~v_5(a)= 2$. In view of (\ref{eq:p1.1}), we have $v_5(D)=17$.  The $x$-Newton polygon of $f(x)$ has two edges . The first edge  is the line segment joining the point (0, 0) with  (5, 2) and the second edge is the line segment joining the point (5, 2) with  (6, $v_5(b)$). Arguing as in case E5, we see that $v_5(ind~\theta)=4$ and $\{1, \theta,\theta^2,\frac{\theta^3}{5},\frac{\theta^4}{5},\frac{\theta^5}{5^2}\}$ is a 5-integral basis of $K$. \\
      { \bf Case G17 :} $v_5(b)= 4, ~v_5(a)= 3$. By virtue of (\ref{eq:p1.1}), we have $v_5(D)=20.$   The $x$-Newton polygon of $f(x)$ has two edges . The first edge  is the line segment joining the point (0, 0) with  (5, 3) and the second edge  is the line segment joining the point (5, 3) with  (6, 4). Proceeding as in case E8, we see that $v_5(ind~\theta)=7$ and $\{1, \theta,\frac{\theta^2}{5},\frac{\theta^3}{5},\frac{\theta^4}{5^2},\frac{\theta^5}{5^3}\}$ is a 5-integral basis of $K$. \\
       { \bf Case G18 :} $v_5(b)= 4,~v_5(a)\geq 4$. In view of (\ref{eq:p1.1}), we have $v_5(D)=20$. The $x$-Newton polygon of $f(x)$ has a single edge, say $S$ joining the points (0, 0) and (6, 4). The polynomial associated  to $f(x)$ with respect to ($x,S)$ is  $Y^2+\overline{c}$ where  $\overline{c}\neq \overline{0}$. Arguing exactly as in case F5,  we see that $v_5(ind~\theta)=8$ and $\{1, \theta,\frac{\theta^2}{5},\frac{\theta^3}{5^2},\frac{\theta^4}{5^2},\frac{\theta^5}{5^3}\}$ is a 5-integral basis of $K$.\\
       { \bf Case G19 :} $v_5(b)\geq 5, ~v_5(a)= 3$. By virtue of (\ref{eq:p1.1}), we have $v_5(D)=23.$   The $x$-Newton polygon of $f(x)$ has two edges . The first edge  is the line segment joining the point (0, 0) with  (5, 3) and the second edge is the line segment joining the point (5, 3) with  (6, $v_5(b)$).  Proceeding as in case E8, we see that $v_5(ind~\theta)=7$ and $\{1, \theta,\frac{\theta^2}{5},\frac{\theta^3}{5},\frac{\theta^4}{5^2},\frac{\theta^5}{5^3}\}$ is a 5-integral basis of $K$. \\
        { \bf Case G20 :} $v_5(b)= 5,~v_5(a)= 4$. In view of (\ref{eq:p1.1}), we have $v_5(D)=25$.  The $x$-Newton polygon of $f(x)$ has two edges . The first edge is the line segment joining the point (0, 0) with  (5, 4) and the second edge is the line segment joining the point (5, 4) with  (6, 5).  Arguing as in case E8, we see that $v_5(ind~\theta)=10$ and $\{1, \theta,\frac{\theta^2}{5},\frac{\theta^3}{5^2},\frac{\theta^4}{5^3},\frac{\theta^5}{5^4}\}$ is a 5-integral basis of $K$. \\
        { \bf Case G21 :} $v_5(b)=5, ~v_5(a)\geq 5$. By virtue of (\ref{eq:p1.1}), we have $v_5(D)=25.$  The $x$-Newton polygon of $f(x)$ has a single edge joining the points (0, 0) and (6, 5). Proceeding as in case E10, we see that $v_5(ind~\theta)=10$ and $\{1, \theta,\frac{\theta^2}{5},\frac{\theta^3}{5^2},\frac{\theta^4}{5^3},\frac{\theta^5}{5^4}\}$ is a 5-integral basis of $K$. \\
           { \bf Case G22 :} $v_5(b)\geq 6,~v_5(a)=4$. In view of (\ref{eq:p1.1}), we have $v_5(D)=29$.  The $x$-Newton polygon of $f(x)$ has two edges . The first edge is the line segment joining the point (0, 0) with  (5, 4) and the second edge is the line segment joining the point (5, 4) with  (6, $v_5(b)$). Arguing as in case E8, we see that $v_5(ind~\theta)=10$ and $\{1, \theta,\frac{\theta^2}{5},\frac{\theta^3}{5^2},\frac{\theta^4}{5^3},\frac{\theta^5}{5^4}\}$ is a 5-integral basis of $K$. 
           
          This completes the proof of Table III.\\
  { \bf Case H1 :}  $v_p(b)=0,v_p(a)\geq 1$ or $v_p(a)=0,v_p(b)\geq 1$. By virtue of (\ref{eq:p1.1}), we have $v_p(D)=0$ and hence $\{1,\theta,\theta^2,\theta^3,\theta^4,\theta^5\}$ is a $p$-integral basis of $K$.\\
  { \bf Case H2 :} $v_p(b)=1,~v_p(a)\geq 1$. In this case, we have $v_p(D)=5$.  The $x$-Newton polygon of $f(x)$ has a single edge joining the points (0, 0) and (6, 1). Arguing as in case E2, we see that $v_p(ind~\theta)=0$ and $\{1,\theta,\theta^2,\theta^3,\theta^4,\theta^5\}$ is a $p$-integral basis of $K$.\\
   { \bf Case H3 :} $v_p(b)\geq 2, ~v_p(a)= 1$. By virtue of (\ref{eq:p1.1}), we have $v_p(D)=6.$ The $x$-Newton polygon of $f(x)$ has two edges . The first edge  is the line segment joining the point (0, 0) with (5, 1) and the second edge is the line segment joining the point (5, 1) with (6, $v_p(b))$. Proceeding as in case E4, we see that $v_p(ind~\theta)=1$ and $\{1, \theta,\theta^2,\theta^3,\theta^4,\frac{\theta^5}{p}\}$ is a $p$-integral basis of $K$.\\
    { \bf Case H4 :} $v_p(b)=2,~v_p(a)\geq 2$. In view of (\ref{eq:p1.1}), we have $v_p(D)=10$. The $x$-Newton polygon of $f(x)$ has a single edge joining the points (0, 0) and (6, 2). Arguing as in case F5, we see that $v_p(ind~\theta)=3$ and $\{1, \theta,\theta^2,\frac{\theta^3}{p},\frac{\theta^4}{p},\frac{\theta^5}{p}\}$ is a $p$-integral basis of $K$.\\
     { \bf Case H5 :} $v_p(b)\geq 3,~v_p(a)= 2$. We have $v_p(D)=12$.  The $x$-Newton polygon of $f(x)$ has two edges . The first edge  is the line segment joining the point (0, 0) with  (5, 2) and the second edge is the line segment joining the point (5, 2) with  (6, $v_p(b)$). Arguing as in case E5, we see that $v_p(ind~\theta)=4$ and $\{1, \theta,\theta^2,\frac{\theta^3}{p},\frac{\theta^4}{p},\frac{\theta^5}{p^2}\}$ is a $p$-integral basis of $K$. \\
      { \bf Case H6 :} $v_p(b)= 3, ~v_p(a)\geq 3$. By virtue of (\ref{eq:p1.1}), we have $v_p(D)=15.$  The $x$-Newton polygon of $f(x)$ has a single edge, say $S$ joining the points (0, 0) and (6, 3) with slope $\frac{1}{2}$. The polynomial associated over the field $\F_p$ to $f(x)$ with respect to ($x,S)$ is $Y^3+\overline{c}$ where  $\overline{c}\neq \overline{0}$. Arguing as in case E6, we see that $v_p(ind~\theta)=6$ and $\{1, \theta,\frac{\theta^2}{p},\frac{\theta^3}{p},\frac{\theta^4}{p^2},\frac{\theta^5}{p^2}\}$ is a $p$-integral basis of $K$. \\
       { \bf Case H7 :} $v_p(b)\geq 4, ~v_p(a)= 3$. We have $v_p(D)=18.$   The $x$-Newton polygon of $f(x)$ has two edges . The first edge is the line segment joining the point (0, 0) with  (5, 3) and the second edge is the line segment joining the point (5, 3) with  (6, $v_p(b)$).  Proceeding as in case E8, we see that $v_p(ind~\theta)=7$ and $\{1, \theta,\frac{\theta^2}{p},\frac{\theta^3}{p},\frac{\theta^4}{p^2},\frac{\theta^5}{p^3}\}$ is a $p$-integral basis of $K$. \\
         { \bf Case H8 :} $v_p(b)= 4,~v_p(a)\geq 4$. In view of (\ref{eq:p1.1}), we have $v_p(D)=20$. The $x$-Newton polygon of $f(x)$ has a single edge, say $S$ joining the points (0, 0) and (6, 4). The polynomial associated  to $f(x)$ with respect to ($x,S)$ is  $Y^2+\overline{c}$ where  $\overline{c}\neq \overline{0}$. As in case F9,  we see that $v_p(ind~\theta)=8$ and $\{1, \theta,\frac{\theta^2}{p},\frac{\theta^3}{p^2},\frac{\theta^4}{p^2},\frac{\theta^5}{p^3}\}$ is a $p$-integral basis of $K$.\\
             { \bf Case H9 :} $v_p(b)\geq 5,~v_p(a)= 4$. In view of (\ref{eq:p1.1}), we have $v_p(D)=24$.  The $x$-Newton polygon of $f(x)$ has two edges . The first edge, say $S_1$  is the line segment joining the point (0, 0) with  (5, 4) and the second edge, say $S_2$ is the line segment joining the point (5, 4) with  (6, $v_p(b)$). The polynomial associated to $f(x)$ with respect to $(x,S_i)$ is linear for $i=1,~2$. As in case E9, we see that $v_p(ind~\theta)=10$ and $\{1, \theta,\frac{\theta^2}{p},\frac{\theta^3}{p^2},\frac{\theta^4}{p^3},\frac{\theta^5}{p^4}\}$ is a $p$-integral basis of $K$. \\
               { \bf Case H10 :} $v_p(b)=5, ~v_p(a)\geq 5$. By virtue of (\ref{eq:p1.1}), we have $v_p(D)=25.$  The $x$-Newton polygon of $f(x)$ has a single edge joining the points (0, 0) and (6, 5). Proceeding as in case E10, we see that $v_p(ind~\theta)=10$ and $\{1, \theta,\frac{\theta^2}{p},\frac{\theta^3}{p^2},\frac{\theta^4}{p^3},\frac{\theta^5}{p^4}\}$ is a $p$-integral basis of $K$. \\
               { \bf Case H11, H12 :} $v_p(ab)=0$. In view of an already  known result $v_p(d_K)=0$ or $1$ according as $v_p(D)$ is even or odd (see \cite{LNV}). Therefore if $m$ denotes the largest integer not exceeding $\frac{v_p(D)}{2}$, then $v_p(ind~\theta)=m$. Claim is that there do not exist any integers $x,~y,~z,~v$ such that $\frac{1}{p}[x+y\theta+z\theta^2+v\theta^3+\theta^4]$ is integral over $\Z_{(p)}$. Suppose to the contrary, there exists such an element $\frac{1}{p}[x+y\theta+z\theta^2+v\theta^3+\theta^4]=\xi_1~(say)$ in $S_{(p)}$. It can be easily seen that $Tr_{K/\mathbb Q}(\theta^i)=0$ for $1\leq i\leq 4$ and $Tr_{K/\mathbb Q}(\theta^5)=-5a$ . Thus $Tr_{K/\mathbb Q}(\xi_1)=\frac{6x}{p}$ which must be in $\Z_{(p)}$ and hence $p\mid x$; consequently $\xi_1-\frac{x}{p}=\frac{1}{p}[y\theta+z\theta^2+v\theta^3+\theta^4]$ will be in  $S_{(p)}$; this implies that $\frac{1}{p}[y+z\theta+v\theta^2+\theta^3]$ is in $S_{(p)}$ because the ideals $pA_K$  and $\theta A_K$ are coprime in view of the hypothesis that $p$ does not divide $N_{K/\mathbb Q}(\theta)=b$. Repeating the above argument with $\frac{1}{p}[y+z\theta+v\theta^2+\theta^3]$, we see that $p\mid y$. Similarly $p\mid z,~p\mid v$. So $\xi_1-\frac{1}{p}[x+y\theta+z\theta^2+v\theta^3]=\frac{\theta^4}{p}\in S_{(p)}$, which is impossible as $\frac{b^4}{p^6}=N_{K/\mathbb Q}(\frac{\theta^4}{p}) $ does not belong to $\Z_{(p)}$. This contradiction proves the claim . It is immediate from the claim and Proposition 3.A that $K$ has a $p$-integral basis of the type $\{1,\theta,\theta^2,\theta^3,\theta^4,\eta\}$ where $\eta=\frac{1}{p^m}[x+y\theta+z\theta^2+v\theta^3+w\theta^4+\theta^5]$ for some $x,~y,~z,~v,~w$ in $\Z$. We now describe the choice of $x,~y,~z,~v,~w$ modulo $p^m$. Since  $Tr_{K/\mathbb Q}(\eta)=\frac{6x-5a}{p^m}\in\Z_{(p)},$ it follows that $6x\equiv 5a~(mod~p^m).$ Similarly using the fact that $Tr_{K/\mathbb Q}(\eta\theta^i)\in\Z_{(p)}$ for $1\leq i\leq 4$, we see that $$5aw+6b\equiv 0~(mod~p^m),~5av+6bw\equiv 0~(mod~p^m),$$
                $$5az+6bv\equiv 0~(mod~p^m),~5ay+6bz\equiv 0~(mod~p^m).$$ On solving the above congruences, we see that 
               $$(5a)^4y\equiv(6b)^4~(mod~p^m),~(5a)^3z\equiv-(6b)^3~(mod~p^m),$$
               $$(5a)^2v\equiv(6b)^2~(mod~p^m),~5aw\equiv -6b~(mod~p^m).$$ 
   {\bf Acknowledgements}\\ The authors are thankful to Dr. Anuj Jakhar for useful  discussions. The first author is grateful to the Council of Scientific and Industrial Research, New Delhi for providing financial support in the form of Junior Research Fellowship through Grant No. 09/135(0878)/2019-EMR-1.  The second author is thankful to Indian National Science Academy for senior scientistship.

\end{document}